\documentclass{amsart}
\usepackage{amsfonts}
\usepackage{amssymb,latexsym}
\usepackage{graphicx,amsmath,amscd}
\usepackage[pdftitle={Paper}, pdfauthor={Ortiz}]{hyperref}
\usepackage{stmaryrd}
\usepackage{hyperref}  
\usepackage{enumerate}
\usepackage{amssymb}
\usepackage{amsthm}
\usepackage{xspace}
\usepackage{latexsym}
\usepackage{array}
\usepackage[T1]{fontenc}
\usepackage{esint}

\usepackage{enumerate}
\usepackage[all,cmtip,2cell]{xy}
\UseTwocells

\usepackage{tikz-cd}
\usepackage{bbold}

\usepackage[all]{xy}

\usepackage{xcolor}

\usepackage{siunitx,amsmath}
 
\newcommand{\itemEq}[1]{%
         \begingroup%
         \setlength{\abovedisplayskip}{0pt}%
         \setlength{\belowdisplayskip}{0pt}%
         \parbox[c]{\linewidth}{\begin{flalign}#1&&\end{flalign}}%
         \endgroup}

\newcommand {\comment}[1]{{\marginpar{*}\scriptsize{\bf *[Comment:}\scriptsize{\ #1}]\bf}}
\newcommand{\diagramh}[1]{{#1}}
\usepackage{tikz-cd}
\usepackage{pb-diagram}
\tikzset{
symbol/.style={
,draw=none
,every to/.append style={
edge node={node [sloped, allow upside down, auto=false]{$#1$}}}
}
}

\newcommand{\pullback}[2]{{}_{#1}\kern-\scriptspace{\times}_{#2}}

\newcommand{\st} [1]     {\scriptscriptstyle{{#1}}}
\newcommand{\rmap}       {\longrightarrow}
\newcommand{\rr}             {\rightrightarrows}
\newcommand{\cc} [1]     {\overline {{#1}}}
\newcommand{\id}         {{\mathrm {Id}}}
\newcommand{\SP} [1]     {{\left\langle {{#1}} \right\rangle}}
\newcommand{\Bi} [1]     {{\left( {{#1}} \right)_{\mathfrak{g}}}}
\newcommand{\pr}             {\mathrm{pr}}
\newcommand{\At}         {{\mathrm {At}}}

\newcommand{\frakg}     {\mathfrak{g}}
\newcommand{\frakh}     {\mathfrak{h}}
\newcommand{\frakd}     {\mathfrak{d}}
\newcommand{\frakt}     {\mathfrak{t}}
\newcommand{\fraku}     {\mathfrak{u}}
\newcommand{\gstar}     {\mathfrak{g}^*}

\newcommand{\G}         {\mathcal{G}}
\newcommand{\sour}        {\mathsf{s}}
\newcommand{\tar}         {{\mathsf{t}}}
\newcommand{\la}           {\mathcal{LA}}
\newcommand{\vb}       {\mathcal{VB}}
\newcommand{\action} {\curvearrowright}

\newcommand{\GG}   {\mathbb{G}}
\newcommand{\p}   {\mathbb{P}}
\renewcommand{\gg}   {\mathbb{g}}

\newcommand{\frakX}     {\mathfrak{X}}
\newcommand{\frakY}     {\mathfrak{Y}}
\newcommand{\frakZ}     {\mathfrak{Z}}
\newcommand{\frakA}     {\mathfrak{A}}
\newcommand{\frakE}     {\mathfrak{E}}

\newcommand{\Cour}[1]      {[\![#1]\!]}

\newcommand{\Lie}        {\mathcal L}

\newcommand\R{\mathbb{R}}
\newcommand\Z{\mathbb{Z}}
\newcommand\N{\mathbb{N}}
\newcommand\Q{\mathbb{Q}}
\newcommand{\C}{\mathbb C}
\newcommand{\M}{\mathcal{M}}

\newcommand{\bitimes}[2]{\,_{#1}\!\!\times_{#2}}

\DeclareMathOperator{\Aut}{Aut} 
\DeclareMathOperator{\AAut}{{\mathbb{Aut}}} 
\DeclareMathOperator{\Hom}{Hom} 
\DeclareMathOperator{\Inn}{Inn}
\DeclareMathOperator{\End}{End} 
\DeclareMathOperator{\Der}{Der}
\DeclareMathOperator{\Vect}{Vect}
\DeclareMathOperator{\Ad}{Ad} 
\DeclareMathOperator{\ad}{ad} 
\DeclareMathOperator{\Ker}{Ker} 
\DeclareMathOperator{\image}{Im} 
\DeclareMathOperator{\Vectd}{\Vect^{\text{def}}}
\DeclareMathOperator{\gl}{\mathfrak{gl}}
\DeclareMathOperator{\oo}{\mathfrak{o}}
\DeclareMathOperator{\so}{\mathfrak{so}}
\DeclareMathOperator{\su}{\mathfrak{su}}
\DeclareMathOperator{\uu}{\mathfrak{u}}
\renewcommand{\aa}{\mathfrak{a}}
\renewcommand{\sl}{\mathfrak{sl}}
\renewcommand{\sp}{\mathfrak{sp}}
\DeclareMathOperator{\GL}{GL}
\DeclareMathOperator{\GGL}{\mathbb{GL}}
\DeclareMathOperator{\SL}{SL}
\DeclareMathOperator{\SO}{SO}
\DeclareMathOperator{\SU}{SU}
\DeclareMathOperator{\OO}{O}
\DeclareMathOperator{\UU}{U}
\DeclareMathOperator{\Sp}{Sp}
\DeclareMathOperator{\tr}{trace}
\DeclareMathOperator{\Diff}{Diff}
\DeclareMathOperator{\Graph}{Graph}
\DeclareMathAlphabet{\mathpzc}{OT1}{pcz}{m}{it}

\theoremstyle{plain}
\newtheorem{theorem}{Theorem}[section]
\newtheorem*{theorem*}{Theorem}

\newtheorem{proposition}[theorem]{Proposition}
\newtheorem{lemma}[theorem]{Lemma}
\newtheorem{corollary}[theorem]{Corollary}
\theoremstyle{definition}
\newtheorem*{definition*}{Definition}

\newtheorem{example}[theorem]{Example}

\newtheorem{definition}[theorem]{Definition}
\newtheorem{remark}[theorem]{Remark}

\begin{document}
%%%%%%%%%%%%%%%%%%%%%%%%%%%%%%%%%%%%%%%%%%%%%%%%%%%%%%%%%%%%%%%%%%%%%%%%%%%
%%%%%%%%%%%%%%%%%%%%%%     %%%%%%%%%%%%%%%%%%%%%%%%%%%%%%%%%%%%%%%%

\author{Juan Sebasti\'an Herrera-Carmona}
%\address{Departamento de Matem\'atica, Universidade de S\~ao Paulo, Rua do Mat\~ao 1010, S\~ao Paulo, SP 05508-090, Brasil} 
\email{sebast1anherr3ra@gmail.com}

\author{Cristian Ortiz}
\address{Departamento de Matem\'atica, Universidade de S\~ao Paulo, Rua do Mat\~ao 1010, S\~ao Paulo, SP 05508-090, Brasil}
\email{cortiz@ime.usp.br}

\thanks{}

\date{\today}
\title{The Chern-Weil-Lecomte characteristic map for $L_{\infty}$-algebras}

\maketitle

\begin{abstract}

In this paper we extend the Chern-Weil-Lecomte characteristic map to the setting of $L_{\infty}$-algebras. In this general framework, characteristic classes of $L_{\infty}$-algebra extensions are defined by means of the Chern-Weil-Lecomte map which takes values in the cohomology of an $L_{\infty}$-algebra with coefficients in a representation up to homotopy. This general set up allows us to recover several known cohomology classes in a unified manner, including: the characteristic class of a Lie 2-algebra, the \v{S}evera class of an exact Courant algebroid and the curvature 3-form of a gerbe with connective structure. We conclude by introducing a Chern-Weil map for principal 2-bundles over Lie groupoids.

\end{abstract}

\tableofcontents

%%%%%%%%%%%%%%%%%%%%%%%%%%%%%%%%%%%%%%%%%%%%%%%%
%%%%%%%%%%%%%%%%%%%%%%%%%%%%%%%%%%%%%%%%%%%%%%%%
%%%%%%%%%%%%%%%%%%%%%%%%%%%%%%%%%%%%%%%%%%%%%%%%

\section{Introduction}

Chern-Weil theory is a geometric setting which permits to construct invariants of principal bundles called characteristic classes. Given a principal $G$-bundle $P\to M$, the Chern-Weil map is an algebra morphism

\begin{equation*}
cw:S^{\bullet}(\mathfrak{g}^*)^{\Ad}\to H_{dR}(M),
\end{equation*}

\noindent which maps an invariant polynomial to a de Rham cohomology class defined in terms of the curvature of a principal connection.

In \cite{Lecomte}, Lecomte introduced characteristic classes associated with any extension of Lie algebras $\hat{\mathfrak{g}}\to \mathfrak{g}$ with kernel $\mathfrak{n}$, together with a Lie algebra representation $\mathfrak{g}\to \mathrm{End}(V)$. The corresponding \textbf{Chern-Weil-Lecomte characteristic map} is defined as 

\begin{equation*}
cw:\mathrm{Sym}^{\bullet}(\mathfrak{n},V)^{\hat{\mathfrak{g}}}\to H^{ev}(\mathfrak{g},V), f\mapsto [f_{K_h}],
\end{equation*}

\noindent where $\mathrm{Sym}^k(\mathfrak{n},V)^{\hat{\mathfrak{g}}}$ is the space of symmetric linear functions $f:\mathfrak{n}^{\otimes k}\to V$ which are $\hat{\mathfrak{g}}$-invariant and $f_{K_h}\in\wedge^{ev}\mathfrak{g}^*\otimes V$ is the Chevalley-Eilenberg cocycle defined by evaluating $f$ at a suitable wedge product of the curvature $K_h\in \wedge^2\mathfrak{g}^*\otimes \mathfrak{n}$ of a splitting $h:\mathfrak{g}\to \hat{\mathfrak{g}}$. Just as in the classical case, the cohomology class $[f_{K_h}]$ is independent of the splitting \cite{Lecomte}. 

The classical Chern-Weil morphism is recovered from this algebraic perspective by looking at the Atiyah extension of a principal bundle $P\to M$ together with the canonical representation of the Lie algebra $\mathfrak{X}(M)$ of vector fields on the space of functions $C^{\infty}(M)$. Several authors have also generalized the Chern-Weil construction to a variety of situations, see for instance \cite{AlekseevMeinrenken, AbadMontoyaQuintero, Crainic, KotovStrobl, LGTX}.

In this paper we are concerned with a generalization of Lecomte's construction to the framework of $L_{\infty}$-algebras. Introduced in \cite{LadaStasheff}, $L_{\infty}$-algebras certain higher structures which generalize Lie algebras in the sense that the Jacobi identity holds up to higher order homotopies. These objects play an important role in deformation theory but also arise as the space of observables in multisymplectic geometry \cite{Rogers}. Other examples of higher structures include: Courant algebroids \cite{LWX}, central to the study of generalized geometry, Lie groupoids and their associated differentiable stacks \cite{behrend-xu} as geometric objects which take care in a smooth way of singularities and quotients, $\mathbb{S}^1$-gerbes \cite{Brylinski} as degree three cohomology analogues of line bundles, n-plectic structures \cite{Rogers} which arise in connection with classical field theories \cite{RyvkinBurzwacher}, principal 2-bundles as the main objects of study in higher gauge theory \cite{BaezSchreiber, waldorf1, Wockel}, among many others. All of the previous examples of higher structures have an associated extension of $L_{\infty}$-algebras, hence it is natural to extend the Chern-Weil-Lecomte construction to this more general framework giving rise to characteristic classes which are invariants of the aforementioned higher structures.

In order to extend the Chern-Weil-Lecomte characteristic map to the setting of $L_{\infty}$-algebras, one needs to study the cohomology of an $L_{\infty}$-algebra with values in a representation up to homotopy \cite{Penkava, Reinhold}. Representations up to homotopy of an $L_{\infty}$-algebra $\mathfrak{g}$ are also called $L_{\infty}$-modules in \cite{LadaMarkl}, these are given by a graded vector space $\mathbb{V}$ together with a suitable notion of action of $\mathfrak{g}$. It turns out that a representation up to homotopy has an associated cochain complex $(C(\mathfrak{g},\mathbb{V}),D)$ which generalizes the Chevalley-Eilenberg complex defined by an ordinary Lie algebra representation. The corresponding cohomology $H(\mathfrak{g},\mathbb{V})$ is called the \textbf{$L_{\infty}$-cohomology} of $\mathfrak{g}$ with values in $\mathbb{V}$.  

It is natural to investigate the invariance of $H(\mathfrak{g},\mathbb{V})$ with respect to a suitable notion of morphism of representations up to homotopy. For that, we study functorial properties of $C(\mathfrak{g},\mathbb{V})$. It is well known that an $L_{\infty}$-algebra structure on $\mathfrak{g}$ is equivalent to a degree one coderivation $d_{\mathfrak{g}}:S(\mathfrak{g}[1])\to S(\mathfrak{g}[1])$ with $d^2_{\mathfrak{g}}=0$. Here $S(\mathfrak{g}[1])$ is the symmetric coalgebra of the graded vector space $\mathfrak{g}[1]$. An \textbf{$L_{\infty}$-morphism} $F:(S(\mathfrak{g}[1]),d_{\mathfrak{g}})\to (S(\mathfrak{h}[1]),d_{\mathfrak{h}})$ is a coalgebra morphism which commutes with the coderivations. An $L_{\infty}$-morphism is completely determined by its projection onto $\mathfrak{h}[1]$, hence by a family of maps $ F^1_n:S^n(\mathfrak{g}
[1])\to \mathfrak{h}[1]; \quad n\geq 1.$ We say that $F$ is \textbf{strict} if $F^1_n=0$ for every $n\geq 2$. Every $L_{\infty}$-algebra $\mathfrak{g}$ has an underlying cochain complex of vector spaces defined by the first order bracket $[\cdot]:\mathfrak{g}\to \mathfrak{g}$ given by the $L_{\infty}$-structure of $\mathfrak{g}$. If $F:(S(\mathfrak{g}[1]),d_{\mathfrak{g}})\to (S(\mathfrak{h}[1]),d_{\mathfrak{h}})$ is an $L_{\infty}$-morphism, then its \textbf{linear part} $F^1_1:\mathfrak{g}[1]\to \mathfrak{h}[1]$ is a cochain map between the underlying complexes.  We say that $F$ is an \textbf{$L_{\infty}$-quasi-isomorphism} if its linear part is a quasi-isomorphism between the underlying complexes.

One can see that a graded vector space $\mathbb{V}$ which is a representation up to homotopy of $\mathfrak{g}$ inherits a canonical structure of cochain complex. Let $\mathbb{V}$ and $\mathbb{W}$ be  representations up to homotopy of $\mathfrak{g}$ and $\mathfrak{h}$, respectively. A \textbf{morphism of representations up to homotopy} $\mathbb{V}\to \mathbb{W}$ is given by a pair $(F,f)$ where $F:\mathfrak{g}\to \mathfrak{h}$ is an $L_{\infty}$-algebra map and $f:\mathbb{W}\to \mathbb{V}$ is a degree preserving map satisfying a suitable equivariance condition, see Definition \ref{def:ruthmorphism}. If both $F$ and $f$ are quasi-isomorphisms, we say that $(F,f)$ is a \textbf{quasi-isomorphism} of representations up to homotopy.

It turns out that every morphism of representations up to homotopy induces a pullback $F^{*}:C(\mathfrak{h},\mathbb{W}) \to C(\mathfrak{g},\mathbb{V})$ which is a cochain map, see Proposition \ref{equivariance}. Our first main result says that the invariance of the $L_{\infty}$-algebra cohomology is controlled by the linear part of a morphism of representations up to homotopy.

\begin{theorem*}

Let $\mathbb{V}$ and $\mathbb{W}$ be representations up to homotopy of $\mathfrak{g}$ and $\mathfrak{h}$, respectively. If  $(F,f):\mathbb{V}\to \mathbb{W}$ is a quasi-isomorphism of representations up to homotopy, then

  \begin{center}
    \begin{tikzcd}
       F^*:C(\mathfrak{h},\mathbb{W}) \arrow[r]&C(\mathfrak{g},\mathbb{V})
    \end{tikzcd}
\end{center}
is a quasi-isomorphism. Hence $H(\mathfrak{g},\mathbb{V})\cong H(\mathfrak{h},\mathbb{W})$.

\end{theorem*}

We proceed now to state the $L_{\infty}$-version of the Chern-Weil-Lecomte map. For that, recall that an extension of $L_{\infty}$-algebras is a short exact sequence of strict $L_{\infty}$-morphisms

$$0\to \mathfrak{n}\to \hat{\mathfrak{g}}\to \mathfrak{g}\to 0.$$

\noindent Suppose that $\mathbb{V}$ is a representation up to homotopy of $\mathfrak{g}$. We also see $\mathbb{V}$ as a representation up to homotopy of $\hat{\mathfrak{g}}$ defined by the pullback representation up to homotopy via $\hat{\mathfrak{g}}\to \mathfrak{g}$. One observes that $\mathfrak{n}[1]$ inherits a canonical representation up to homotopy of $\hat{\mathfrak{g}}$, which extends naturally to any exterior power $\wedge^k\mathfrak{n}[1]$. Hence one considers the space $\underline{\mathrm{Hom}}^{\bullet}(\wedge^k\mathfrak{n}[1],\mathbb{V})^{\hat{\mathfrak{g}}}$ of all morphisms of representations up to homotopy. The following is our second main result.

\begin{theorem*}
 Let $K_h\in C^1(\mathfrak{g},\mathfrak{n}[1])$ be the curvature of a section $h:\mathfrak{g}\to \hat{\mathfrak{g}}$ of an extension of $L_{\infty}$-algebras

  \begin{center}
 \begin{tikzcd}
    0\arrow[r]&\mathfrak{n}\arrow[r,"\iota"]&\hat{\mathfrak{g}}\arrow[r,"\pi"]&\mathfrak{g}\arrow[r]&0,
 \end{tikzcd}
\end{center}

\noindent and $\mathbb{V}$ a representation up to homotopy of $\mathfrak{g}$. If $f\in \underline{\Hom}^{\bullet}(\wedge^k\mathfrak{n}[1],\mathbb{V})^{\hat{\mathfrak{g}}}$, then

\begin{enumerate}

\item the element $f_{K_h}=f\circ K^{\wedge k}_h$ is a cocycle in the complex $(C(\mathfrak{g},\mathbb{V}),D_{\rho})$,

\item the cohomology class $[f_{K_h}]\in H^{\bullet+k}(\mathfrak{g},\mathbb{V})$ is independent of the section $h$.

\end{enumerate}

\end{theorem*}

Here $K^{\wedge k}_h\in C^k(\mathfrak{g},\mathbb{V})$ is a suitable exterior power defined in terms of cochain products, see subsection \ref{subsec:cochainproducts}. As a consequence, for every $k\in \N$ there is a well defined map

$$cw:\underline{\Hom}^{\bullet}(\wedge^k\mathfrak{n}[1],\mathbb{V})^{\hat{\mathfrak{g}}} \to H^{\bullet+k}(\mathfrak{g},\mathbb{V}); f\mapsto [f_{K_h}]$$

\noindent called the \textbf{Chern-Weil-Lecomte map} associated with an $L_{\infty}$-algebra extension.

The cohomology classes $cw(f)=[f_{K_h}]\in H^{\bullet}(\mathfrak{g},\mathbb{V})$ are invariants of an extension of the $L_{\infty}$-algebra $\mathfrak{g}$ together with the representation up to homotopy $\mathbb{V}$. This is the content of Theorem \ref{thm:naturality}.

An important class of $L_{\infty}$-algebras is given by \textbf{2-term} $L_{\infty}$-algebras, that is, those whose underlying complex has length 2. As shown in \cite{BaezCrans}, 2-term $L_{\infty}$-algebras correspond to Lie 2-algebras, which are categorified  Lie algebras. There are two different types of Lie 2-algebras called \textbf{strict} and \textbf{semi-strict}. The former are Lie 2-algebras in which the usual Jacobi identity is satisfied and they are equivalent to crossed modules of Lie algebras. Among the main examples of strict Lie 2-algebras, one encounters: multiplicative vector fields on a Lie groupoid \cite{BEL, OrtizWaldron}, the sections of an LA-groupoid \cite{OrtizWaldron}, infinitesimal symmetries of an $\mathbb{S}^1$-gerbe \cite{Collier1, Collier2, KrepskiVaughan}, among others. The latter are Lie 2-algebras in which the Jacobi identity holds up to a natural isomorphism and important examples include: sections of a Courant algebroid \cite{RoytenbergWeinstein}, the Atiyah Lie 2-algebra associated with a closed 3-form \cite{FiorenzaRogersSchreiber}, the Poisson Lie 2-algebra of 2-plectic structures \cite{Rogers} and multiplicative forms on quasi-Poisson groupoids \cite{ChenLangLiu}.

We apply the Chern-Weil-Lecomte construction of characteristic classes of $L_{\infty}$-algebras to recover well known cohomology classes associated with a variety of structures, including: the characteristic class of any Lie 2-algebra, the \v{S}evera class of an exact Courant algebroid and the 3-curvature form of an $\mathbb{S}^1$-gerbe with connective structure.

Besides the aforementioned examples of characteristic classes defined via the $L_{\infty}$-version of the Chern-Weil-Lecomte map, we also apply our construction to introduce a Chern-Weil morphism for principal 2-bundles over Lie groupoids. Principal 2-bundles equipped with a 2-connection are the main objects of study in higher gauge theory \cite{BaezSchreiber, waldorf1}, which is a geometric setting in which parallel transport is defined not only along paths, but also along surfaces. The notion of principal 2-bundle considered in this paper can be thought of as a strict version of a stacky principal bundle as in \cite{BNZ}. Several notions of principal 2-bundles have been studied by different authors, including \cite{Wockel,waldorf1, CCP}. Also, in \cite{LGTX} it is introduced a Chern-Weil map for principal $G$-bundles over Lie groupoids, which are particular cases of principal 2-bundles in which the structural 2-group is the unit groupoid defined by a Lie group.

Given a Lie 2-group $\GG$, a \textbf{principal 2-bundle} is given by a right Lie 2-group action $\p\curvearrowleft \GG$ together with a Lie groupoid fibration $\pi:\p\to \mathbb{X}$ which is $\GG$-invariant and principal in the sense that  $\p\times \GG\to \p\times_{\mathbb{X}} \p$ is a Lie groupoid isomorphism. A \textbf{2-connection} on $\p=(P_1\rr P_0)$ is defined as a $\mathbb{g}$-valued multiplicative 1-form whose components are ordinary connections on $P_
1$ and $P_
0$, respectively. Here $\mathbb{g}$ denotes the strict Lie 2-algebra of the Lie 2-group $\GG$. See Definition \ref{def:2connection}. Let $\mathfrak{h}\to \mathfrak{g}_0$ be the crossed module associated with $\mathbb{g}=\mathfrak{g}_1\rr \mathfrak{g}_0$. Then, every 2-connection can be written as $\theta=\bar\theta+s^*\theta_0$ where $\bar\theta\in \Omega^1(P_1,\mathfrak{h})$ and $\theta_0\in\Omega^1(P_0,\mathfrak{g}_0)$ satisfying certain conditions.

The \textbf{curvature} of $\theta$ is defined as the multiplicative 2-form $\Omega=d\theta + \frac{1}{2}[\theta,\theta]\in \Omega^2(\p,\mathbb{g})$. Due to the multiplicativity of the curvature 2-form, one has $\Omega=\bar\Omega+s^*\Omega_0$ where $\bar\Omega\in \Omega^2(P_1,\mathfrak{h})$ and $\Omega_0\in\Omega^2(P_0,\mathfrak{g}_0)$ is the curvature of the connection $\theta_0$. We introduce the following terminology.

\begin{definition*}
Let $\theta=\bar\theta+s^*\theta_0$ be a 2-connection on a principal 2-bundle $\p\curvearrowleft\GG\to \mathbb{X}$ with curvature 2-form $\Omega=\bar\Omega + s^*\Omega_0$.  A \textbf{curving} for $\theta$ is a basic 2-form $B\in\Omega^2(P_0,\mathfrak{h})$ with:

\begin{enumerate}

\item $\Omega_0+\partial\circ B=0$, and
\item $\bar\Omega+(t^*-s^*)(B)=0$.
 
\end{enumerate}

\end{definition*}

A \textbf{connective structure} on a principal 2-bundle $\p\curvearrowleft\GG\to \mathbb{X}$ is a pair $(\theta,B)$ where $\theta$ is a 2-connection on $\p$ and $B$ is a curving for $\theta$. The \textbf{curvature 3-form} of a connective structure $(\theta,B)$ is the 3-form $\nabla^{\theta_0}B\in \Omega^3(P_0,\mathfrak{h})$ defined as the covariant derivative of $B$ induced by $\theta_0$. In the special case of principal 2-bundles with $\GG=(\mathbb{S}^1\rr \{*\})$, one recovers the usual notions of connective structure and curvature 3-form on $\mathbb{S}^1$-gerbes.

Just as ordinary connections, a 2-connection on $\p\curvearrowleft \GG \to \mathbb{X}$ is equivalent to a multiplicative splitting of the Atiyah sequence

$$0\to \Ad(\mathbb{P})\to\At(\mathbb{P})\to T\mathbb{X}\to0.$$

\noindent This is a short exact sequence of LA-groupoids over $\mathbb{X}$.  As shown in \cite{OrtizWaldron}, every LA-groupoid $\mathbb{E}\to \mathbb{X}$ has an associated strict Lie 2-algebra $\mathbb{L}(\mathbb{E})$ of multiplicative sections. In particular, there is extension of Lie 2-algebras of multiplicative sections

$$0\to\mathbb{L}(\Ad(\mathbb{P}))\to\mathbb{L}(\At(\mathbb{P}))\to\mathbb{L}(T\mathbb{X})\to 0.$$

\noindent called the \textbf{Atiyah extension} of $\p$. A 2-connection induces a splitting $h:\mathbb{L}(T\mathbb{X})\to \mathbb{L}(\At(\mathbb{\p}))$ of the Atiyah extension which is not necessarily an $L_{\infty}$-morphism. The curvature 3-form of a connective structure on a principal 2-bundle has the following geometric interpretation.

\begin{theorem*}

Let $(\theta, B)$ be a connective structure on a principal 2-bundle $\p\curvearrowleft \GG\to \mathbb{X}$. If the curvature 3-form vanishes, then the induced splitting $h:\mathbb{L}(T\mathbb{X})\to \mathbb{L}(\At(\mathbb{\p}))$ defined by the 2-connection $\theta$ is a weak morphism of 2-term $L_{\infty}$-algebras.
\end{theorem*}

This result generalizes a theorem by Krepski and Vaughan \cite{KrepskiVaughan} in the setting of $\mathbb{S}^1$-gerbes with connective structure.

We explain now our last main result. Let $\p\curvearrowleft \GG\to \mathbb{X}$ be a principal 2-bundle. The 2-vector space of multiplicative functions on $\mathbb{F(X)}$ is the one associated with the 2-term complex $C^{\infty}(X_0)\to C^{\infty}_{mult}(\mathbb{X})$ given by the truncation of the usual groupoid complex. The Lie derivative extends to a natural representation up to homotopy of the Lie 2-algebra of multiplicative vector fields $\mathbb{L}(T\mathbb{X})$ on $\mathbb{F(X)}$ \cite{HOW}. Combining this representation up to homotopy with the splitting $h:\mathbb{L}(T\mathbb{X})\to \mathbb{L}(\At(\mathbb{\p}))$ defined by a 2-connection on $\p$, one is led to the following version of the Chern-Weil map for principal 2-bundles.

\begin{theorem*}
Let $\p\curvearrowleft\GG\to \mathbb{X}$ be a principal 2-bundle equipped with a 2-connection. Let $h:\mathbb{L}(T\mathbb{X})\to \mathbb{L}(\At(\p))$ the induced splitting of the Atiyah extension of Lie 2-algebras. For every $k\geq 0$ there exists a natural map

\begin{align}\label{eq:CWL2bundles}
cw:\underline{\Hom}^{\bullet}(\wedge^k\mathbb{L}(\Ad(\p))&[1],\mathbb{F(X)})^{\mathbb{L}(\mathrm{At}(\mathbb{P}))}\to H^{\bullet+k}(\mathbb{L}(T\mathbb{X}),\mathbb{F(X)})\\ \nonumber
&f\mapsto [f\circ K^k_h],
\end{align}

\noindent where $K_h\in C^1(\mathbb{L}(T\mathbb{X}),\mathbb{L}(\Ad(\p))[1])$ is the curvature of $h$. The map $cw$ is independent of the 2-connection.
\end{theorem*}

It is natural to investigate the relation between the cohomology $H(\mathbb{L}(T\mathbb{X}),\mathbb{F(X)})$ of the Lie 2-algebra of multiplicative vector fields $\mathbb{L}(T\mathbb{X})$ with values in $\mathbb{F}(\mathbb{X})$ and the de Rham cohomology $H_{dR}(\mathbb{X})$, which is defined as the total cohomology of the Bott-Shulman-Stasheff complex of the base Lie groupoid $\mathbb{X}$. We plan to investigate this question in a future work.

This work is organized as follows. In Section 2 we present the basics on $L_{\infty}$-algebras and their morphisms. In Section 3, we study the cohomology of an $L_{\infty}$-algebra with coefficients in a representation up to homotopy. The main result of this section is Theorem \ref{CEquasiiso} which says that the $L_{\infty}$-algebra cohomology with coefficients in a representation up to homotopy is invariant by quasi-isomorphisms of representations up to homotopy. In Section 4 we study extensions of $L_{\infty}$-algebras and we prove Theorem \ref{thm:lecomte} which is the $L_{\infty}$-version of the Chern-Weil-Lecomte map and the main result of this section. We also discuss several examples of cohomology classes which can be described by means of the Chern-Weil-Lecomte map. In Section 5 we specialize our results to the setting of principal 2-bundles. We give a detailed presentation of principal 2-bundles, 2-connections and curvature. We show Theorem \ref{thm:flatconnectivestructure} which give an interpretation of the curvature 3-form of a connective structure on a principal 2-bundle as the failure of a splitting of the Atiyah extension to be a weak $L_{\infty}$-morphism. Then we state Theorem \ref{thm:CWL2bundles} which introduces a Chern-Weil map for principal 2-bundles over Lie groupoids. \\

\textbf{Acknowledgments:} We would like to thank Olivier Brahic, Eduardo Hoefel, Maria Amelia Salazar, Gabriel Sevestre, James Waldron and Marco Zambon for several comments which have improved this work. This work is based on the first author's doctoral thesis \cite{Herrera-Carmona} carried out at the University of Sao Paulo. Herrera-Carmona thanks the Institute of Mathematics and Statistics of the University of Sao Paulo and CAPES-Brazil for financial support. Ortiz was partially supported by National Council of Research and Development CNPq-Brazil, Bolsa de Produtividade em Pesquisa Grant 315502/2020-7 and by Grant 2019/13204-0 Sao Paulo Research Foundation - FAPESP and CONICYT - Chile.

%%%%%%%%%%%%%%%%%%%%%%%%%%%%%%%%%%%%%%%%%%%%%%%%%%
%%%%%%%%%%%%%%%%%%%%%%%%%%%%%%%%%%%%%%%%%%%%%%%%%%
%%%%%%%%%%%%%%%%%%%%%%%%%%%%%%%%%%%%%%%%%%%%%%%%%%
%%%%%%%%%%%%%%%%%%%%%%%%%%%%%%%%%%%%%%%%%%%%%%%%%%

\section{Background material}

In this section we briefly recall the definition of $L_{\infty}$-algebras and their morphisms. We follow \cite{Reinhold} closely. 

%%%%%%%%%%%%%%%%%%%%%%%
%%%%%%%%%%%%%%%%%%%%%%%%
\subsection{Graded vector spaces}

Let $\mathbb{V}=\bigoplus_{n\in \mathbb{Z}}V_n$ be a graded vector space. An element $v\in V_n$ is called homogeneous of degree $n$ and its degree is denoted by $|v|$. Given a second graded vector space $\mathbb{W}$, the space $\underline{\Hom}(\mathbb{V},\mathbb{W})$ is the graded vector space whose homogeneous component of degree $p$ is defined by

$$\underline{\Hom}^p(\mathbb{V},\mathbb{W}):=\bigoplus_{n\in \Z}\mathrm{Hom}(V_n,W_{n+p}).$$
A linear map $f\in \underline{\Hom}^p(\mathbb{V},\mathbb{W})$ is called homogeneous of degree $|f|=p$. A linear map $f:\mathbb{V}\to \mathbb{W}$ is called degree preserving if it is homogeneous of degree zero. If $\mathbb{W}=\mathbb{V}$ we write $\underline{\End}(\mathbb{V}):=\underline{\Hom}(\mathbb{V},\mathbb{V})$.

The tensor product $\mathbb{V}\otimes\mathbb{W}$ is the graded vector space with grading given by

$$(\mathbb{V}\otimes\mathbb{W})_n:=\bigoplus_{i+j=n}V_i\otimes W_j,$$
for all $n\in \mathbb{Z}$. In particular, for every $n\geq 0$ one has a graded vector space $\mathbb{V}^{\otimes n}:= \bigotimes_{i=1}^{n} \mathbb{V}$.

One can also define the tensor product of two homogeneous linear maps $f\in \underline{\Hom}^p(\mathbb{V},\mathbb{W})$ and $g\in \underline{\Hom}^q(\mathbb{V}' ,\mathbb{W}')$ as the homogeneous linear map $f\otimes g\in \underline{\Hom}^{p+q}(\mathbb{V}\otimes\mathbb{V}',\mathbb{W}\otimes\mathbb{W}')$ defined by

\begin{equation}\label{eq:tensormaps}
(f\otimes g)(v\otimes v'):=(-1)^{|g||v|}f(v)\otimes g(v'),
\end{equation}

\noindent for any homogeneous elements $v \in \mathbb{V}$ and $v'\in \mathbb{V}'$. The tensor product of composable linear maps is composable and one has the following interchange law

\begin{equation}\label{eq:interchangelaw}
(f_1\otimes f_2)\circ (g_1\otimes g_2)=(-1)^{|f_2||g_1|}(f_1\circ g_1)\otimes (f_{2}\circ g_2),
\end{equation}

\noindent where $f_1,g_1$ and  $f_2,g_2$ are couples of composable homogeneous linear maps.

The graded vector space $\mathbb{V}^{\otimes n}$ carries a natural $\Sigma_n$-action by degree preserving maps $\hat{\epsilon}:\Sigma_n\to \underline{\End}(\mathbb{V}^{\otimes n})$, where $\Sigma_n$ denotes the symmetric group in $n$-letters. For every transposition $\sigma_i=(i, i+1), 1\leq i \leq n$ the action is defined by

$$\hat{\epsilon}(\sigma_i)(v_1\otimes\dots\otimes v_n)=(-1)^{|v_i||v_{i+1}|}v_1\otimes\cdots \otimes v_{i+1}\otimes v_i\otimes \cdots\otimes v_{n}$$

\noindent where $v_i\in \mathbb{V}$ are homogeneous elements with $1\leq i \leq n$. The action of an arbitrary permutation $\sigma\in \Sigma_n$ is defined as

\begin{equation}\label{eq:koszulsign}
\hat{\epsilon}(\sigma)(v_1\otimes \cdots \otimes v_n):=\epsilon(\sigma)v_{\sigma(1)}\otimes \cdots v_{\sigma(n)}
\end{equation}
where $\epsilon(\sigma)$ is the \textbf{Koszul sign}. Although the Koszul sign also depends on $v_1,...,v_n$ we write $\epsilon(\sigma)$ to make our notation simpler. This action naturally extends to the tensor algebra $T(\mathbb{V}):=\oplus_{n\geq 0}\mathbb{V}^{\otimes n}$, which is equipped with the graded product given by concatenation. For every $n\geq 0$ there is a graded subspace $I_{\epsilon}\subset \mathbb{V}^{\otimes n}$ defined by

$$I_{\epsilon}:=\mathrm{span}\{ v_1\otimes\cdots \otimes v_n - \epsilon(\sigma)v_{\sigma(1)}\otimes \cdots \otimes v_{\sigma(n)}; \sigma\in \Sigma_n \}.$$

The \textbf{$n$-th symmetric power} of $\mathbb{V}$ is defined by $S^{n}(\mathbb{V}):=\mathbb{V}^{\otimes n}/I_{\epsilon}$. The \textbf{symmetric algebra} of $\mathbb{V}$ is the free unital graded commutative algebra given by

$$S(\mathbb{V})=\bigoplus_{n\geq 0}S^{n}(\mathbb{V}):=\bigoplus_{n\geq 0} \mathbb{V}^{\otimes n}/I_{\epsilon}.$$

The canonical projection $\pi_S:T(\mathbb{V})\to S(\mathbb{V})$ is a surjective linear degree preserving map, whose restriction to the fixed points $\mathrm{Sym}^n(\mathbb{V})\subset \mathbb{V}^{\otimes n}$ of the action $\hat{\epsilon}:\Sigma_n\to \underline{\mathrm{End}}(\mathbb{V}^{\otimes n})$ yields an isomorphism $\mathrm{Sym}^n(\mathbb{V})\cong S^n(\mathbb{V})$. For $v_1\otimes\cdots \otimes v_n\in \mathbb{V}^{\otimes n}$ we write $v_1\vee \cdots \vee v_n:=\pi_S(v_1\otimes\cdots \otimes v_n)$ and the graded product in the symmetric algebra is denoted $\mu_S:S(\mathbb{V})\otimes S(\mathbb{V})\to S(\mathbb{V})$. For a homogeneous element $v=v_1\vee\cdots \vee v_n$ in $S^n(\mathbb{V})$ its  \textbf{weight} is the integer $n$ and its \textbf{degree} is $|v|=|v_1|+\cdots+|v_n|$.

The \textbf{exterior algebra} of $\mathbb{V}$ is defined in a similar manner by 

$$\wedge \mathbb{V}=\bigoplus_{n\geq 0}\wedge^n(\mathbb{V}),$$
where $\wedge^n(\mathbb{V}):=\mathbb{V}^{\otimes n}/I_{\chi}$ for every $n\geq 0$. Here $I_{\chi}\subset \mathbb{V}^{\otimes n}$ is defined similarly by considering the skew-symmetric $\Sigma_n$-action on $\mathbb{V}^{\otimes n}$ defined on transpositions by

$$\hat{\chi}(\sigma_i)(v_1\otimes\dots\otimes v_n)=-(-1)^{|v_i||v_{i+1}|}v_1\otimes\cdots \otimes v_{i+1}\otimes v_i\otimes \cdots\otimes v_{n}.$$

The tensor algebra $T(\mathbb{V})$ carries the structure of a coassociative, cocommutative coalgebra with coproduct given by deconcatenation. It turns out that the symmetric algebra $S(\mathbb{V})$ inherits a coalgebra structure which we recall now. A permutation $\sigma \in \Sigma_{n}$ is called a $(k,n-k)$-\textbf{unshuffle} with $0\leq k\leq n$, if $\sigma(1)<\cdots<\sigma(k)$ and $\sigma(k+1)<\cdots<\sigma(n)$. The set of $(k,n-k)$-unshuffles is denoted by $Sh(k,n-k)$. The coproduct $\Delta_S:S(\mathbb{V})\to S(\mathbb{V})\otimes S(\mathbb{V})$ is defined by the symmetric deconcatenation

\begin{equation}\label{eq:symmetriccoproduct}
\Delta_S(v_1\vee\cdots\vee v_n)=\sum_{k=0}^{n}\sum_{\sigma \in Sh(k,n-k)}\epsilon(\sigma)(v_{\sigma(1)}\vee\cdots\vee v_{\sigma(k)})\otimes (v_{\sigma(k+1)}\vee\cdots\vee v_{\sigma(n)}).
\end{equation}

This makes $S(\mathbb{V})$ a coassociative and cocommutative coalgebra, which is also coaugmented with $S(\mathbb{V})=\mathbb{R}1\oplus\bar{S}(\mathbb{V})$ where $\bar{S}(\mathbb{V}):=\oplus_{n\geq 1}S^n(\mathbb{V})$.  The induced coproduct $\bar{\Delta}_S:\bar{S}(\mathbb{V})\to \bar{S}(\mathbb{V})\otimes \bar{S}(\mathbb{V})$ is defined as $\Delta_S=1\otimes \id + \bar{\Delta}_S+ \id\otimes 1$ and it is referred to as the \textbf{reduced coproduct}.

%It follows from the definition of the coalgebra structure on $S(\mathbb{V})$ that

%\begin{equation}\label{eq:projectionidentity}
 %   \mathrm{pr}_{S^{n+1}(\mathbb{V})}=\pi_{n+1}\circ\mathrm{pr}_{\mathbb{V}}^{\otimes(n+1)}\circ \Delta^n_S; \text{ for every } n\in \N.
%\end{equation}

%Here the map $\Delta^n_S:S(\mathbb{V})\to S(\mathbb{V})^{\otimes (n+1)}$ is defined inductively by $\Delta^0_S=\id_{S(\mathbb{V})}$ and $\Delta^n_S:=(\Delta_S\otimes \id^{\otimes n})\circ \Delta^{n-1}_S$ for $n\geq 1$. The grading on $S(\mathbb{V})$ determines canonical projections $\mathrm{pr}_{S^{n}(\mathbb{V})}:S(\mathbb{V})\to S^n(\mathbb{V})$ for $n>1$ and $\mathrm{pr}_{\mathbb{V}}:S(\mathbb{V})\to \mathbb{V}$. Hence the map $\mathrm{pr}_{\mathbb{V}}^{\otimes(n+1)}:S(\mathbb{V})^{\otimes (n+1)}\to \mathbb{V}^{\otimes (n+1)}$ denotes the natural extension to the tensor product. Finally, the map $\pi_{n+1}:\mathbb{V}^{\otimes (n+1)}\to S^{n+1}(\mathbb{V})$ appearing in \eqref{eq:projectionidentity} is defined as follows: 

%\begin{align*}
%\pi_{n+1}(v_1\otimes \cdots \otimes v_{n+1})= \frac{1}{(n+1)!}v_1\vee \cdots \vee v_{n+1}.
%\end{align*}

%More details can be found in \cite[page 8]{Reinhold}.

%%%%%%%%%%%%%%%%%%%%%%%%%%%%%%%%%%%%
%%%%%%%%%%%%%%%%%%%%%%%%%%%%%%%%%%%%

\subsection{$L_{\infty}$-algebras}\label{subsec:Linfty}

Let $\mathfrak{g}=\bigoplus_{n\leq 0}\mathfrak{g}_n$ be a non-positively graded vector space. The structure of an $L_{\infty}$-algebra on $\mathfrak{g}$ can be seen in three equivalent ways, either as higher Lie brackets, higher symmetric functions or coderivations of the coalgebra $S(\mathbb{V})$, respectively. Each of these perspectives will be explained below. For more details, the reader can see \cite{LadaStasheff,LadaMarkl,Manetti} and references therein.

An \textbf{$L_{\infty}$-structure} on $\mathfrak{g}=\bigoplus_{n\leq 0}\mathfrak{g}_n$ is given by a family of linear maps $[\cdot]^k:\wedge^k\mathfrak{g}\to \mathfrak{g}; k\geq 1$ of degree $2-k$ called \textbf{higher brackets}, with vanishing graded Jacobiator:

\begin{equation}\label{eq:gradedJacobi}
\sum_{i+j=n+1}\sum_{\sigma \in Sh(i,n-i)}(-1)^{i(j-1)}\chi(\sigma)\left[\left[v_{\sigma(1)}\wedge\cdots\wedge v_{\sigma(i)}\right]^i\wedge v_{\sigma(i+1)}\wedge\cdots\wedge v_{\sigma(n)}\right]^{j}=0,
\end{equation}

\noindent for every $n\geq 1$.

\begin{definition}
An \textbf{$L_{\infty}$-algebra} is a pair $(\mathfrak{g},[\cdot]^{\bullet})$ given by a non-positively graded vector space equipped with an $L_{\infty}$-structure. 

\end{definition}

Let us see some examples of $L_{\infty}$-algebras.

\begin{example}\label{ex:complex}

Let $(\mathfrak{g},[\cdot]^{\bullet})$ be an $L_{\infty}$-algebra. The degree one map $\partial:=[\cdot]^1:\mathfrak{g}\to \mathfrak{g}$ satisfies $\partial^2=0$. In particular, an $L_{\infty}$-algebra in which $[\cdot]^k=0$ for every $k\geq 2$ is the same that a cochain complex of vector spaces.
\end{example}

\begin{example}\label{ex:dglas}

Let $\mathfrak{g}=\bigoplus_{n\leq 0}\mathfrak{g}_n$ be an $L_{\infty}$-algebra in which $[\cdot]^k=0$ for every $k\geq 3$. This is equivalent to having a structure of a differential graded Lie algebra (dgla for short)  on the graded vector space $\mathfrak{g}$ with bracket $[\cdot,\cdot]:=[\cdot]^2:\mathfrak{g}\otimes \mathfrak{g}\to \mathfrak{g}$ and differential $d:=[\cdot]^1:\mathfrak{g}\to \mathfrak{g}$.

\end{example}

\begin{example}\label{ex:crossedmodule}

An $L_{\infty}$-algebra $\mathfrak{g}=\mathfrak{g}_{-1}\oplus \mathfrak{g}_0$ concentrated in degrees -1 and 0 is called 2-term $L_{\infty}$-algebra. In this case, the only not necessarily trivial brackets are $[\cdot]^i:\wedge^i\mathfrak{g}\to \mathfrak{g}$ with $i=1,2,3$. As shown in \cite{BaezCrans}, the category of 2-term $L_{\infty}$-algebras is equivalent to that of semistrict Lie 2-algebras. If $[\cdot]^3=0$, one recovers the notion of a strict Lie 2-algebra, which is also equivalent to having a crossed module of Lie algebras. That is, a cochain complex of Lie algebras $\partial:\mathfrak{g}_{-1}\to \mathfrak{g}_0$ together with an action $D:\mathfrak{g}_{0}\to \mathrm{Der}(\mathfrak{g}_{-1})$ satisfying the Peiffer identities:

\begin{enumerate}
\item $D_{\partial(u)}v=\ad_uv$
\item $\partial(D_xu)=\ad_x\partial(u)$
\end{enumerate}

\noindent for every $u,v\in \mathfrak{g}_{-1}$ and $x\in \mathfrak{g}_{0}$.

\end{example}

The \emph{dec\'alage} isomorphism $S^n(\mathfrak{g}[1])\cong (\wedge^n\mathfrak{g})[n], n\geq 1$ (see \cite{Reinhold}), yields an alternative description of an $L_{\infty}$-algebra structure on a graded vector space $\mathfrak{g}=\bigoplus_{n\leq 0}\mathfrak{g}_n$. Namely, as a family of linear maps $\lambda_k:S^{k}(\mathfrak{g}[1])\to \mathfrak{g}[1], k\geq 1$ of degree 1 satisfying

\begin{equation}\label{eq:symmetricJacobi}
\sum_{i=1}^n\sum_{\sigma\in Sh(i,n-i)}\epsilon(\sigma)\lambda_{n-i+1}(\lambda_{i}(v_{\sigma(1)}\vee\cdots\vee v_{\sigma(i)})\vee v_{\sigma(i+1)}\vee\cdots \vee v_{\sigma(n)})=0,
\end{equation}
for every $n\geq 1$.

The family $\lambda_k:S^k(\mathfrak{g}[1])\to \mathfrak{g}[1]; k\geq 1$ yields a linear map $\lambda:S(\mathfrak{g}[1])\to \mathfrak{g}[1]$ via $\lambda=\sum_k\lambda_k$. This data is equivalent to a degree 1 coderivation  $d_{\lambda}:S(\mathfrak{g}[1])\to S(\mathfrak{g}[1])$ defined by

\begin{equation}\label{eq:inducedcoderivation}
d_{\lambda}:=\mu_S\circ (\lambda\otimes\id)\circ \Delta_S.
\end{equation}

Also, the maps $\lambda_k:S^k(\mathfrak{g}[1])\to \mathfrak{g}[1]$ satisfy \eqref{eq:symmetricJacobi} if and only if the induced coderivation squares to zero, i.e. $d_{\lambda}^2=0$.

\begin{remark}\label{rmk:dlambda}

The map $\underline{\mathrm{Hom}}(S(\mathfrak{g}[1]),\mathfrak{g}[1])\to \mathrm{Coder}(S(\mathfrak{g}[1])); \lambda \mapsto d_{\lambda}$ as in \eqref{eq:inducedcoderivation} is an isomorphism. In particular, an $L_{\infty}$-algebra structure on $\mathfrak{g}$ is equivalent to a degree one coderivation $d\in \mathrm{Coder}(S(\mathfrak{g}[1]))$ with $d(1)=0$ and $d^2=0$.

\end{remark}

Accordingly, an $L_{\infty}$-algebra is denoted either by $(\mathfrak{g},[\cdot]^{\bullet}), (\mathfrak{g},\lambda)$ or $(S(\mathfrak{g}[1]),d)$.

%Therefore, a $L_{\infty}$-algebra is the same thing that the pair $(Sym(\mathfrak{g}\left[1\right]),d)$ where $d$ is a coderivation of square zero. For more details about these equivalences see for example \cite{Dehling}. In the forthcoming, for  the cases  where is not strictly necessary specify the $L_{\infty}$-structure we shall denote a $L_{\infty}$-algebra simply by $\mathfrak{g}$. For the cases in which  be need specify the $L_{\infty}$-structure we denote by $(\mathfrak{g},\left[\cdot\right]), (\mathfrak{g}\left[1\right],\lambda)$ or $(Sym(\mathfrak{g}\left[1\right]),d)$.

\subsection{$L_{\infty}$-algebra morphisms} Viewing an $L_{\infty}$-structure on a graded vector space as a degree 1 coderivation on its symmetric coalgebra allows us to introduce $L_{\infty}$-algebra morphisms in a simple manner.

\begin{definition}
An $L_{\infty}$-\textbf{morphism} $F:(S(\mathfrak{g}[1]),d_{\mathfrak{g}})\to (S(\mathfrak{h}[1]),d_{\mathfrak{h}})$ is a coalgebra morphism which commutes with the differentials. 
\end{definition}

That is, $F:S(\mathfrak{g}[1])\to S(\mathfrak{h}[1])$ is a degree zero linear map with the following properties:

\begin{align}
F\otimes F\circ \Delta_{S}=&\Delta_{S}\circ F \label{eq:coalgebramap}\\ 
F\circ d_{\mathfrak{g}}=&d_{\mathfrak{h}}\circ F. \label{eq:Fcommutesdifferentials}
\end{align}

Any coalgebra morphism $F:S(\mathfrak{g}[1])\to S(\mathfrak{h}[1])$ is uniquely determined by its projection onto $\mathfrak{h}[1]$, therefore equivalent to a collection of degree zero linear maps

$$ F^1_n:S^n(\mathfrak{g}
[1])\to \mathfrak{h}[1]; \quad n\geq 1.$$

An $L_{\infty}$-\textbf{morphism} $F:(S(\mathfrak{g}[1]),d_{\mathfrak{g}})\to (S(\mathfrak{h}[1]),d_{\mathfrak{h}})$ is called \textbf{strict} if $F_n^1=0$ for $n\geq 2$. This is equivalent to having $F=S(F^1_1)$ where

\begin{equation}\label{eq:symmetricextensionmap}
(SF^1_1)(v_1\vee\cdots\vee v_n)=F_1^1(v_1)\vee\cdots\vee F_1^1(v_n),
\end{equation}

\noindent is the canonical extension of $F^1_1:\mathfrak{g}[1]\to \mathfrak{h}[1]$. In this case, identity \eqref{eq:Fcommutesdifferentials} can be written in terms of symmetric brackets. More precisely, if $d_{\mathfrak{g}}$ and $d_{\mathfrak{h}}$ are respectively induced by $\lambda^{\mathfrak{g}}:S(\mathfrak{g}[1])\to \mathfrak{g}[1]$ and $\lambda^{\mathfrak{h}}:S(\mathfrak{h}[1])\to \mathfrak{h}[1]$ as in \eqref{eq:inducedcoderivation}, then \eqref{eq:Fcommutesdifferentials} is equivalent to

\begin{equation}\label{eq:strictmorphism}
\lambda^{\mathfrak{h}}\circ S(F^1_1)= F^1_1\circ \lambda^{\mathfrak{g}}.
\end{equation}

As shown in Example \ref{ex:complex}, an $L_{\infty}$-algebra $(\mathfrak{g},[\cdot]^{\bullet})$ has an underlying complex of vector spaces with differential $\partial:\mathfrak{g}\to \mathfrak{g}$ where $\partial:=[\cdot ]^{1}$. If $F$ is an $L_{\infty}$-algebra morphism, then $F^1_1:\mathfrak{g}[1]\to \mathfrak{h}[1]$ is a morphism between the underlying complexes. We say that $F$ is an  \textbf{ $L_{\infty}$-quasi-isomorphism} if the induced map in cohomology $H(F_1^1):H(\mathfrak{g},\partial)\to H(\mathfrak{h},\partial)$ is an isomorphism.

%\begin{example}
%Let us consider an extension of Lie algebras

%\begin{center}
 %   \begin{tikzcd}
  %  0\arrow[r]&\mathfrak{n}\arrow[r,"\iota"]&\hat{\mathfrak{g}}\arrow[r,"\pi"]&\mathfrak{g}\arrow[r]&0
   % \end{tikzcd}
%\end{center}

%Let 2-$\mathrm{Der}(\mathfrak{n})$ be the Lie 2-algebra corresponding to the crossed module of Lie algebras $\ad:\mathfrak{n}\to \mathrm{Der}(\mathfrak{n})$. The choice of a splitting $\sigma:\mathfrak{g}\to \hat{\mathfrak{g}}$ yields an $L_{\infty}$-morphism $F:\mathfrak{g}\to 2\text{-}\mathrm{Der}(\mathfrak{n})$ whose components are

%$$F^1_0:=0:0\to \mathfrak{n}; \quad F^1_1:=\ad\circ \sigma:\mathfrak{g}\to \mathrm{Der}(\mathfrak{n}); \quad F^1_2:=K_{\sigma}: \wedge^2\mathfrak{g}\to \mathfrak{n},$$

%\noindent where $K_{\sigma}(x,y):=[\sigma(x),\sigma(y)]-\sigma[x,y]$ is the curvature of $\sigma$. For details, see \cite{MehtaZambon}.

%\end{example}

An important example of $L_{\infty}$-quasi-isomorphism is given by the minimal model of an $L_{\infty}$-algebra. An $L_{\infty}$-algebra $(\mathfrak{g},[\cdot]^{\bullet})$ is called \textbf{minimal} if $[\cdot]^1=0$. We say that it is \textbf{contractible} if $[\cdot]^k=0$ for every $k\geq 2$ and $H(\mathfrak{g},\partial)=0$ where $\partial:=[\cdot]^1:\mathfrak{g}\to \mathfrak{g}$ is the underlying cochain complex.

\begin{example}[\textbf{Kontsevich decomposition}]\label{ex:minimalcontractible}

Every $L_{\infty}$-algebra $\mathfrak{g}$ is isomorphic to a direct sum $\mathfrak{g}\cong\mathfrak{g}_m\oplus \mathfrak{g}_c$ of a minimal and a contractible $L_{\infty}$-algebra, see \cite{Kontsevich}. It turns out that $\mathfrak{g}$ is quasi-isomorphic to its minimal component $\mathfrak{g}_m$, which is unique up to quasi-isomorphism and referred to as the \textbf{minimal model} of $\mathfrak{g}$.

\end{example}

The minimal model of a strict 2-term $L_{\infty}$-algebra can be easily described in terms of crossed modules. This is explained in the next example.

\begin{example}[\textbf{Minimal model of 2-term $L_{\infty}$-algebras}]\label{ex:2termminimal}

As shown in \cite{BaezCrans}, minimal 2-term $L_{\infty}$-algebras are in one-to-one correspondence with quadruples $(\mathfrak{g}_0,V,\rho,\theta)$ where $\rho:\mathfrak{g}_0\to \mathrm{End}(V)$ is a Lie algebra representation and $\theta\in Z^3_{CE}(\mathfrak{g}_0,V)$ is a 3-cocycle in the corresponding Chevalley-Eilenberg complex. The correspondence assigns to every minimal 2-term $L_{\infty}$-algebra $\mathfrak{g}=\mathfrak{g}_{-1}\oplus \mathfrak{g}_0$ the vector space $V:=\mathfrak{g}_{-1}$ together with the representation $\rho:\mathfrak{g}_0\otimes V\to V; (x,u)\mapsto [x,u]^2$ and 3-cocycle $\theta:=[\cdot]^3$. See Thm. 55 in \cite{BaezCrans}. Isomorphic minimal 2-term $L_{\infty}$-algebras give rise to an isomorphism between the corresponding quadruples. In particular, there is a one-to-one correspondence between quasi-isomorphism classes of 2-term $L_{\infty}$-algebras and isomorphisms of quadruples $(\mathfrak{g}_0,V,\rho,[\theta])$ with $[\theta]\in H^3_{CE}(\mathfrak{g}_0,V)$. The cohomology class $[\theta]\in H^3_{CE}(\mathfrak{g}_0,V)$ is called the \textbf{characteristic class} of a 2-term $L_{\infty}$-algebra. Now we describe the minimal model of a \emph{strict} 2-term $L_{\infty}$-algebra $\mathfrak{g}=\mathfrak{g}_{-1}\oplus \mathfrak{g}_0$, hence providing an explicit description of its characteristic class. For that, let $\partial:\mathfrak{g}_{-1}\to \mathfrak{g}_{0}$ be the underlying crossed module of $\mathfrak{g}$ and consider $\mathfrak{h}=\mathfrak{h}_{-1}\oplus \mathfrak{h}_0$ with $\mathfrak{h}_{-1}:=\ker(\partial)$ and $\mathfrak{h}_0:=\mathrm{coker}(\partial)$. The choice of a splitting $h:\mathfrak{h}_0\to \mathfrak{g}_0$ of the canonical projection $\pi:\mathfrak{g}_0\to \mathfrak{h}_0$ induces a quasi-isomorphism of complexes  

\begin{equation}\label{eq:eq1ex:2termminimal}
(\iota,h):\big{(}0:\mathfrak{h}_{-1}\to \mathfrak{h}_0\big{)}\to \big{(}\partial:\mathfrak{g}_{-1}\to \mathfrak{g}_{0}\big{)}.
\end{equation}

\noindent This quasi-isomorphism of complexes is part of a \emph{weak} quasi-isomorphism of $L_{\infty}$-algebras $F:\mathfrak{h}\to \mathfrak{g}$ with $F^1_1=(\iota,h)$ and $F^1_2:\wedge^2\mathfrak{h}_{0}\to \mathfrak{g}_{-1}$ defined as $F^1_2=\sigma_* K_h$ where $K_h:\wedge^2\mathfrak{h}_0\to \mathrm{im}(\partial)$ and $\sigma:\mathrm{im}(\partial)\to \mathfrak{g}_{-1}$ is a splitting of the exact sequence $\ker(\partial)\to \mathfrak{g}_{-1}\to \mathrm{im}(\partial)$. The $L_{\infty}$-structure on $\mathfrak{h}$ has bracket of order three given by

$$\theta=d_{h^*D}\sigma_* K_h,$$

\noindent where $h^*D:\mathfrak{h}_0\times \mathfrak{h}_{-1}\to \mathfrak{h}_{-1}; (\bar x, u)\mapsto [h(\bar x),u]$ is the pullback via $h$ of the representation $\mathfrak{g}_{0}\to \mathrm{Der}(\mathfrak{g}_{-1})$. Hence, the characteristic class of the strict 2-term $L_{\infty}$-algebra $\mathfrak{g}$ is

$$[d_{h^*D}\sigma_* K_h]\in H^3_{CE}(\mathrm{coker}(\partial),\ker(\partial)).$$

\end{example}

%%%%%%%%%%%%%%%%%%%%%%%%%
%%%%%%%%%%%%%%%%%%%%%%%%%
%%%%%%%%%%%%%%%%%%%%%%%%%

\section{$L_{\infty}$-algebra cohomology}

In this section we recall the notion of $L_{\infty}$-algebra cohomology with coefficients in a representation up to homotopy. For that we follow \cite{Reinhold} closely. See also \cite{LadaMarkl}.

%%%%%%%%%%%%%%%%%%%%%%%%%%%%%%
%%%%%%%%%%%%%%%%%%%%%%%%%%%%%%

\subsection{Representations up to homotopy}

Let $(\mathfrak{g},\lambda)$ be an $L_{\infty}$-algebra with induced coderivation $d_{\lambda}:S(\mathfrak{g}[1])\to S(\mathfrak{g}[1])$ as in \eqref{eq:inducedcoderivation}. Given a graded vector space $\mathbb{V}$, the tensor product $S(\mathfrak{g}[1])\otimes \mathbb{V}$ has a natural structure of left $S(\mathfrak{g}[1])$-comodule  defined by

$$\Delta_{\mathbb{V}}:=\Delta_S\otimes \id_{\mathbb{V}}:S(\mathfrak{g}[1])\to S(\mathfrak{g}[1])\otimes \mathbb{V},$$
where $\Delta_S$ denotes the symmetric coproduct in $S(\mathfrak{g}[1])$ as in \eqref{eq:symmetriccoproduct}.

\begin{definition}\label{def:ruth}

A \textbf{representation up to homotopy (or ruth)} of $\mathfrak{g}$ on $\mathbb{V}$ is a degree one coderivation $D:S(\mathfrak{g}[1])\otimes \mathbb{V}\to S(\mathfrak{g}[1])\otimes \mathbb{V}$ satisfying:

\begin{itemize}

\item[i)] $\Delta_{\mathbb{V}}\circ D=(d_{\lambda}\otimes \id_{S(\mathfrak{g}[1])\otimes \mathbb{V}}+\id_{S(\mathfrak{g}[1])}\otimes D)\circ \Delta_{\mathbb{V}}$

\item[ii)] $D^2=0.$
\end{itemize}

\end{definition}

The set of all ruths of $(\mathfrak{g},\lambda)$ is denoted by $\mathrm{Rep}^{\infty}(\mathfrak{g})$. An element in  $\mathrm{Rep}^{\infty}(\mathfrak{g})$ is written $(\mathbb{V},D)$ or simply $\mathbb{V}$. As shown in \cite{Reinhold}, ruths can be equivalently described as actions of $S(\mathfrak{g}[1])$ on a graded vector space $\mathbb{V}$. More precisely, a representation up to homotopy of $\mathfrak{g}$ on $\mathbb{V}$ is equivalent to a linear map

\begin{equation}\label{eq:ruthasmap}
\rho:S(\mathfrak{g}[1])\otimes \mathbb{V}\to \mathbb{V},
\end{equation}
of degree one satisfying the following condition

\begin{equation}\label{eq:equivariance}
    \rho \circ (d_{\lambda}\otimes \id_{\mathbb{V}})+\rho\circ(\id_{S(\mathfrak{g}[1])}\otimes \rho)\circ(\Delta_S \otimes\id_{\mathbb{V}})=0.
\end{equation}

In agreement with this perspective, a representation up to homotopy of $\mathfrak{g}$ is denoted by $(\mathbb{V},\rho)$ or simply $\mathbb{V}$.

\begin{remark} One can think of Definition \ref{def:ruth} as a coalgebra version of the well-known fact that a Lie algebra representation is fully characterized by its Chevalley-Eilenberg differential. Also, Definition \ref{def:ruth} resembles the notion of representation up to homotopy as in \cite{AbadCrainic} and \cite{GraciaSazMehta}. However, describing a representation up to homotopy as an action map $\rho:S(\mathfrak{g}[1])\otimes \mathbb{V}\to \mathbb{V}$ satisfying \eqref{eq:equivariance} is more geometric and we follow this perspective.
\end{remark}

\begin{example}[\textbf{Adjoint representation}]\label{ex:adjointruth}
Let $(\mathfrak{g},\lambda)$ be an $L_{\infty}$-algebra. The \textbf{adjoint representation} of $\mathfrak{g}$ is the representation up to homotopy defined by the higher brackets
 
$$\ad:S(\mathfrak{g}\left[1\right])\otimes \mathfrak{g}\left[1\right]\to \mathfrak{g}\left[1\right],\quad \ad(x\otimes y):=\lambda(x\vee y).$$

\end{example}

Just as with Lie algebras and their representations, looking at ruths of a given $L_{\infty}$-algebra allows us to define two algebraic objects. Namely, its category of representations up to homotopy and the Chevalley-Eilenberg cohomology with coefficients in a representation up to homotopy. The former is described in what follows.

%$(\mathfrak{g},\lambda)$ form a category $\mathrm{Rep}_{\infty}(\mathfrak{g})$ in which morphisms are given by equivariant maps. 

\begin{definition}\label{def:ruthmorphism}
Let $(\mathbb{V},\rho)$ and $(\mathbb{V}',\rho')$ be two representations up to homotopy of $\mathfrak{g}$ and $\mathfrak{g}'$, respectively. A \textbf{ruth morphism} $(F,f):(\mathbb{V},\rho)\to (\mathbb{V}',\rho')$ is given by a couple of maps $F:\mathfrak{g}\to \mathfrak{g}'$ which is an $L_{\infty}$-morphism and $f:\mathbb{V}'\to \mathbb{V}$ a degree-preserving map satisfying
\begin{equation}\label{eq:ruthmorphism}
    f\circ \rho' \circ (F\otimes \id_{\mathbb{V}'})=\rho\circ(\id_{S(\mathfrak{g}[1])}\otimes f).
\end{equation}
\end{definition}

In the particular case that $\mathfrak{g}=\mathfrak{g}'$ and $F=\id$, we simply say that the map $f:\mathbb{V}'\to \mathbb{V}$ is \textbf{ruth map}. One easily checks that if $\mathfrak{g}$ and $\mathfrak{g}'$ are Lie algebras, $\mathbb{V}$ and $\mathbb{V}'$ are usual representations, then Definition \ref{def:ruthmorphism} recovers the known notion of morphism of Lie algebra representations.

One observes that ruth morphisms can be composed according to the next proposition, whose proof is a straightforward computation.

\begin{proposition}\label{compositionequiavrianmaps}

Let $(\mathbb{V},\rho), (\mathbb{V}',\rho')$ and $(\mathbb{V}'',\rho'')$ be representations up to homotopy of $\mathfrak{g},\mathfrak{g}'$ and $\mathfrak{g}''$, respectively. If $(F,f):(\mathbb{V},\rho)\to (\mathbb{V}',\rho')$ and $(F',f'):(\mathbb{V}',\rho')\to (\mathbb{V}'',\rho'')$ are ruths morphisms, then
\[(F'\circ F, f\circ f'):(\mathbb{V},\rho)\to (\mathbb{V}'',\rho''),\]
is a ruths morphism.

\end{proposition}

Given an $L_{\infty}$-algebra $\mathfrak{g}$, its \textbf{category of representations up to homotopy} is the category $\mathcal{R}ep_{\infty}(\mathfrak{g})$ whose objects are ruths of $\mathfrak{g}$ and morphisms are defined as ruth maps. 

One observes that an $L_{\infty}$-algebra morphism $F:\mathfrak{g}\to \mathfrak{g}'$ yields a pullback $F^*:\mathrm{Rep}_{\infty}(\mathfrak{g}')\to \mathrm{Rep}_{\infty}(\mathfrak{g})$. If $(\mathbb{V}',\rho')\in \mathrm{Rep}_{\infty}(\mathfrak{g}')$ is a ruth, the \textbf{pullback ruth} $(\mathbb{V}',F^*\rho')\in \mathrm{Rep}_{\infty}(\mathfrak{g})$ is defined as

\begin{equation}\label{eq:pullbackruth}
F^*\rho':S(\mathfrak{g}[1])\otimes \mathbb{V}'\to \mathbb{V}',\quad F^*\rho':=\rho'\circ (F\otimes  \id_{\mathbb{V}'}).
\end{equation}

\begin{remark}\label{decomposition.eq.maps}

Let $(\mathbb{V},\rho)$ and $(\mathbb{V}',\rho')$ be representations up to homotopy of $\mathfrak{g}$ and $\mathfrak{g}'$, respectively. Any ruth morphism $(F,f):(\mathbb{V},\rho)\to (\mathbb{V}',\rho')$ factors through a pair of ruth morphisms $(F,\id_{\mathbb{V}'}):(\mathbb{V}',F^*\rho')\to (\mathbb{V}',\rho')$ and $(\id_{\mathfrak{g}},f):(\mathbb{V},\rho)\to (\mathbb{V}',F^*\rho')$. That is, $(F,f)=(\id_{\mathfrak{g}},f)\circ (F,\id_{\mathbb{V}'})$.

\end{remark}

%Let $\mathbb{V}$ be a graded vector space. The space of endomorphisms $\underline{\mathrm{End}}(\mathbb{V})$ is a graded Lie algebra with respect to the graded commutator $[F,G]:=F\circ G-(-1)^{|F||G|}G\circ F$. Additionally, if $\partial:\mathbb{V}\to \mathbb{V}$ is a degree one linear map with $\partial^2=0$, then $\underline{\mathrm{End}}(\mathbb{V})$ is a dgla with differential $D:=[\partial,\cdot]$. In particular, for a given $L_{\infty}$-algebra $(\mathfrak{g},\lambda)$ the induced coderivation $d_{\lambda}:S(\mathfrak{g})\to S(\mathfrak{g})$ as in \eqref{eq:inducedcoderivation} satisfies $d_{\lambda}^2=0$ and hence there is an associated dgla structure on $\underline{\mathrm{End}}(S(\mathfrak{g}))$ with differential $D=[d_{\lambda},\cdot]$. As shown in \cite{Reinhold}, the subspace of coderivations $\mathrm{Coder}(S(\mathfrak{g}))\subset \underline{\mathrm{End}}(S(\mathfrak{g}))$ is closed under the graded commutator, hence $\mathrm{Coder}(S(\mathfrak{g}))$ has an induced structure of graded Lie algebra. Combining this with Remark \ref{rmk:dlambda} one concludes that $\underline{\mathrm{Hom}}(S(\mathfrak{g}),\mathfrak{g})$ inherits the structure of graded Lie algebra.

We proceed now to recall the Chevalley-Eilenberg complex of an $L_{\infty}$-algebra together with a representation up to homotopy. For details, see \cite{Dehling} and \cite{Reinhold} and the references therein. Let $(\mathfrak{g},\lambda)$ be an $L_{\infty}$-algebra and $\mathbb{V}$ a graded vector space. Consider the graded vector space

\begin{equation}\label{eq:cochains}
C(\mathfrak{g},\mathbb{V}):=\bigoplus_{p\in \mathbb{Z}} C^p(\mathfrak{g},\mathbb{V}),
\end{equation}

\noindent where $C^p(\mathfrak{g},\mathbb{V})=\underline{\Hom}^p(S(\mathfrak{g}\left[1\right]),\mathbb{V})$. In other words $C(\mathfrak{g},\mathbb{V}):=\underline{\Hom}(S(\mathfrak{g}\left[1\right]),\mathbb{V})$ is the graded vector space of graded linear maps on $S(\mathfrak{g}[1])$ with values in $\mathbb{V}$. Recall that $\mathfrak{g}$ has an associated coproduct $\Delta_S:S(\mathfrak{g}[1])\to S(\mathfrak{g}[1])\otimes S(\mathfrak{g}[1])$ as in \eqref{eq:symmetriccoproduct} and induced degree 1 coderivation $d_{\lambda}:S(\mathfrak{g}[1])\to S(\mathfrak{g}[1])$ as in \eqref{eq:inducedcoderivation}.

A representation up to homotopy $(\mathbb{V},\rho)\in \mathrm{Rep}_{\infty}(\mathfrak{g})$ determines a cochain complex $(C(\mathfrak{g},\mathbb{V}),D_{\rho})$ where the differential $D_{\rho}:C(\mathfrak{g},\mathbb{V})\to C(\mathfrak{g},\mathbb{V})$ is given by

\begin{equation}\label{eq:CEdifferential} 
D_{\rho}\alpha:=\rho\circ (\id_{S(\mathfrak{g}\left[1\right])}\otimes \alpha)\circ\Delta_S -(-1)^{p}\alpha\circ d_{\lambda},
\end{equation}

\noindent for every homogeneous element $\alpha \in C^p(\mathfrak{g},\mathbb{V})$. 

%The complex $(C_{\rho}(\mathfrak{g},\mathbb{V}),D_{\rho})$ is referred to as the \textbf{Chevalley-Eilenberg complex} of $(\mathbb{V},\rho)\in \mathrm{Rep}_{\infty}(\mathfrak{g})$.

\begin{definition}\label{def:Linftycohomology}

The cohomology $H(\mathfrak{g},\mathbb{V})$ of the complex $(C(\mathfrak{g},\mathbb{V}),D_{\rho})$ is called the \textbf{$L_{\infty}$- algebra cohomology} of $\mathfrak{g}$ with coefficients in $\mathbb{V}$.

\end{definition}

If  $\mathfrak{g}$ is a Lie algebra and $\mathbb{V}$ is a usual representation, then the cochain complex $(C(\mathfrak{g},\mathbb{V}),D_{\rho})$ coincides with the usual Chevalley-Eilenberg complex associated to a Lie algebra representation. In particular, the Chevalley-Eilenberg cohomology of a Lie algebra is a particular instance of the $L_{\infty}$-algebra cohomology with coefficients in a ruth. For more details see \cite{Reinhold}.

%\begin{definition}[Equivariance up to homotopy]
%Let $\rho$ and $\rho'$ be two representations up to homotopy of $\mathfrak{g}$ on $\mathbb{V}$ and $\mathfrak{h}$ on $\mathbb{W}$, respectively. A $(\rho,\rho')$-\textbf{equivariant morphism} from $\mathfrak{g}$ to $\mathfrak{h}$ is a couple of maps $(F,f)$ where $F:\mathfrak{g}\to \mathfrak{h}$ is a $L_{\infty}$-morphism and $f:\mathbb{W}\to \mathbb{V}$ is a degree-preserving map such that the next identity holds
%\begin{equation}\label{eq:ruthmorphism}
  %  f\circ \rho' \circ (F\otimes \id_{v})=\rho\circ(\id_{s}\otimes f).
%\end{equation}

%In the particular case that $\mathfrak{g}=\mathfrak{h}$ and $F=\id$ we simply say that the map $f:\mathbb{W}\to \mathbb{V}$ is $(\rho,\rho')$-equivariant.
%\end{definition}

The $L_{\infty}$-algebra cohomology with coefficients is functorial as stated in the next proposition.

\begin{proposition}\label{equivariance}

Let $(F,f):(\mathbb{V},\rho)\to (\mathbb{V}',\rho')$ be a morphism between ruths of $\mathfrak{g}$ and $\mathfrak{g}'$. The pullback map $F^*:C(\mathfrak{g}',\mathbb{V}')\to C(\mathfrak{g},\mathbb{V}); \alpha\mapsto F^*\alpha:=f\circ \alpha \circ F$  is a cochain map and hence induces a well-defined map $F^*:H(\mathfrak{g}',\mathbb{V}')\to H(\mathfrak{g},\mathbb{V}')$.

\end{proposition}
\begin{proof}
Let us see that $F^* D_{\rho'}=D_{\rho}F^*$. Indeed, if $\alpha$ is a homogeneous element of degree $p$ in  $C(\mathfrak{h},\mathbb{W})$, then
\begin{align*}
    F^*(D_{\rho'}\alpha)=&f\circ (\rho'\circ (\id_{S(\mathfrak{g}'\left[1\right])}\otimes \alpha)\circ\Delta_S -(-1)^{p}\alpha\circ d_{\lambda'})\circ F\\
    =&f\rho'(\id_{S(\mathfrak{g}'[1])}\otimes \alpha)(F\otimes F)\Delta_S-(-1)^{p} f\circ \alpha \circ F\circ d\\
    =&f\rho'(F\otimes \id_{\mathbb{V}})(\id_{S(\mathfrak{g}'[1])}\otimes \alpha)(\id_{S(\mathfrak{g}'[1])}\otimes F)\Delta_S-(-1)^{p}(F^*\alpha) d\\
    =&\rho(\id_{S(\mathfrak{g}[1])}\otimes f)(\id_{S(\mathfrak{g}[1])}\otimes \alpha)(\id_{S(\mathfrak{g}[1])}\otimes F)\Delta_S-(-1)^p(F^*\alpha) d,\quad \text{by equation (\ref{eq:ruthmorphism}})\\
    =&\rho(\id_{S(\mathfrak{g}[1])}\otimes(F^*\alpha))\Delta_S-(-1)^p(F^*\alpha)d\\
    =&D_{\rho}(F^*\alpha).
\end{align*}
\end{proof}

\subsection{Tensor products}

Recall that given graded vector spaces $\mathbb{V}_1,\mathbb{V}_2$, the \textbf{twisting map} is defined by

\begin{equation}\label{eq:twistingmap}
T:\mathbb{V}_1\otimes \mathbb{V}_2\to \mathbb{V}_2\otimes \mathbb{V}_1; v\otimes w\mapsto (-1)^{|v||w|}w\otimes v.
\end{equation} 

The following result shows that representations up to homotopy have a natural tensor product operation.

\begin{theorem}\label{thm:tensorproducts}

Let $\mathfrak{g}$ be  a $L_{\infty}$-algebra and  $(\mathbb{V}_1,\rho_1), (\mathbb{V}_2,\rho_2)\in \mathrm{Rep}_{\infty}(\mathfrak{g})$. The tensor product $\mathbb{V}_1\otimes\mathbb{V}_2$ inherits the structure of a representation up to homotopy given by

$$\rho_1\otimes \rho_2:S(\mathfrak{g}\left[1\right])\otimes (\mathbb{V}_1\otimes\mathbb{V}_2)\to \mathbb{V}_1\otimes\mathbb{V}_2,\quad \rho:=\rho_1\otimes\id_{\mathbb{V}_2}+(\id_{\mathbb{V}_1}\otimes\rho_2)\circ (T\otimes \id_{\mathbb{V}_2})$$

\noindent where $T:S(\mathfrak{g}\left[1\right])\otimes\mathbb{V}_1\to \mathbb{V}_1 \otimes S(\mathfrak{g}\left[1\right])$ is the twisting map. More concretely, on homogeneous elements $x \in S(\mathfrak{g}\left[1\right]), v\in \mathbb{V}_1$ and $w\in \mathbb{V}_2$ one has

\begin{equation}\label{eq:tensorruthhomogeneous}
\rho_1\otimes \rho_2(x\otimes v\otimes w)=\rho_1(x\otimes v)\otimes w+(-1)^{(|x|+1)|v|}v\otimes \rho_2(x\otimes w).
\end{equation}

\end{theorem}

The proof of Theorem \ref{thm:tensorproducts} is a straightforward but long computation which can be found in \cite{Herrera}. The \textbf{tensor product ruth} of $(\mathbb{V}_1,\rho_{1})$ and $(\mathbb{V}_2,\rho_2)$ is denoted by $(\mathbb{V}_1\otimes \mathbb{V}_2,\rho_1\otimes \rho_2)$.

In particular, if $n\geq 2$ and $(\mathbb{V},\rho)\in \mathrm{Rep}_{\infty}(\mathfrak{g})$ then the $n^{th}$-tensor power $\mathbb{V}^{\otimes n}$ inherits the structure of a ruth defined as follows:

\begin{equation}\label{eq:ntensorruth}
\rho^{\otimes n}:=\rho\otimes \id^{\otimes n-1}+\sum_{k=2}^{n}(\id^{\otimes k-1}\otimes\rho \otimes \id^{\otimes n-k})\circ T^k,
\end{equation}

\noindent where $T^k:\mathbb{V}^{\otimes n}\to \mathbb{V}^{\otimes n}; k\geq 2$, is the degree preserving map

\begin{equation}\label{eq:kshiftmap}
T^k:=(\id^{\otimes k-2}\otimes T\otimes \id^{\otimes n-k})\circ T^{k-1},
\end{equation}

\noindent and $T:\mathbb{V}\otimes \mathbb{V}\to \mathbb{V}\otimes\mathbb{V}, v\otimes w\mapsto (-1)^{|v||w|}w\otimes v$ is the twisting map. More concretely, on homogeneous elements $x\in S(\mathfrak{g}[1])$ and $v_1,...,v_n\in \mathbb{V}$ one has

$$T^k(v_1\otimes\cdots\otimes v_n)=(-1)^{|v_1|(\sum_{j=2}^{k}|v_j|)}v_2\otimes\cdots\otimes v_k\otimes v_1\otimes v_{k-1}\otimes \cdots\otimes v_n$$

\noindent and

\begin{align*}
\rho^{\otimes n}(x\otimes v_1\otimes&\cdots \otimes v_n)=\rho(x\otimes v_1)\otimes v_2\otimes\cdots\otimes v_n+\\
&+\sum_{k=2}^{n}(-1)^{(|x|+1)(\sum_{j=1}^{k-1}|v_j|)}v_1\otimes\cdots \otimes v_{k-1}\otimes \rho (x\otimes v_k)\otimes v_{k+1}\otimes\cdots\otimes v_n.
\end{align*}
 
\subsection{Cochain products}\label{subsec:cochainproducts}

Let $\mathfrak{g}$ be a $L_{\infty}$-algebra and $\mathbb{V}_i; i=1,2,3$ graded vector spaces. A degree preserving map $m:\mathbb{V}_1\otimes \mathbb{V}_2\to \mathbb{V}_3$ yields a degree zero linear map

\begin{equation}\label{eq:cochainmap}
\wedge_m:C(\mathfrak{g},\mathbb{V}_1)\otimes C(\mathfrak{g},\mathbb{V}_2)\to C(\mathfrak{g},\mathbb{V}_3),
\end{equation}

\noindent defined on homogeneous elements $\alpha \in C^p(\mathfrak{g},\mathbb{V}_1)$ and  $\beta\in C^q(\mathfrak{g},\mathbb{V}_2)$ by 

\begin{equation}\label{eq:cochainproduct}
\alpha\wedge_{m}\beta:=m\circ \alpha\otimes \beta \circ \Delta_S.
\end{equation}

The element $\alpha\wedge_m\beta \in C^{p+q}(\mathfrak{g},\mathbb{V}_3)$ is called the \textbf{cochain product} between $\alpha$ and $\beta$ \textbf{along $m$}.

We say that $m:\mathbb{V}_1\otimes \mathbb{V}_2\to \mathbb{V}_3$ is \textbf{symmetric} if $m=m\circ T$. Similarly, if $m=-m\circ T$ one says that $m$ is \textbf{skew-symmetric}. The following can be verified by a straightforward computation.

\begin{proposition}\label{prop:symmetrycochain}
Let $\mathfrak{g}$ be a $L_{\infty}$-algebra and $m:\mathbb{V}_1\otimes \mathbb{V}_2\to \mathbb{V}_3$ a degree zero map with induced cochain product $\wedge_m:C(\mathfrak{g},\mathbb{V}_1)\otimes C(\mathfrak{g},\mathbb{V}_2)\to C(\mathfrak{g},\mathbb{V}_3)$. Then $\wedge_m$ is symmetric (resp. skew-symmetric) if $m$ is symmetric (resp. skew-symmetric).

\end{proposition}

Assume now that $(\mathbb{V}_i,\rho_i)\in \mathrm{Rep}_{\infty}(\mathfrak{g}); i=1,2,3$ are representations up to homotopy. It is reasonable to require $m:\mathbb{V}_1\otimes \mathbb{V}_2\to \mathbb{V}_3$ to be an equivariant map, that is

\begin{equation}\label{eq:equivariantcochainproduct}
m\circ (\rho_1\otimes \rho_2)=\rho_3\circ (\id_{S(\mathfrak{g}[1])}\otimes m),
\end{equation}

\noindent where $\rho_1\otimes \rho_2$ denotes the tensor product ruth on $\mathbb{V}_1\otimes \mathbb{V}_2$ as in Theorem \ref{thm:tensorproducts}. In this case, the following Leibniz identity holds.

\begin{proposition}\label{prop:leibnizcochain}
Let $\mathfrak{g}$ be an $L_{\infty}$-algebra and $(\mathbb{V}_i,\rho_i)\in\mathrm{Rep}_{\infty}(\mathfrak{g})$ ruths with induced differentials $D_{\rho_i}:C(\mathfrak{g},\mathbb{V}_i)\to C(\mathfrak{g},\mathbb{V}_i)$ as in \eqref{eq:CEdifferential}. If $m:\mathbb{V}_1\otimes \mathbb{V}_2\to \mathbb{V}_3$ is equivariant, then

    \begin{equation}\label{eq:leibnizcochain}
    D_{\rho_3}(\alpha \wedge_{m}\beta)=(D_{\rho_1}\alpha)\wedge_{m}\beta +(-1)^{|\alpha|}\alpha\wedge_{m} (D_{\rho_2}\beta),
    \end{equation}

\noindent for every homogeneus elements $\alpha\in C(\mathfrak{g},\mathbb{V}_1)$ and $\beta\in C(\mathfrak{g},\mathbb{V}_2)$.

\end{proposition}

\begin{proof}

We start by computing the left hand side of \eqref{eq:leibnizcochain}. If $\alpha\in C(\mathfrak{g},\mathbb{V}_1)$ and $\beta\in C(\mathfrak{g},\mathbb{V}_2)$ are homogeneous, then using the definition of $D_{\rho_{i}}$ together with the interchange law \eqref{eq:interchangelaw}, one has
\begin{align*}
     D_{\rho_3}(\alpha\wedge_{m}\beta)=&\rho_3\circ (\id_{S(\mathfrak{g}[1])}\otimes \alpha\wedge_{m} \beta)\circ \Delta_S-(-1)^{|\alpha|+|\beta|}\alpha\wedge_{m}\beta\circ d\\
     =&\rho_3\circ (\id_{S(\mathfrak{g}[1])}\otimes(m\circ \alpha\otimes\beta \circ \Delta_S))\circ \Delta_S-(-1)^{|\alpha|+|\beta|}m\circ \alpha\otimes \beta \circ \Delta_S\circ d\\
     =&\rho_3\circ (\id_{S(\mathfrak{g}[1])}\otimes m)\circ (\id_{S(\mathfrak{g}[1])}\otimes\alpha\otimes\beta)\circ (\id_{S(\mathfrak{g}[1])}\otimes\Delta_S)\circ\Delta_S-\\
     &(-1)^{|\alpha|+|\beta|}m\circ \alpha\otimes\beta \circ (\id_{S(\mathfrak{g}[1])}\otimes d+d\otimes\id_{S(\mathfrak{g}[1])})\circ \Delta_S\\
     =&m\circ \rho_1\otimes \rho_2\circ (\id_{S(\mathfrak{g}[1])}\otimes \alpha\otimes \beta)\circ \Delta_S^2-(-1)^{|\alpha|+|\beta|}m\circ (\alpha\otimes(\beta\circ d))\circ \Delta_S+\\
     &-(-1)^{|\alpha|+|\beta|}m\circ (-1)^{|\beta|}((\alpha\circ d)\otimes \beta) \circ \Delta_S)\\
     =&A+ B\\
     \end{align*}
     
\noindent where 

\begin{align*}
A=&m\circ \big{(}(\rho_1\circ (\id_{S(\mathfrak{g}[1])}\otimes\alpha))\otimes \beta+(\id_{S(\mathfrak{g}[1])}\otimes \rho_2)\circ(T\otimes \id_{S(\mathfrak{g}[1])})\circ(\id_{S(\mathfrak{g}[1])}\otimes \alpha\otimes \beta)\big{)}\circ \Delta_S^2\\
 B=&m\circ( -(-1)^{|\alpha|+|\beta|}\alpha\otimes(\beta\circ d) -(-1)^{|\alpha|} (\alpha\circ d)\otimes \beta) \circ \Delta_S.\\
 \end{align*}

\noindent Since the coproduct is coassociative and cocommutative $\Delta_S^2=(\Delta_S\otimes\id_s)\circ \Delta_S=(\id_s\otimes\Delta_S)\circ \Delta_S$ and $(T\otimes\id_s)\circ\Delta_S^2=\Delta_S^2$, one concludes
     \begin{align*}
     A=&m\circ((\rho_1\circ(\id_s\otimes\alpha))\otimes \beta \circ (\Delta_S\otimes \id_s)\circ \Delta_S+(\id_s\otimes \rho_2)\circ (\alpha\otimes\id_s\otimes \beta)\circ (\id_s\otimes \Delta_S)\circ \Delta_S)\\
     =&m\circ(((\rho_1\circ(\id_s\otimes\alpha)\circ \Delta_S)\otimes \beta) \circ  \Delta_S+(-1)^{|\alpha|}(\alpha \otimes (\rho_2\circ (\id_s\otimes \beta)\circ \Delta_S))\circ \Delta_S)\\
     =&m\circ(((\rho_1\circ(\id_s\otimes\alpha)\circ \Delta_S)\otimes \beta) +(-1)^{|\alpha|}(\alpha \otimes (\rho_2\circ (\id_s\otimes \beta)\circ \Delta_S)))\circ \Delta_S.\\
 \end{align*}
 
 \noindent Now, one can easily see that

 \begin{align*}
      D_{\rho_3}(\alpha\wedge_{m}\beta)=&A+B\\
      =&D_{\rho_1} \alpha\wedge_{m}\beta+(-1)^{|\alpha|}\alpha\wedge_{m}D_{\rho_2}\beta,
 \end{align*}
as claimed.
\end{proof}

It follows from Proposition \ref{prop:leibnizcochain} that the product \eqref{eq:cochainmap} descends to a map in cohomology

\begin{align}\label{eq:cochainmapcohomology}
\wedge_m:H^{\bullet}(\mathfrak{g},\mathbb{V}_1)\otimes &H^{\bullet}(\mathfrak{g},\mathbb{V}_2)\to H^{\bullet}(\mathfrak{g},\mathbb{V}_3)\\
[\alpha]\otimes&[\beta]\mapsto [\alpha\wedge_m\beta].\nonumber
\end{align}

The following special instance of Proposition \ref{prop:leibnizcochain} will be relevant in section \ref{sec:CWL}. Let $\mathbb{V}$ be a graded algebra with product $m:\mathbb{V}\otimes \mathbb{V}\to \mathbb{V}$ and $\rho:S(\mathfrak{g}[1])\otimes \mathbb{V}\to \mathbb{V}$ is a representation up to homotopy. Using \eqref{eq:tensorruthhomogeneous} one has that $m:\mathbb{V}\otimes \mathbb{V}\to \mathbb{V}$ is a ruth map if and only if for every homogeneous elements $x\in S(\mathfrak{g}[1])$ and $v,w\in \mathbb{V}$ the following holds

\begin{equation}\label{eq:ruthderivations}
(\rho\otimes \rho)(x\otimes m(v\otimes w))=m(\rho(x\otimes v)\otimes w)+(-1)^{(|x|+1)|v|}m(v\otimes \rho(x\otimes w)).
\end{equation}

\noindent In other words, property \eqref{eq:ruthderivations} is equivalent to saying that $\mathfrak{g}$ acts on $(\mathbb{V},m)$ by graded derivations.

\begin{corollary}\label{cor:cochaincohomologygradedalgebra}
 
Let  $\mathfrak{g}$ be an $L_{\infty}$-algebra and $(\mathbb{V},m)$ a graded algebra. If $\rho:S(\mathfrak{g}[1])\otimes \mathbb{V}\to \mathbb{V}$ is an action by derivations, then $(H^{\bullet}(\mathfrak{g},\mathbb{V}),\wedge_m)$ is a graded algebra where $\wedge_m$ is defined by \eqref{eq:cochainmapcohomology} 

\begin{align}\label{eq:cochainmapgradedalgebra}
\wedge_m:H^{\bullet}(\mathfrak{g},\mathbb{V})\otimes &H^{\bullet}(\mathfrak{g},\mathbb{V})\to H^{\bullet}(\mathfrak{g},\mathbb{V})\\
[\alpha]\otimes&[\beta]\mapsto [\alpha\wedge_m\beta].\nonumber
\end{align}

\end{corollary}

Let us see some examples of induced cochain products.

\begin{example}\label{ex:Idcochain}
Let $(\mathbb{V},\rho)$ be a representation up to homotopy of $\mathfrak{g}$ on $\mathbb{V}$. The identity map $\id_{\mathbb{V}\otimes\mathbb{V}}:\mathbb{V}\otimes\mathbb{V}\to \mathbb{V}\otimes\mathbb{V}$ determines a cochain product map

$$C(\mathfrak{g},\mathbb{V})\otimes C(\mathfrak{g},\mathbb{V})\to C(\mathfrak{g},\mathbb{V}\otimes \mathbb{V}).$$
In this case, the cochain product of $\alpha,\beta\in C(\mathfrak{g},\mathbb{V})$ is denoted by $\alpha\wedge_{\id}\beta$. Note that this cochain product is neither symmetric nor skew-symmetric.

\end{example}

\begin{example}\label{ex:evaluationcochain}
Consider now a graded vector space $\mathbb{V}$. The evaluation map is given by 

\begin{equation}\label{eq:evaluationmap}
ev:\underline{\mathrm{End}}(\mathbb{V})\otimes \mathbb{V}\to \mathbb{V}; T\otimes v\mapsto T(v).
\end{equation}

\noindent Clearly \eqref{eq:evaluationmap} has degree zero, so it induces a product on cochains $\wedge_{ev}:C(\mathfrak{g},\underline{\mathrm{End}}(\mathbb{V}))\otimes C(\mathfrak{g},\mathbb{V})\to C(\mathfrak{g},\mathbb{V})$.

\end{example}

\begin{example}\label{ex:cochainbracket}
Let $(L,\left[\cdot,\cdot\right])$ be a graded Lie algebra. That is, $L=\bigoplus_{n\in \mathbb{Z}}L_n$ is a graded vector space and $\left[\cdot,\cdot\right]:L\otimes L\to L$ is a degree zero linear map which is skew-symmetric and satisfies the graded Jacobi identity:

    $$\left[x,\left[y,z\right]\right]=\left[y,\left[x,z\right]\right]+(-1)^{rs}\left[y,\left[x,z\right]\right].$$
    
\noindent If $\mathfrak{g}$ is an $L_{\infty}$-algebra, the bracket $\left[\cdot,\cdot\right]:L\otimes L\to L$  defines a cochain product $\wedge_{[\cdot,\cdot]}:C(\mathfrak{g},L)\otimes C(\mathfrak{g},L)\to C(\mathfrak{g},L)$ which is skew-symmetric.  
\end{example}

In particular, if $(L,[\cdot,\cdot], d)$ is a dgla (see Example \ref{ex:dglas}), that is, the linear map $d:L\to L$ has degree +1, it squares to zero $d^2=0$ and 

$$d(\left[x,y\right])=\left[dx,y\right]+(-1)^{|x|}\left[x,dy\right],$$
then the cochain product $\wedge_{[\cdot,\cdot]}:C(\mathfrak{g},L)\otimes C(\mathfrak{g},L)\to C(\mathfrak{g},L)$  is well-defined too. As shown in Corollary \ref{cor:cochaincohomologygradedalgebra}, if the space of coefficients of a ruth is a graded algebra, then such an additional structure descends to cohomology. The next result shows that if the space of coefficients is a dgla, then the corresponding space of \emph{cochains} inherits the structure of a dgla.

\begin{proposition}\label{prop:convolutionDGLA}
Let $(\mathfrak{g},\lambda)$ be an $L_{\infty}$-algebra with  induced degree 1 coderivation $d_{\lambda}:S(\mathfrak{g}[1])\to S(\mathfrak{g}[1])$ as in \eqref{eq:inducedcoderivation}. Assume that $(L, \left[\cdot,\cdot\right])$ is a graded Lie algebra, then $(C(\mathfrak{g},L),\wedge_{\left[\cdot,\cdot\right]})$ is a graded Lie algebra. Moreover, if $(L,\left[\cdot,\cdot\right],d)$ is a dgla, then $(C(\mathfrak{g},L),\wedge_{\left[\cdot,\cdot\right]},\textbf{d})$ is dgla with differential given by 
\[\textbf{d}(\alpha):=d\circ \alpha-(-1)^{|\alpha|}\alpha\circ d_{\lambda}\]
for every homogeneous element $\alpha\in C(\mathfrak{g},L)$.
\end{proposition}

\begin{proof}
It follows from Proposition \ref{prop:symmetrycochain} that the product $\wedge_{\left[\cdot,\cdot\right]}$ is skew-symmetric, we will see that it also satisfies the Jacobi identity. For that, let $\alpha,\beta,\gamma\in C(\mathfrak{g},L)$ be homogeneous elements, then:
\begin{align*}
    \alpha\wedge_{\left[\cdot,\cdot\right]}(\beta\wedge_{\left[\cdot,\cdot\right]} \gamma)=&\left[\cdot,\cdot\right]\circ(\alpha \otimes \beta\wedge_{\left[\cdot,\cdot\right]} \gamma)\circ \Delta_S=\left[\cdot,\cdot\right]\circ (\alpha \otimes (\left[\cdot,\cdot\right]\circ \beta \otimes \gamma\circ \Delta))\circ \Delta\\
    =&\left[\cdot,\cdot\right]\circ(\id\otimes \left[\cdot,\cdot\right])\circ\alpha\otimes\beta\otimes\gamma \circ(\id\otimes \Delta)\circ\Delta\\
    =&\left[\cdot,\cdot\right]\circ(\left[\cdot,\cdot\right]\otimes\id)\circ \alpha\otimes\beta\otimes\gamma \circ \Delta^2+\left[\cdot,\cdot\right]\circ(\id\otimes \left[\cdot,\cdot\right])\circ(T\otimes\id)\circ\alpha\otimes\beta\otimes\gamma \circ \Delta^2\\
    =&(\alpha\wedge_{\left[\cdot,\cdot\right]}\beta)\wedge_{\left[\cdot,\cdot\right]}\gamma+(-1)^{|\alpha||\beta|}\beta\wedge_{\left[\cdot,\cdot\right]}(\alpha\wedge_{\left[\cdot,\cdot\right]}\gamma).
\end{align*}

Assume now that $(L,\left[\cdot,\cdot\right],d)$ is dgla. We will show that the map $\textbf{d}$ is derivation of the product $\wedge_{\left[\cdot,\cdot\right]}$ with $\textbf{d}^2=0.$ The derivation property follows by a direct computation on homogeneous elements $\alpha,\beta\in C(\mathfrak{g},L)$. In fact,
\begin{align*}
    \textbf{d}(\alpha\wedge_{\left[\cdot,\cdot\right]}\beta)=&d\circ \alpha\wedge_{\left[\cdot,\cdot\right]}\beta-(-1)^{|\alpha|+|\beta|}\alpha\wedge_{\left[\cdot,\cdot\right]}\beta\circ d_{\lambda}\\
    =&d\circ\left[\cdot,\cdot\right]\circ\alpha\otimes\beta\circ\Delta-(-1)^{|\alpha|+|\beta|}\left[\cdot,\cdot\right]\circ\alpha\otimes\beta \circ\Delta\circ d_{\lambda}\\
    =&\left[\cdot,\cdot\right]\circ(d\otimes \id+\id\otimes d)\circ\alpha\otimes\beta\circ\Delta-(-1)^{|\alpha|+|\beta|}\left[\cdot,\cdot\right]\circ\alpha\otimes\beta \circ(d_{\lambda}\otimes\id+\id\otimes d_{\lambda})\circ\Delta\\
    =&\left[\cdot,\cdot\right]\circ (d\circ \alpha)\otimes\beta \circ\Delta +\left[\cdot,\cdot\right]\circ(-1)^{|\alpha|}\alpha\otimes(d\circ\beta)\circ\Delta-(-1)^{|\alpha|}\left[\cdot,\cdot\right]\circ(\alpha\circ d_{\lambda})\otimes\beta \circ\Delta\\
    &-(-1)^{|\alpha|+|\beta|}\left[\cdot,\cdot\right]\circ\alpha\otimes(\beta\circ d_{\lambda})\circ \Delta\\
    =&\left[\cdot,\cdot\right]\circ(d\circ\alpha-(-1)^{|\alpha|}\alpha\circ d_{\lambda})\otimes\beta \circ\Delta+(-1)^{|\alpha|}\left[\cdot,\cdot\right]\circ\alpha\otimes(d\circ\beta-(-1)^{|\beta|}\beta\circ d_{\lambda})\circ\Delta\\
    =&\textbf{d}\alpha\wedge_{\left[\cdot,\cdot\right]}\beta+(-1)^{|\alpha|}\alpha\wedge_{\left[\cdot,\cdot\right]}\textbf{d}\beta
\end{align*}

\noindent As $d^2=0$ and $d^2_{\lambda}=0$, it follows that $\textbf{d}^2=0$ as well.
\end{proof}

As for any dgla, elements of degree 1 has a well defined curvature.

\begin{definition}\label{def:curvaturedgla}

Let $\mathfrak{g}$ be an $L_{\infty}$-algebra and $(L,[\cdot,\cdot],d)$ a dgla. Let $C(\mathfrak{g},L)$ be the dgla of Proposition \ref{prop:convolutionDGLA}. If $\theta\in C^1(\mathfrak{g},L)$ is an element of degree 1, its \textbf{curvature} is the element $\Omega_{\theta}\in C^2(\mathfrak{g},L)$ defined by

$$\Omega_{\theta}=\textbf{d}\theta+ \frac{1}{2}\theta\wedge_{[\cdot,\cdot]}\theta.$$

\end{definition}

A \textbf{Maurer-Cartan} element in the dgla $C(\mathfrak{g},L)$ is a homogeneous element $\theta\in C^1(\mathfrak{g},L)$ whose curvature vanishes. That is, 

$$\textbf{d}\theta+ \frac{1}{2}\theta\wedge_{[\cdot,\cdot]}\theta=0.$$

The set of Maurer-Cartan elements in $C(\mathfrak{g},L)$ is denoted $MC(C(\mathfrak{g},L))$. 

\begin{example}\label{ex:dglaendomorphisms}

Every cochain complex $(\partial:\mathbb{V}\to \mathbb{V})$ has an associated dgla $\mathfrak{gl}(\mathbb{V}):=\underline{\mathrm{End}}(\mathbb{V})$  with bracket defined as

$$\left[ S,T\right]=S\circ T-(-1)^{|S||T|}T\circ S,$$

\noindent and differential $\delta:\mathfrak{gl}(\mathbb{V})\to \mathfrak{gl}(\mathbb{V})$ given by

$$\delta(T):=\left[\partial, T\right]=\partial\circ T-(-1)^{|T|}T\circ \partial.$$

\noindent In particular, if $(\mathfrak{g},\lambda)$ is an $L_{\infty}$-algebra, then $C(\mathfrak{g},\mathfrak{gl}(\mathbb{V}))$ is a dgla whose differential is given by

\begin{equation}\label{eq:differentialCgl}
\bf{d}\alpha=[\partial,\alpha]-(-1)^{|\alpha|}\alpha\circ d_{\lambda},
\end{equation}

\noindent for every homogeneous element $\alpha\in C(\mathfrak{g},\mathfrak{gl}(\mathbb{V}))$.

\end{example}

\begin{definition}\label{def:reducedcochain} 
Let $\mathfrak{g}$ be an $L_{\infty}$-algebra and $\mathbb{V}$ a graded vector space. A cochain $\alpha\in C(\mathfrak{g},\mathbb{V})$ is called \textbf{reduced} if $\alpha(1)=0$. The set of reduced cochains is denoted by $\overline{C}(\mathfrak{g},\mathbb{V})$.
\end{definition}

Recall that the symmetric coalgebra is coaugmented with $S(\mathfrak{g}[1])=\mathbb{R}\oplus \overline{S}(\mathfrak{g}[1])$, where

$$\overline{S}(\mathfrak{g}[1])=\oplus_{k\geq 1}S^k(\mathfrak{g}[1]).$$

The natural inclusion $\underline{\mathrm{Hom}}(\overline{S}(\mathfrak{g}[1]),\mathbb{V})\hookrightarrow C(\mathfrak{g},\mathbb{V})$ has image given by the space of reduced cochains. From now on we use the identification $\overline{C}(\mathfrak{g},\mathbb{V})\cong \underline{\mathrm{Hom}}(\overline{S}(\mathfrak{g}[1]),\mathbb{V})$.

Now, if $(\mathbb{V},\partial)$ is a cochain complex, then the dgla $\mathfrak{gl}(\mathbb{V})$ of Example \ref{ex:dglaendomorphisms} induces a dgla structure on $C(\mathfrak{g},\mathfrak{gl}(\mathbb{V}))$ as in Proposition \ref{prop:convolutionDGLA}. One observes that this dgla structure restricts to reduced cochains, yielding a dgla $\overline{C}(\mathfrak{g},\mathfrak{gl}(\mathbb{V}))$. It turns out that Maurer-Cartan elements in the dgla $\overline{C}(\mathfrak{g},\mathfrak{gl}(\mathbb{V}))$ are precisely representations up to homotopy of $\mathfrak{g}$ on $\mathbb{V}$. We proceed now to explain this relation and its consequences.

\subsubsection{Maurer-Cartan elements and ruths}

Let $\mathfrak{g}$ be an $L_{\infty}$-algebra. If $\mathbb{V}$ is a graded vector space, there is an isomorphism $\underline{\mathrm{Hom}}(S(\mathfrak{g}[1])\otimes \mathbb{V},\mathbb{V})\cong \underline{\mathrm{Hom}}(S(\mathfrak{g}[1]),\underline{\mathrm{End}}\mathbb{V})$ and hence a degree one map $\rho:S(\mathfrak{g}[1])\otimes \mathbb{V}\to \mathbb{V}$ is equivalent to having a degree one map $\partial:\mathbb{V}\to \mathbb{V}; v\mapsto \rho(1\otimes v)$,  together with a degree one map $\overline{\rho}:\overline{S}(\mathfrak{g}[1])\to \underline{\mathrm{End}}(\mathbb{V})$. We refer to $\partial\in \underline{\mathrm{End}}(\mathbb{V})_1$ and $\overline{\rho}\in \overline{C}^1(\mathfrak{g},\underline{\mathrm{End}}(\mathbb{V}))$ as the \textbf{components} of $\rho:S(\mathfrak{g}[1])\otimes \mathbb{V}\to \mathbb{V}$ with respect to the natural coagmentation of $S(\mathfrak{g}[1])$.

As shown in \cite{Reinhold}, $\rho:S(\mathfrak{g}[1])\otimes \mathbb{V}\to \mathbb{V}$ is a representation up to homotopy if and only if $\partial^2=0$ and 

\begin{equation}\label{eq:overlinerhoruth}
    \left[\partial,-\right]\circ \bar{\rho}+\bar{\rho}\circ d_{\lambda}+\frac{1}{2}\left[\cdot,\cdot\right]\circ \bar{\rho}\otimes\bar{\rho}\circ\bar{\Delta}_s=0.
\end{equation}

\noindent One immediately observes that \eqref{eq:overlinerhoruth} reads $\Omega_{\overline\rho}=0$ where $\Omega_{\overline\rho}$ is the curvature of $\overline\rho$ as defined in \ref{def:curvaturedgla}. In other words, ruths of $\mathfrak{g}$ on the cochain complex $(\partial:\mathbb{V}\to \mathbb{V})$ are just Maurer-Cartan elements in the dgla $\overline{C}(\mathfrak{g},\mathfrak{gl}(\mathbb{V}))$.

Let us consider now the categories $L_{\infty}$-Alg and DGLA, respectively. Given a cochain complex $(\mathbb{V},\partial)$, there is a well-defined map on objects

\begin{align}\label{eq:functorinftydgla}
L_{\infty}\text{-Alg}&\to \text{DGLA}\\
&\mathfrak{g}\mapsto \overline{C}(\mathfrak{g},\mathfrak{gl}(\mathbb{V})) \nonumber
\end{align}

\noindent We will show that this map is actually a functor. For that, we need some preliminary results about homotopies between $L_{\infty}$-morphisms.

\begin{remark}\label{rmk:quasiinverse}
A well-known result of Kontsevich \cite{Kontsevich} says that given an $L_{\infty}$-quasi-isomorphism $F:\mathfrak{g}\to \mathfrak{h}$ there exists an $L_{\infty}$-quasi-isomorphism $G:\mathfrak{h}\to \mathfrak{g}$ such that the $L_{\infty}$-morphisms $F\circ G:\mathfrak{h}\to \mathfrak{h}$ and $G\circ F:\mathfrak{g}\to \mathfrak{g}$ induce the identity map between the underlying cohomologies $H(\mathfrak{h},\partial_{\mathfrak{h}})$ and $H(\mathfrak{g},\partial_{\mathfrak{g}})$, respectively. We refer to $G:\mathfrak{h}\to \mathfrak{g}$ as an \textbf{$L_{\infty}$-quasi-inverse} of $F:\mathfrak{g}\to \mathfrak{h}$.
\end{remark}

We recall now the notion of homotopy between $L_{\infty}$-morphisms. Let $\mathfrak{g},\mathfrak{h}$ be $L_{\infty}$-algebras with underlying differential graded coalgebras $(S(\mathfrak{g}[1]), d_{\mathfrak{g}})$ and $(S(\mathfrak{h}[1]), d_{\mathfrak{h}})$, respectively. Since there is no risk of confusion, we denote both the coproducts in $S(\mathfrak{g}[1])$ and $S(\mathfrak{h}[1])$ simply by $\Delta_S$.

A \textbf{homotopy} between $L_{\infty}$-morphisms $F,G:S(\mathfrak{g}[1])\to S(\mathfrak{h}[1])$ is a degree -1 map $H:S(\mathfrak{g}[1])\to S(\mathfrak{h}[1])$ which satisfies

\begin{equation}\label{eq:Linftyhomotopy}
F-G=d_{\mathfrak{h}}\circ H+H\circ d_{\mathfrak{g}}.
\end{equation}

\noindent In this case we say that $F$ and $G$ are \textbf{homotopic} and we write $F\sim G$. 

Given $F:\mathfrak{g}\to \mathfrak{h}$ an $L_{\infty}$-quasi-isomorphism, there exists an $L_{\infty}$-quasi-isomorphism $G:\mathfrak{h}\to \mathfrak{g}$ such that $F\circ G\sim \id_{\mathfrak{h}}$ and $G\circ F\sim \id_{\mathfrak{g}}$. This was shown by Kajiura \cite{Kajiura} in the case of $A_{\infty}$-morphisms. However, as mentioned in \cite{Kajiura}, the very same proof and techniques apply to the setting of $L_{\infty}$-morphisms. Kajiura's proof relies on the minimal model of an $A_{\infty}$-algebra, whose analogue for an $L_{\infty}$-algebra is given by Kontsevich decomposition of an $L_{\infty}$-algebra into minimal and contractible components \cite{Kontsevich}.

\begin{remark}\label{rmk:homotopiesquasiinverse}
In what follows we use that any $L_{\infty}$-quasi-isomorphism $F:\mathfrak{g}\to \mathfrak{h}$ has a quasi-inverse $G:\mathfrak{h}\to \mathfrak{g}$ and there exist homotopies

\begin{equation}\label{eq:homotopiesquasiinverse}
G\circ F-\id_{\mathfrak{g}}=d_{\mathfrak{g}}\circ H+H\circ d_{\mathfrak{g}};\quad \text{and}\quad F\circ G-\id_{\mathfrak{h}}=d_{\mathfrak{h}}\circ H' +H' \circ d_{\mathfrak{h}},
\end{equation}

\end{remark}

We are ready to show that \eqref{eq:functorinftydgla} is actually a functor.

\begin{theorem}\label{thm:quasiandruth}
Let $\mathfrak{g}$ and $\mathfrak{h}$ be two $L_{\infty}$-algebras. An $L_{\infty}$-morphism $F:\mathfrak{g}\to \mathfrak{h}$ induces a DGLA-morphism 
\[F^*:\overline{C}(\mathfrak{h},\mathfrak{gl}(\mathbb{V}))\to \overline{C}(\mathfrak{g},\mathfrak{gl}(\mathbb{V})),\quad F^{*}\alpha:=\alpha\circ F.
\]
If $F$ is an $L_{\infty}$-quasi-isomorphism then $F^{*}$ is a quasi-isomorphism of dgla's.
\end{theorem}
\begin{proof}

Let $\alpha,\beta \in \overline{C}(\mathfrak{h},\mathfrak{gl}(\mathbb{V}))$ be two homogeneous elements. Then 
\begin{align*}
    F^*[\alpha,\beta]=F^*(\alpha\wedge_{\left[\cdot,\cdot\right]}\beta)=&\left[\cdot,\cdot\right]\circ\alpha\otimes\beta \circ \Delta_S\circ F\\
    =&\left[\cdot,\cdot\right]\circ\alpha\otimes\beta \circ F\otimes F\circ\Delta_S\\
    =&\left[\cdot,\cdot\right]\circ(\alpha\circ F)\otimes (\beta\circ F)\circ \Delta_S\\
    =&F^*\alpha\wedge_{\left[\cdot,\cdot\right]}F^*\beta\\
    =& [F^*\alpha,F^*\beta]
\end{align*}

\noindent showing that $F^*$ preserves graded Lie brackets. Regarding the compatibility with the differentials, it follows from \eqref{eq:differentialCgl}  that:

\begin{align*}
    F^*(\textbf{d}\alpha)=&(\left[\partial,\cdot\right]\circ\alpha-(-1)^{|\alpha|}\alpha\circ d_{\lambda^{\mathfrak{h}}})\circ F\\
     =&\left[\partial,\cdot\right]\circ (\alpha\circ F)-(-1)^{|\alpha|}(\alpha\circ F)\circ d_{\lambda^{\mathfrak{g}}}\\
       =&\left[\partial,\cdot\right]\circ F^*\alpha-(-1)^{|\alpha|} F^*\alpha \circ d_{\lambda^{\mathfrak{g}}}\\
       =&\textbf{d}(F^*\alpha).
\end{align*}

\noindent

Hence, $F^*$ is a morphism of dgla's. Assume now that $F$ is an $L_{\infty}$-quasi-isomorphism and consider an $L_{\infty}$-quasi-inverse $G: \mathfrak{h}\to \mathfrak{g}$ together with homotopies as in \eqref{eq:homotopiesquasiinverse}. For every homogeneous element $\alpha\in \overline{C}(\mathfrak{h},\mathfrak{gl}(\mathbb{V}))$ one has
\begin{align*}
    (F\circ G-\id_{\mathfrak{h}})^*(\alpha) =&\alpha\circ (F\circ G-\id_{\mathfrak{h}})\\
    =&\alpha\circ (d_{\mathfrak{h}}\circ H' +H' \circ d_{\mathfrak{h}})\\
    =&-(-1)^{|\alpha|}\left[\partial,\cdot\right]\circ \alpha \circ H'+\alpha\circ d_{\mathfrak{h}}\circ H'+(-1)^{|\alpha|}\left[\partial,\cdot\right]\circ \alpha\circ H'+\alpha\circ H'\circ d_{\mathfrak{h}}\\
    =&-(-1)^{|\alpha|}H'^*(\textbf{d}\alpha)+(-1)^{|\alpha|}(\left[\partial,\cdot\right]\circ H'^*\alpha-(-1)^{|\alpha|-1}H'^*\alpha \circ d_{\mathfrak{h}})\\
    =&(-1)^{|\alpha|}(\textbf{d}(H'^*\alpha)-H'^*(\textbf{d}\alpha)).
\end{align*}

\noindent As a consequence,  in cohomology the map $G^*:\overline{C}(\mathfrak{g},\mathfrak{gl}(\mathbb{V}))\to \overline{C}(\mathfrak{h},\mathfrak{gl}(\mathbb{V}))$ yields a left inverse of $F^*$. Similarly, $G^*$ is a right inverse of $F^*$ in cohomology. This shows that $F^*$ is an isomorphism in cohomology, hence a quasi-isomorphism of dgla's.
\end{proof}

\subsection{Invariance of the $L_{\infty}$-algebra cohomology}

In this subsection we show that the $L_{\infty}$-algebra cohomology with coefficients in a representation up to homotopy is invariant under quasi-isomorphisms of ruths. 

\subsubsection{Reduced complex}

Let $\mathfrak{g}$ be an $L_{\infty}$-algebra and $\mathbb{V}$ a graded vector space.  Let $\overline{C}(\mathfrak{g},\mathbb{V})$ denote the set of reduced cochains.

\begin{lemma}\label{lemma:Drhosplitting}

Let $\rho:S(\mathfrak{g}[1])\otimes \mathbb{V}\to \mathbb{V}$ be a degree one map with associated components $\partial:\mathbb{V}\to \mathbb{V}$ and $\overline{\rho}:\overline{S}(\mathfrak{g}[1])\to \underline{\mathrm{End}}(\mathbb{V})$. If $\alpha\in \overline{C}(\mathfrak{g},\mathbb{V})$ then $D_{\rho}\alpha\in \overline{C}(\mathfrak{g},\mathbb{V})$. Also, the codifferential $D_{\rho}$ decomposes as

$$D_{\rho}\alpha=\overline{\rho}\wedge_{ev}\alpha+\partial\circ\alpha-(-1)^{|\alpha|}\alpha\circ d_{\lambda}.$$
\end{lemma}

\begin{proof}
Note that
\begin{align*}
    D_{\rho}\alpha(1)=&\rho\circ (\id_{S(\mathfrak{g}[1])}\otimes \alpha)\circ \Delta_S(1)-(-1)^{|\alpha|}\alpha \circ d_{\lambda}(1)\\
    =&\rho\circ (\id_{S(\mathfrak{g}[1])}\otimes \alpha)(1\otimes 1)-0\\
    =&\rho(1\otimes \alpha(1))\\
    =&\partial(\alpha(1))\\
    =&0.
\end{align*}

Consider now an element $x\in \overline{S}(\mathfrak{g}[1])$, then
\begin{align*}
    D_{\rho}\alpha(x)=&\rho(\id_{S(\mathfrak{g}[1])}\otimes \alpha)(\overline{\Delta}(x)+1\otimes x+x\otimes 1)-(-1)^{|\alpha|}\alpha(d_{\lambda}(x))\\
    =&\rho(\id_{S(\mathfrak{g}[1])}\otimes \alpha)\overline{\Delta}(x)+\rho(\id_{S(\mathfrak{g}[1])}\otimes \alpha)(1\otimes x)+\rho(\id_{S(\mathfrak{g}[1])}\otimes \alpha)(x\otimes 1)-(-1)^{|\alpha|}\alpha (d_{\lambda}(x))\\
    =&\overline{\rho}(\id_{S(\mathfrak{g}[1])}\otimes\alpha)\overline{\Delta}(x)+\partial(\alpha(x))+(-1)^{|\alpha||x|}\overline{\rho}_{x}(\alpha(1))-(-1)^{|\alpha|}\alpha (d_{\lambda}(x))\\
\end{align*}
and  since $\overline{\rho}(1\otimes\alpha)\overline{\Delta}=\overline{\rho}\wedge_{ev}\alpha$, one concludes 
\[D_{\rho}\alpha=\overline{\rho}\wedge_{ev}\alpha+\partial\circ \alpha-(-1)^{|\alpha|}\alpha\circ d_{\lambda}.\]
\end{proof}

Due to Lemma \ref{lemma:Drhosplitting} the operator $D_{\rho}:C(\mathfrak{g},\mathbb{V})\to C(\mathfrak{g},\mathbb{V})$ restricts to an operator $D_{\overline\rho}:\overline{C}(\mathfrak{g},\mathbb{V})\to \overline{C}(\mathfrak{g},\mathbb{V})$ on reduced cochains and

\begin{equation}\label{eq:reducedDrho}
D_{\overline{\rho}}\alpha=\overline{\rho}\wedge_{ev}\alpha+\partial\circ \alpha-(-1)^{|\alpha|}\alpha\circ d_{\lambda}.
\end{equation}

If $\rho:S(\mathfrak{g}[1])\otimes \mathbb{V}\to \mathbb{V}$ is not a representation up to homotopy, then $\overline\rho\in C^1(\mathfrak{g},\underline{\mathrm{End}}(\mathbb{V}))$ is not a Maurer-Cartan element. However, one has the following result which is the analogue of the classical expression for the square of the covariant derivative of a linear connection.

\begin{proposition}\label{prop:DsquareMC}

Let $\alpha\in \overline{C}(\mathfrak{g},\mathbb{V})$ be a reduced cochain. Then the following holds

$$D_{\overline{\rho}}^2(\alpha)=\Omega_{\overline\rho}\wedge_{ev}(\alpha).$$
\end{proposition}

\begin{proof}
Let $\alpha\in \overline{C}(\mathfrak{g},\mathbb{V})$ be a homogeneous element. It follows from \eqref{eq:reducedDrho} that

\begin{align*}
D^2_{\overline{\rho}}\alpha=&\overline{\rho}\wedge_{ev}(D_{\overline{\rho}}\alpha)+\partial\circ(D_{\overline{\rho}}\alpha)-(-1)^{|\alpha|+1}D_{\overline{\rho}}\alpha \circ d_{\lambda}\\
=&\overline{\rho}\wedge_{ev}(\overline{\rho}\wedge_{ev}\alpha)+\overline{\rho}\wedge_{ev}(\partial\circ \alpha)+\overline{\rho}\wedge_{ev}(-(-1)^{|\alpha|}\alpha\circ d_{\lambda})+\partial\circ (\overline{\rho}\wedge_{ev}\alpha)+ \partial\circ\partial \circ \alpha+\\
&-(-1)^{|\alpha|}\partial\circ \alpha\circ d_{\lambda}-(-1)^{|\alpha|+1}\overline{\rho}\wedge_{ev}\alpha \circ d_{\lambda}-(-1)^{|\alpha
|+1}\partial\circ \alpha\circ d_{\lambda}-(-1)^{|\alpha|+1}((-1)^{|\alpha|+1}\alpha\circ d_{\lambda}\circ d_{\lambda})\\
=& A+B+C\\
\end{align*}

\noindent where $A=\overline{\rho}\wedge_{ev}(\overline{\rho}\wedge_{ev}\alpha)$, $B=\overline{\rho}\wedge_{ev}(\partial\circ \alpha)+\partial\circ(\overline{\rho}\wedge_{ev}\alpha)$ and $C=-(-1)^{|\alpha|}\overline{\rho}\wedge_{ev}(\alpha\circ d_{\lambda})+(-1)^{|\alpha|}(\overline{\rho}\wedge_{ev}\alpha)\circ d_{\lambda}$. Here we have used the identities $\partial^2=0$ and $d^2_{\lambda}=0$.

A simple computation yields:

\begin{align*}
 A=&\frac{1}{2}(\overline{\rho}\wedge_{\left[\cdot,\cdot\right]}\overline{\rho})\wedge_{ev}\alpha\\
    B=&\left(\left[\partial,\cdot\right]\circ\overline{\rho}\right)\wedge_{ev}\alpha\\
    C=& (\overline{\rho} \circ d_{\lambda})\wedge_{ev}\alpha
\end{align*}

\noindent which together imply

\begin{align*}
    D_{\overline{\rho}}^2\alpha=&(\frac{1}{2}\overline{\rho}\wedge_{\left[\cdot,\cdot\right]}\overline{\rho})\wedge_{ev}\alpha+\left(\left[\partial,\cdot\right]\circ \overline{\rho}\right)\wedge_{ev}\alpha+(\overline{\rho}\circ d)\wedge_{ev}\alpha\\
    =&\left(\frac{1}{2}\overline{\rho}\wedge_{\left[\cdot,\cdot\right]}\overline{\rho}+\left[\partial,\cdot\right]\circ \overline{\rho} +\overline{\rho}\circ d \right)\wedge_{ev}\alpha\\
    =& \Omega_{\overline{\rho}}\wedge_{ev}\alpha
\end{align*}

\noindent as desired.

\end{proof}

In particular, a representation up to homotopy of $\mathfrak{g}$ on a cochain complex $(\mathbb{V},\partial)$ yields a cochain complex $(\overline{C}(\mathfrak{g},\mathbb{V}),D_{\overline{\rho}})$ called the \textbf{reduced complex}.

%\[\text{\textcolor{blue}{New Definition}}\]
%\begin{definition}
%For a $L_{\infty}$-algebra $\mathfrak{g}$ and a representation up to homotopy $\overline{\rho}:\overline{Sym}(\mathfrak{g}\left[1\right])\to \mathbb{V}$ on the dg-vector space $(\mathbb{V},\partial)$ we called the \textbf{$L_{\infty}$ reduced cohomology of $\mathfrak{g}$ with values in $\mathbb{V}$} the cohomology of the subcomplex 
%\[\left(\Hom(\overline{Sym}(\mathfrak{g}\left[1\right]),\mathbb{V}),D_{\overline{\rho}}\right).\]
%We denoted this complex by $C_{red}(\mathfrak{g},\mathbb{V})$.
%\end{definition}

\subsubsection{The Canonical Spectral Sequence}

Let $\mathfrak{g}$ be an $L_{\infty}$ algebra. The symmetric coalgebra $S(\mathfrak{g}[1])$ comes equipped with its primitive filtration

$$S_{\left[0\right]}(\mathfrak{g}\left[1\right])=0,\quad S_{\left[1\right]}(\mathfrak{g}\left[1\right])=\mathbb{R},\quad S_{\left[p\right]}(\mathfrak{g}\left[1\right])=\bigoplus_{n<p}S^n(\mathfrak{g}\left[1\right]),$$

\noindent for every $p>1$, which is clearly ascending and bounded below. Also, It is well-known that the primitive filtration is exhaustive and Hausdorff, that is

$$\bigcup_{n=0}^{\infty}S_{\left[n\right]}(\mathfrak{g}\left[1\right])=S(\mathfrak{g}\left[1\right]),\quad \bigcap_{n=0}^{\infty}S_{\left[n\right]}(\mathfrak{g}\left[1\right])=0,$$

\noindent respectively. It is easy to see that the degree one coderivation $d:S(\mathfrak{g}[1])\to S(\mathfrak{g}[1])$ induced by the $L_{\infty}$-structure preserves the filtration. Additionally, if $(\mathbb{V},\rho)\in \mathrm{Rep}_{\infty}(\mathfrak{g})$ is a representation up to homotopy, then the cochain complex $(C(\mathfrak{g},\mathbb{V}),D_{\rho})$ inherits a descending filtration defined by

$$ F^{p}C(\mathfrak{g},\mathbb{V}):=\left\lbrace \alpha:S(\mathfrak{g}\left[1\right])\to \mathbb{V}\,|\, \alpha|_{S_{\left[p\right]}(\mathfrak{g}\left[1\right])}=0\right\rbrace,$$

\noindent for every $p\geq 0$. As a consequence of Lemma \ref{lemma:Drhosplitting}, one immediately observes that

$$D_{\rho}(F^pC^{\bullet}(\mathfrak{g},\mathbb{V}))\subseteq F^pC^{\bullet+1}(\mathfrak{g},\mathbb{V}).$$

\noindent It is clear that  $F^0C(\mathfrak{g},\mathbb{V})=C(\mathfrak{g},\mathbb{V})$ is given by the whole complex while  $F^1C(\mathfrak{g},\mathbb{V})=\overline{C}(\mathfrak{g},\mathbb{V})$ coincides with the complex of reduced cochains. In general, one has

$$\cdots F^{p+1}C(\mathfrak{g},\mathbb{V})\subseteq F^pC(\mathfrak{g},\mathbb{V})\subseteq\cdots\subseteq \overline{C}(\mathfrak{g},\mathbb{V})\subseteq C(\mathfrak{g},\mathbb{V}),$$

\noindent so the filtration is exhaustive. Also, one can see that this filtration is complete since

 $$C(\mathfrak{g},\mathbb{V})/F^nC(\mathfrak{g},\mathbb{V})\simeq \Hom(S_{[n]}(\mathfrak{g}\left[1\right]),\mathbb{V}),$$
 
\noindent which implies

$$\lim_{\longleftarrow}\Hom(S_{[n]}(\mathfrak{g}\left[1\right]),\mathbb{V})\simeq\Hom(S(\mathfrak{g}\left[1\right]),\mathbb{V})=C(\mathfrak{g},\mathbb{V}).$$

As a consequence, we have the following result about the spectral sequence associated to this filtration.

\begin{proposition}

The spectral sequence associated to the filtration $F_{\bullet}C(\mathfrak{g},\mathbb{V})$ weakly converges and 

$$E^{p,q}_1(\mathfrak{g};\mathbb{V})\Rightarrow H^{p+q}(\mathfrak{g};\mathbb{V}),$$
for every $p,q>0$.
\end{proposition}
\begin{proof}

The descending filtration above is exhaustive and complete then the result follows by Theorem 10 Section 5.5 in \cite{Weibel}. 

\end{proof}

The spectral sequence associated to the filtration $F_{\bullet}C(\mathfrak{g},\mathbb{V})$ of the complex $C(\mathfrak{g},\mathbb{V})$ starts with

$$E^{p,q}_0=F^pC^{p+q}(\mathfrak{g},\mathbb{V})/F^{p+1}C^{p+q}(\mathfrak{g},\mathbb{V})\simeq \Hom^{p+q}(S^p(\mathfrak{g}\left[1\right]),\mathbb{V})$$

\noindent and $d_0:E^{p,q}_0\to E^{p,q+1}_0$ is given by $\bar{D_{\rho}}$, for $\alpha \in \Hom^{p+q}(S^p(\mathfrak{g}\left[1\right]),\mathbb{V})$ we have \begin{equation}\label{d_0spectralsequence}
    d_0\alpha=\partial\alpha-(-1)^{p+q}\alpha {d_1^1}^{\otimes},
\end{equation} where ${d_1^1}^{\otimes}$ is the coderivation induced by $\lambda_1$. Explicitly for $x_i\in \mathfrak{
g}\left[1\right], 1\leq i\leq p$ holds that 
\begin{align*}
    d_{0}\alpha(x_1\vee\cdots\vee x_p)=&\partial(\alpha(x_1\vee\cdots\vee x_p))-\sum_{\sigma \in Sh(1,p-1)}(-1)^{p+q}\epsilon(\sigma)\alpha(\lambda_1(x_{\sigma(1)})\vee x_{\sigma(2)}\cdots\vee x_{\sigma(p)}).
\end{align*}

We proceed now to show that the $L_{\infty}$-cohomology with coefficients in a ruth is invariant by ruth quasi-isomorphisms.

\begin{definition}\label{def:ruthquasiiso}
Let $(\mathbb{V},\rho)$ and $(\mathbb{V}',\rho')$ be representations up to homotopy of $\mathfrak{g}$ and $\mathfrak{g}'$, respectively. A morphism $(F,f):(\mathbb{V},\rho)\to (\mathbb{V}',\rho')$ of representations up to homotopy is called a \textbf{quasi-isomorphism} if $F:\mathfrak{g}\to \mathfrak{g}'$ is an $L_{\infty}$-quasi-isomorphism and $f:\mathbb{V}'\to \mathbb{V}$ is a quasi-isomorphism between the underlying cochain complexes. 
\end{definition}

Let $F,G:\mathfrak{g}\to \mathfrak{h}$ be $L_{\infty}$-morphisms. A coderivation \textbf{along F and G} is a homogeneous map $D:S(\mathfrak{g}[1])\to S(\mathfrak{h}[1])$ satisfying

\begin{equation}\label{eq:FGcoderivation}
\Delta_S\circ D=(F\otimes D+D\otimes G)\circ \Delta_S,
\end{equation}

\noindent where by abuse of notation we denote $\Delta_S$ the coproduct in $S(\mathfrak{g}[1])$ and $S(\mathfrak{h}[1])$. The \textbf{iterated coproduct} is the map $\Delta^n_S:S(\mathfrak{g}[1])\to S(\mathfrak{g}[1])^{\otimes n}$ given by $\Delta^n_S=(\Delta_S\otimes \id^{\otimes (n-1)})\circ \Delta_S$ for $n\geq 1$ and $\Delta^0_S:=\id$. A straightforward computation shows that if $D:S(\mathfrak{g}[1])\to S(\mathfrak{h}[1])$ is a coderivation along $F$ and $G$, then the following holds:

\begin{equation}\label{eq:iteratedcoderivation}
\Delta^n_S\circ D=\big{(}\sum_{i+j=n}F^{\otimes i}\otimes D \otimes G^{\otimes j}\big{)}\circ \Delta^n_S.
\end{equation}

Let $F:\mathfrak{g}\to \mathfrak{h}$ be an $L_{\infty}$-quasi-isomorphism with a quasi-inverse $G:\mathfrak{h}\to \mathfrak{g}$ \cite{Kontsevich}. As observed in Remark \ref{rmk:homotopiesquasiinverse} there are homotopies $H:S(\mathfrak{g}[1])\to S(\mathfrak{g}[1])$ and $H':S(\mathfrak{h}[1])\to S(\mathfrak{h}[1])$ satisfying equations \eqref{eq:homotopiesquasiinverse}, see \cite{Kajiura}. The construction of the previous homotopies shows that $H$ is actually a coderivation along $G\circ F$ and $\id_{\mathfrak{g}}$. In particular, it follows from \eqref{eq:iteratedcoderivation} that $H$ preserves the primitive filtration of $S(\mathfrak{g}[1])$. A similar result holds for $H'$.

Now we are ready to state the main result of this section.

\begin{theorem}\label{CEquasiiso}

Let $(\mathbb{V},\rho)$ and $(\mathbb{W},\tau)$ be representations up to homotopy of $\mathfrak{g}$ and $\mathfrak{h}$, respectively. If  $(F,f):(\mathbb{V},\rho)\to (\mathbb{W},\tau)$ is a ruth quasi-isomorphism, then

  \begin{center}
    \begin{tikzcd}
       F^*:C(\mathfrak{h},\mathbb{W}) \arrow[r]&C(\mathfrak{g},\mathbb{V})
    \end{tikzcd}
\end{center}
is a quasi-isomorphism.

\end{theorem}

\begin{proof}

Note that as $F$ is a coalgebra morphism it preserves the primitive filtration on $S(\mathfrak{g}\left[1\right])$ and $S(\mathfrak{h}\left[1\right])$ then $F^*$ preserves the filtration on the cochain complexes $C_{\rho'}(\mathfrak{h},\mathbb{W})$ and $C_{\rho}(\mathfrak{g},\mathbb{V})$. Thus, $F^*$ induces a morphism between spectral sequences. Now we shall see that \[F^*_1:E_{1}^{p,q}(\mathfrak{h};\mathbb{W})\to E_{1}^{p,q}(\mathfrak{g};\mathbb{V})\] is a isomorphism. For this we take a $L_{\infty}$-homotopy inverse $T$ for $F$, it exists by \cite[Thm.7.5]{Kajiura}. Note that by equation \eqref{eq:iteratedcoderivation} a coderivation preserves the primitive filtration, then in particular for a homotopy coderivation $H$ of $S(\mathfrak{h}\left[1\right])$ the equation  \[F\circ T-\id=Hd+dH\]
holds on $S^p(\mathfrak{h}\left[1\right])$ for all $p\geq 1$. On the other hand, one has that for $p,q >0$ the map $F_0^*:E_{0}^{p,q}(\mathfrak{h};\mathbb{W})\to E_{0}^{p,q}(\mathfrak{g};\mathbb{V})$ depends only on $F_1^1$ and $f$. Indeed, given $x_i\in \mathfrak{g}\left[1\right], 1\leq i \leq p$ and $\alpha\in E_{0}^{p,q}(\mathfrak{h};\mathbb{W})$, the following holds

$$F_0^*\alpha(x_1\vee \cdots\vee x_p)=f\alpha(F_1^1(x_1)\vee F_1^1(x_2)\vee\cdots\vee F_1^1(x_p)).$$

\noindent Hence, if $t$ is a homotopy inverse for $f$, then the map $(T,t):\mathfrak{h}\to \mathfrak{g}$ although it is not a ruths morphism, it yields a well-defined map 

$$T_0^*:E_{0}^{p,q}(\mathfrak{g};\mathbb{V})\to E_{0}^{p,q}(\mathfrak{h};\mathbb{W}),$$

\noindent for every $p,q >0$. Now it is straightforward verify that $T_0^*$ is a homotopy left inverse for $F_0^*,$
\begin{align*}
    T_0^*F_0^*\alpha=&(\id+h\partial+\partial h)\alpha(\id +H{d_1^1}^{\otimes}+{d_1^1}^{\otimes}H)\\
    =&\alpha+ d_0\left(h\alpha+h\alpha H{d_1^1}^{\otimes}+(-1)^{|\alpha|}(\alpha H+h\partial \alpha H)\right)
\end{align*}

\noindent In a similar way we can see that $T_0^*$ is a homotopy right inverse of $F_0^*$. Hence, $F_1^*$ is an isomorphism, and therefore by the Eilenberg-Moore Comparison Theorem \cite[Thm 5.5.11]{Weibel} we conclude that 

$$H(F^*):H(\mathfrak{h},\mathbb{W})\to H(\mathfrak{g},\mathbb{V})$$

\noindent is an isomorphism.
\end{proof}

The $L_{\infty}$-algebra cohomology is an invariant in the derived category of representations up to homotopy of $L_{\infty}$-algebras. This is the content of the following result, whose proof is straightforward from Theorem \ref{CEquasiiso}.

\begin{corollary}\label{cor:moritainvarianceLinftycohomology}

Let $\begin{tikzcd}\mathfrak{g} &\arrow[l,"F"']\tilde{\mathfrak{h}}\arrow[r,"G"]&\mathfrak{g}'\end{tikzcd}$ be quasi-isomorphisms of $L_{\infty}$-algebras. Let $(\mathbb{V},\rho)$ and $(\mathbb{V}',\rho')$ be ruths of $\mathfrak{g}$ and $\mathfrak{g}'$, respectively. If $g:\mathbb{V}'\to \mathbb{V}$ is a quasi-isomorphism of complexes such that $(G,g):(\mathbb{V},F^*\rho)\to (\mathbb{V}',\rho')$ is a quasi-isomorphism of ruths, then

$$H(\mathfrak{g},\mathbb{V})\cong H(\mathfrak{g}',\mathbb{V}').$$

\end{corollary}

%\[\text{\textcolor{blue}{New Theorem}}\]
%\begin{theorem}
%\[E_1^{p,q}(\mathfrak{g};\mathbb{V})\simeq \Hom^{p+q}(Sym^{p}(H(\mathfrak{g})\left[1\right]),H(\mathbb{V}))\]
%\end{theorem}
%\begin{proof}
%Let consider the minimal model of $\mathfrak{g}$, that is, a $L_{\infty}$-structure on $H(\mathfrak{g}):=H(\mathfrak{g},\left[\cdot\right]^1)$ together with a $L_{\infty}$-quasi-isomorphism $q:H(\mathfrak{g})\to \mathfrak{g}$. Then $\rho^*:=q^*\rho$ determines a representation up to homotopy of $H(\mathfrak{g})$ on $\mathbb{V}$ such that 
%\[(q,\id):H(\mathfrak{g})\to \mathfrak{g}\]
%is a $(\rho^*,\rho)$-equivariant map, thus by the before Theorem \ref{CEquasiiso} one has  \[E_{1}^{p,q}(\mathfrak{g};\mathbb{V})\simeq E_{1}^{p,q}(H(\mathfrak{g});\mathbb{V})\] and as \[ E_{1}^{p,q}(H(\mathfrak{g});\mathbb{V})=H(E_{0}^{p,q}(H(\mathfrak{g});\mathbb{V}),d_0)\]
%where $d_0$ is induced by $\partial$, by the equation (\ref{d_0spectralsequence}). Hence, 
%\begin{align*}
%    H(E_{0}^{p,q}(H(\mathfrak{g});\mathbb{V}),d_0)\simeq& H(\Hom^{p+q}(Sym^p(H(\mathfrak{g})\left[1\right]),\mathbb{V}),d_0)\\
%\simeq& \Hom^{p+q}(Sym^p(H(\mathfrak{g})\left[1\right]),H(\mathbb{V})).
%\end{align*}
%\end{proof}

%%%%%%%%%%%%%%%%%%%%%%%%%%%%%%%
%%%%%%%%%%%%%%%%%%%%%%%%%%%%%%%
%%%%%%%%%%%%%%%%%%%%%%%%%%%%%%%

\section{Extensions of $L_{\infty}$-algebras and the Chern-Weil-Lecomte map}\label{sec:CWL}

\subsection{Extensions of $L_{\infty}$-algebras}

We start by recalling the notion of extension in the category of $L_{\infty}$-algebras. For more details see  \cite{MehtaZambon}, \cite{ChuangLazarev}, and \cite{Reinhold}. 

\begin{definition}
An \textbf{extension of $L_{\infty}$-algebras} is a short exact sequence of $L_{\infty}$-algebras
\begin{center}
    \begin{tikzcd}\label{extensionLinfinity}
    0\arrow[r]&\mathfrak{n}\arrow[r,"\iota"]&\hat{\mathfrak{g}}\arrow[r,"\pi"]&\mathfrak{g}\arrow[r]&0
    \end{tikzcd}
\end{center}
in which $\pi$ and $\iota$ are strict $L_{\infty}$-morphisms. In this case we say that $\hat{\mathfrak{g}}$ is an extension of $\mathfrak{g}$ by $\mathfrak{n}$.
\end{definition}

As $\pi$ is a strict $L_{\infty}$-morphism it is determined by its projection $\pi_1^1$ onto $\mathfrak{g}[1]$. More precisely $\pi=S(\pi_1^1)$ where the right hand side of the equality is the symmetric extension of $\pi_1^1$ defined in \eqref{eq:symmetricextensionmap}. As a consequence, a linear section $h_1^1:\mathfrak{g}\left[1\right]\to \hat{\mathfrak{g}}\left[1\right]$ of $\pi_1^1$ induces a coalgebra morphism $h=S(h_1^1)$ which is clearly a linear section of $\pi$. 

In this section, $\mu:S(\mathfrak{g}[1])\times S(\mathfrak{g}[1])\to S(\mathfrak{g}[1])$ denotes the multiplication map in the symmetric algebra of $\mathfrak{g}[1]$ and $\lambda:S(\mathfrak{g}[1])\to \mathfrak{g}[1]$ the linear map induced by the symmetric brackets $\lambda_k:S^k(\mathfrak{g}[1])\to \mathfrak{g}[1]$ defining the $L_{\infty}$-structure of $\mathfrak{g}$, see \eqref{eq:symmetricJacobi}. Similarly, the corresponding maps for $\hat{\mathfrak{g}}$ are denoted by  $\hat{\mu}$, $\hat{\lambda}$ and $\hat{\lambda_k}$, respectively.

\begin{proposition}\label{prop:curvature} Let $h:\mathfrak{g}\to \hat{\mathfrak{g}}$ be a section of an extension of $L_{\infty}$-algebras \eqref{extensionLinfinity}.
For every $k\geq 1$ the linear map
 
$$K_{h}^k:=\hat{\lambda}_k\circ h^{\otimes k}-h\circ \lambda_k:\mathfrak{g}\left[1\right]^{\otimes k}\to \hat{\mathfrak{g}}\left[1\right]$$
has the following properties:
\begin{enumerate}
    \item $K_h^k$ is symmetric of degree 1, and 
    \item $\mathrm{Im}(K_h^k)\subseteq \mathfrak{n}\left[1\right]$.
\end{enumerate}

\end{proposition}
\begin{proof}
To prove \textit{(1)} let $x_i\in \mathfrak{g}, k \geq 1$ be homogeneous elements. Then 
\begin{align*}
    \hat{\lambda}_k\circ h^{\otimes k}(x_1\otimes \cdots\otimes x_k)=&\hat{\lambda}_k(h(x_1)\otimes\cdots\otimes h(x_k))\\
    =&\epsilon(\sigma;h(x_1),\dots,h(x_n))\hat{\lambda}_{k}(h(x_{\sigma(1)}),\otimes\cdots\otimes h(x_{\sigma(k)}))\\
    =&\epsilon(\sigma;x_1,\dots,x_n)(\hat{\lambda}_k\circ h^{\otimes k})(x_{\sigma(1)}\otimes \cdots\otimes x_{\sigma(k)}).
\end{align*}
The last identity follows from the fact that $h$ is a degree-preserving map, hence $\epsilon(\sigma;h(x_1),\dots,h(x_n))=\epsilon(\sigma;x_1,\dots,x_n)$. Then
\begin{align*}
    K_h^k(x_1\otimes \cdots\otimes x_n)=&\epsilon(\sigma;x_1,\dots,x_n)(\hat{\lambda}_k\circ h^{\otimes k})(x_{\sigma(1)}\otimes \cdots\otimes x_{\sigma(k)})\\
    &-h(\epsilon(\sigma;x_1,\cdots,x_n)\lambda_k(x_{\sigma(1)}\otimes \cdots\otimes x_{\sigma(k)}))\\
    =&\epsilon(\sigma;x_1,\cdots,x_n)(\hat{\lambda}_k\circ h^{\otimes k}-h\circ \lambda_k)(x_{\sigma(1)}\otimes\cdots\otimes x_{\sigma(k)})\\
    =&\epsilon(\sigma;x_1,\cdots,x_n)K_h^k(x_{\sigma(1)}\otimes\cdots\otimes x_{\sigma(k)}).
\end{align*}
which shows that $K_h^k$ is symmetric. As both $\lambda_k$ and $\hat{\lambda}_k$ have degree one and $h$ has degree zero, then $K_h^k$ has degree 1.\\
In order to show \textit{(2)}, note that since $\pi:\hat{\mathfrak{g}}\to \mathfrak{g}$ is a strict $L_{\infty}$-algebra morphism, then $\pi_1^1\circ \hat{\lambda}_k=\lambda_k\circ(\pi_1^1)^{\otimes k}$. Hence
\begin{align*}
    \pi_1^1(K_h^k)=&\pi_1^1\circ\hat{\lambda}_k\circ h^{\otimes k}-\pi_1^1\circ h\circ\lambda_k\\
    =&\lambda_k\circ (\pi_1^1)^{\otimes k}\circ h^{\otimes k}-\pi_1^1\circ h\circ\lambda_k\\
    =&\lambda_k\circ(\pi_1^1\circ h)^{\otimes k}-(\pi_1^1\circ h)\circ \lambda_k\\
    =&\lambda_k-\lambda_k\\
    =&0,
\end{align*}
where we have used the interchange law \eqref{eq:interchangelaw} and the fact that $\pi_1^1\circ h=\id$.
Then, by exactness of the sequence, one has 
$$\mathrm{Im}(K_h^k)\subseteq \text{Ker}(\pi_1^1)=\text{Im}(\iota)=\mathfrak{n}\left[1\right],$$
as desired.
\end{proof}

As a consequence of Proposition \ref{prop:curvature}, the family of maps $K_h^k:S^k(\mathfrak{g}\left[1 \right])\to \mathfrak{n}\left[1\right]; k\geq 1$ defines a degree one linear map

\begin{equation}\label{eq:curvature}
K_h:=\sum_{k\geq 1}K_h^k:S(\mathfrak{g}\left[1\right])\to \mathfrak{n}\left[1\right].
\end{equation}

\begin{definition}\label{def:curvature}

Let $h:\mathfrak{g}[1]\to \hat{\mathfrak{g}}[1]$ be a section of an extension of $L_{\infty}$-algebras as in \ref{extensionLinfinity}. The element $K_h\in C^1(\mathfrak{g},\mathfrak{n}[1])$ is called the \textbf{curvature} of $h$. If $K_{h}=0$, then section $h$ is called \textbf{flat}.

\end{definition}

\begin{remark}
The curvature $K_h\in C^1(\mathfrak{g},\mathfrak{n}[1])$ measures the failure of $h:\mathfrak{g}[1]\to \hat{\mathfrak{g}}[1]$ being an $L_{\infty}$-algebra morphism, since $K_h=\hat{\lambda}\circ S(h)-h\circ\lambda$. In other words, a section $h:\mathfrak{g}[1]\to \hat{\mathfrak{g}}[1]$ is flat if and only if it is an $L_{\infty}$-algebra morphism.
\end{remark}

Let $(\mathfrak{g}\left[1\right],\lambda)$ be an $L_{\infty}$-algebra. A graded subspace $\mathfrak{h}\left[1\right]\subseteq \mathfrak{g}[1]$ is an \textbf{ideal} if it is invariant under the adjoint representation of $\mathfrak{g}$, see Example \ref{ex:adjointruth}. That is,

$$\lambda(x\vee y) \in \mathfrak{h}\left[1\right],$$

\noindent for every $x\in \mathfrak{h}\left[1\right]$ and $y\in S(\mathfrak{g}\left[1\right])$.

 \begin{example}\label{ex:idealextension}
Given an extension of $L_{\infty}$-algebras as in \ref{extensionLinfinity}, the space $\text{Ker}(\pi_1^1)=\mathfrak{n}\left[1\right]$ is an ideal of $\hat{\mathfrak{g}}\left[1\right]$, since for $x\in \mathfrak{n}\left[1\right]$ and $y\in S(\hat{\mathfrak{g}}\left[1\right])$ we have 
 \begin{align*}
     \pi_1^1(\hat{\lambda}(x\vee y))=&\lambda(\pi_1^1(x)\vee \pi(y))\\
     =&\lambda(0\vee \pi(y))\\
     =&0.
 \end{align*}
 Hence, $\hat{\lambda}(x\vee y) \in \mathfrak{n}\left[1\right]$.
 \end{example}

It follows from example \ref{ex:idealextension} that the adjoint representation up to homotopy $\ad:S(\hat{\mathfrak{g}}[1])\otimes \hat{\mathfrak{g}}[1]\to \hat{\mathfrak{g}}[1]$ restricts to a ruth $\ad|_{\mathfrak{n}}:S(\hat{\mathfrak{g}}[1])\otimes \mathfrak{n}[1]\to \mathfrak{n}[1]$. If $h:\mathfrak{g}\to \hat{\mathfrak{g}}$ is a section of \eqref{extensionLinfinity} with zero curvature, then by applying the pullback functor $h^*:\mathrm{Rep}_{\infty}(\hat{\mathfrak{g}})\to \mathrm{Rep}_{\infty}(\mathfrak{g})$ defined in \eqref{eq:pullbackruth} one gets a ruth 

\begin{equation}\label{eq:pullbackadjoint}
\rho:=h^*(\ad|_{\mathfrak{n}}):S(\mathfrak{g}[1])\otimes \mathfrak{n}[1]\to \mathfrak{n}[1].
\end{equation}

\noindent In particular, the $L_{\infty}$-algebra cochain complex $(C(\mathfrak{g},\mathfrak{n}[1]),D_{\rho})$ is well defined and $D^2_{\rho}=0$, see \eqref{eq:CEdifferential}.

In general, given a section $h:\mathfrak{g}\to \hat{\mathfrak{g}}$ its curvature does not necessarily vanish, hence $h:\mathfrak{g}\to \hat{\mathfrak{g}}$ is not an $L_{\infty}$-morphism in general. In spite of this, one can still consider the map $\rho=h^*(\ad|_{\mathfrak{n}})$ as in \eqref{eq:pullbackadjoint} but it does not define a representation up to homotopy. In particular, the degree 1 map $D_{\rho}:C(\mathfrak{g},\mathfrak{n}[1])\to C(\mathfrak{g},\mathfrak{n}[1])$ does not square to zero. 

The maps $\rho: S(\mathfrak{g}[1])\otimes \mathfrak{n}[1]\to \mathfrak{n}[1]$ and $D_{\rho}:C(\mathfrak{g},\mathfrak{n}[1])\to C(\mathfrak{g},\mathfrak{n}[1])$ are given by

\begin{equation}\label{eq:rhoadjoint}
\rho=\hat{\lambda}\circ \hat{\mu}\circ (S(h)\otimes \id_{\mathfrak{n}[1]}); 
\end{equation}

and 

\begin{equation}\label{eq:Dadjoint}
D_{\rho}\alpha:=\rho\circ(\id_{S(\mathfrak{g}[1])}\otimes \alpha)\circ \Delta_{S}-(-1)^{p}\alpha\circ d_{\lambda},
\end{equation}
for every homogeneous cochain $\alpha\in C^p(\mathfrak{g},\mathfrak{n}[1])$.

We have the following version of the Bianchi identity.

\begin{theorem}\label{thm:bianchi}
Let $h:\mathfrak{g}\to \hat{\mathfrak{g}}$ be a section of an extension of $L_{\infty}$-algebras \eqref{extensionLinfinity} with curvature $K_h\in C^1(\mathfrak{g},\mathfrak{n}[1])$. Then, the following holds 
$$D_{\rho}K_h=0.$$
\end{theorem}

\begin{proof}
The proof is a long, but straightforward computation:
\begin{align*}
    D_{\rho}K_h=&\rho\circ (\id_{S(\mathfrak{g}[1])}\otimes K_h)\circ\Delta_S-(-1)^{1}K_h\circ d_{\lambda}\\
    =&\rho\circ(\id_{S(\mathfrak{g}[1])}\otimes (\hat{\lambda}\circ S(h))\circ\Delta_S-\rho\circ (\id_{S(\mathfrak{g}[1])}\otimes h\circ\lambda)\circ\Delta_S\\
    &+(\hat{\lambda}\circ S(h))\circ d_{\lambda}-(h\circ\lambda) \circ d_{\lambda}.\\
    \end{align*}

Since $\lambda\circ d_{\lambda}=0$, we conclude that $D_{\rho}K_h= A-B+C$, where

\begin{align*}
A=&\rho\circ(\id_{S(\mathfrak{g}[1])}\otimes (\hat{\lambda}\circ
    S(h)))\circ\Delta_S\\
 B=&\rho\circ (\id_{S(\mathfrak{g}[1])}\otimes (h\circ\lambda))\circ\Delta_S\\
 C=&\hat{\lambda}\circ S(h)\circ d_{\lambda}.
\end{align*}

\noindent We proceed now to show that $A=0$. Indeed, since $\rho:=\hat{\lambda}\circ\hat{\mu}\circ(S(h)\otimes \id_{\mathfrak{n}[1]})$ is defined by \eqref{eq:rhoadjoint}, one has

\begin{align*}
    A=&\hat{\lambda}\circ\hat{\mu}\circ(S(h)\otimes \id_{\mathfrak{n}[1]})\circ(\id_{S(\mathfrak{g}[1])}\otimes(\hat{\lambda}\circ S(h)))\circ\Delta_S\\
    =&\hat{\lambda}\circ\hat{\mu}\circ(S(h)\otimes(\hat{\lambda}\circ S(h))\circ \Delta_S\\
    =&\hat\lambda \circ\hat \mu \circ (\id_{S(\mathfrak{g}[1])}\otimes \hat\lambda)\circ (S(h)\otimes S(h))\circ \Delta_S\\
    =&\hat\lambda \circ\hat \mu \circ (\id_{S(\mathfrak{g}[1])}\otimes \hat\lambda)\circ (\Delta_S\otimes S(h))
    \end{align*}    

\noindent where we have used the interchange law \eqref{eq:interchangelaw} twice and the fact that $S(h)$ is a coalgebra morphism.  But $\hat\mu=\hat\mu\circ T$ and $\Delta_S=T\circ \Delta_S$, where $T$ is the twisting map \eqref{eq:twistingmap}, so replacing these expressions above and using $d_{\hat\lambda}=\hat\mu\circ (\hat\lambda\otimes \id_{S(\hat{\mathfrak{g}}[1])})\circ \Delta_S $ (see \eqref{eq:inducedcoderivation}), one concludes
    
  \begin{align*}  
   A=&\hat{\lambda}\circ\hat{\mu}\circ{ T\circ(\id_{S(\hat{\mathfrak{g}}[1])}\otimes \hat{ \lambda})\circ T}\circ\Delta_S\otimes S(h)\\
    =&\hat{\lambda}\circ\hat{\mu}\circ(\hat{\lambda}\otimes \id_{S(\hat{\mathfrak{g}}[1])})\circ\Delta_S\circ S(h)\\
    =& \hat{\lambda}\circ d_{\hat\lambda}\circ S(h)\\
    =&0,
\end{align*}

\noindent where we used $\hat{\lambda}\circ d_{\hat\lambda}=0$.

Regarding the term $B$, one can use the definition of $\rho$ as above, which yields
\begin{align*}
    B=&\hat{\lambda}\circ\hat{\mu}\circ(S(h)\otimes \id_{\mathfrak{n}[1]})\circ(\id_{S(\mathfrak{g}[1])}\otimes (h\circ \lambda))\circ\Delta_S\\
    =&\hat{\lambda}\circ\hat{\mu}\circ(S(h)\otimes(h\circ \lambda))\circ\Delta_S,
\end{align*}

\noindent where we have used the interchange law \eqref{eq:interchangelaw}.

Finally, we compute the term $C$. For that, using the definition of the coderivation $d_{\lambda}=\mu\circ (\lambda\otimes \id_{S(\mathfrak{g}[1])})\circ \Delta_S$ as in \eqref{eq:inducedcoderivation}, we have
\begin{align*}
    C=&\hat{\lambda}\circ S(h)\circ\mu\circ(\lambda\otimes \id_{S(\mathfrak{g}[1])})\circ\Delta_S\\
    =&\hat{\lambda}\circ \hat{\mu}\circ(S(h)\otimes S(h))\circ(\lambda \otimes \id_{S(\mathfrak{g}[1])})\circ\Delta_S\\
    =&\hat{\lambda}\circ \hat{\mu}\circ ((S(h)\circ \lambda)\otimes S(h))\circ \Delta_{S}
\end{align*}    
    
\noindent where the interchange law was used in the last identity. Again, because $\hat\mu=\hat\mu\circ T$ and $\Delta_S=T\circ \Delta_S$, where $T$ is the twisting map \eqref{eq:twistingmap}, one gets

 \begin{align*}   
    C=&\hat{\lambda}\circ \hat{\mu}\circ ((S(h)\circ \lambda)\otimes S(h))\circ \Delta_{S}\\
    =&\hat{\lambda}\circ \hat{\mu}\circ T\circ((h\circ\lambda)\otimes S(h))\circ T\circ\Delta_S\\
    =&\hat{\lambda}\circ \hat{\mu}\circ(S(h)\otimes(h\circ \lambda))\circ\Delta_S,
\end{align*}

\noindent showing that $B=C$. Hence,

\begin{align*}
    D_{\rho}K_h=& A-B+C\\
    =&0.
\end{align*}
\end{proof}
Note that if $h_0,h_1:\mathfrak{g}\to \hat{\mathfrak{g}}$ are sections of \eqref{extensionLinfinity}, then $h_1-h_0\in \underline{\Hom}(\mathfrak{g}[1],\mathfrak{n}[1])$. In other words, the space of sections of \eqref{extensionLinfinity} is an affine space. In particular, there is a one-parameter family of sections $h_t:\mathfrak{g}\to \hat{\mathfrak{g}}$ defined as

$$h_t:=h_0+t(h_1-h_0), t\in [0,1].$$

\noindent Such a family of sections induces maps $\rho_t:S(\mathfrak{g}[1])\otimes \mathfrak{n}[1]\to \mathfrak{n}[1]$ and $D_{\rho_t}:C(\mathfrak{g},\mathfrak{n}[1])\to C(\mathfrak{g},\mathfrak{n}[1])$ given by \eqref{eq:rhoadjoint} and \eqref{eq:Dadjoint}, respectively. Let $K_{h_t}\in C^1(\mathfrak{g},\mathfrak{n}[1])$ be the curvature of $h_t$. One has the cochain $\alpha\in C(\mathfrak{g},\mathfrak{n}[1])$ defined by

\begin{equation}\label{eq:curvaturevariationcochain}
\alpha:= (h_1-h_0)\circ \mathrm{pr}_{\mathfrak{g}[1]},
\end{equation}

\noindent where $\mathrm{pr}_{\mathfrak{g}[1]}:S(\mathfrak{g}[1])\to \mathfrak{g}[1]$ is the canonical projection. It turns out that $\alpha$ controls the variation of the one-parameter family of curvatures $K_{h_t}$.

\begin{proposition}\label{prop:curvaturevariation}

Let $h_t:\mathfrak{g}\to \mathfrak{\mathfrak{g}}$ be a one-parameter family of sections given by $h_t=h_0+t(h_1-h_0), t\in [0,1]$. Then

$$\frac{d}{dt}K_{h_t}=D_{\rho_t}\alpha,$$
with $\alpha\in C(\mathfrak{g},\mathfrak{n}[1])$ as in \eqref{eq:curvaturevariationcochain}.

\end{proposition}

\begin{proof}

For every $\epsilon\in \mathbb{R}$, consider the variation $h_{t+\epsilon}=h_t+\epsilon\alpha$ and its curvature $K_{h_{t+\epsilon}}$. Note that 

$$\left.\frac{d}{d\epsilon}\right|_{\epsilon=0}K_{h_{t+\epsilon}}=\frac{d}{dt}K_{h_t},$$
for every $t\in[0,1]$. We proceed now to compute the curvature $K_{h_{t+\epsilon}}$. If $c_{t,\epsilon}:=S(h_t+\epsilon\alpha)-S(h_t)-S(\epsilon\alpha)$, then:

\begin{align*}
    K_{h_{t+\epsilon}}=&\hat{\lambda}\circ S(h_{t+\epsilon})-h_{t+\epsilon}\circ \lambda\\
    =&\hat{\lambda}\circ S(h_{t}+\epsilon\alpha)-(h_t+\epsilon\alpha)\circ \lambda\\
    =&\hat{\lambda}\circ(S(h_t)+S(\epsilon\alpha)+c_{t,\epsilon})-h_t\circ \lambda -\epsilon\alpha\circ \lambda\\
    =& K_{h_t}+(\hat{\lambda}\circ S(\epsilon\alpha)-\epsilon\alpha\circ\lambda)+\hat{\lambda}\circ c_{t,\epsilon}.
\end{align*}

Since $\left.\frac{d}{d\epsilon}\right|_{\epsilon=0}K_{h_t}=0$, we conclude

$$\frac{d}{dt}K_{h_t}=\left.\frac{d}{d\epsilon}\right|_{\epsilon=0}K_{h_{t+\epsilon}}=A+B,$$

\noindent where $A=\left.\frac{d}{d\epsilon}\right|_{\epsilon=0}(\hat{\lambda}\circ S(\epsilon\alpha)-\epsilon\alpha\circ\lambda)$ and $B=\left.\frac{d}{d\epsilon}\right|_{\epsilon=0}\hat{\lambda}\circ c_{t,\epsilon}.$

Let us compute the term $A$ in the expression above.

\begin{align*}
    A=&\sum_{k}\left.\frac{d}{d\epsilon}\right|_{\epsilon=0}\hat{\lambda}_k\circ t_0^k S^k(\alpha)-\left.\frac{d}{d\epsilon}\right|_{\epsilon=0}\epsilon\alpha\circ \lambda\\
    =&\hat{\lambda}_1\circ S^1(\alpha)\circ \mathrm{pr}_{\mathfrak{g}[1]}-\alpha\circ \lambda,\\
    =&\hat{\lambda}_1\circ \alpha-\alpha\circ \lambda,
\end{align*}

Regarding the term $B$, one has
 
\begin{equation}\label{eq:eq1curvaturevariation}
    \left.\frac{d}{d\epsilon}\right|_{\epsilon=0}\hat{\lambda}\circ c_{t,\epsilon}=\sum_k\hat{\lambda}_k\left(\left.\frac{d}{d\epsilon}\right|_{\epsilon=0}c_{t,\epsilon}^k\right).
\end{equation}

On the one hand
\begin{align*}
    \left.\frac{d}{d\epsilon}\right|_{\epsilon=0}c_{t,\epsilon}^k=&\left.\frac{d}{d\epsilon}\right|_{\epsilon=0}(S^k(h_t+\epsilon \alpha)-S^k(h_t)-S^k(\epsilon\alpha))\\
    =&\left.\frac{d}{d\epsilon}\right|_{\epsilon=0}S^k(h_t+\epsilon\alpha)-\left.\frac{d}{d\epsilon}\right|_{\epsilon=0}S^k(\epsilon\alpha),
\end{align*}
where we have used $\left.\frac{d}{d\epsilon}\right|_{\epsilon=0}S^k(h_t)=0$. Note that 

$$\left.\frac{d}{d\epsilon}\right|_{\epsilon=0}S^k(\epsilon\alpha)=
\begin{cases}
S^1(\alpha); \quad k=1\\
0; \quad k\geq 2.
\end{cases}$$

On the other hand, for homogeneous elements $x_1,...,x_k \in \mathfrak{g}[1],$ the following holds:

\begin{align*}
    \left.\frac{d}{d\epsilon}\right|_{\epsilon=0}S^k&(h_t+\epsilon\alpha)(x_1\vee\cdots\vee x_k)=\left.\frac{d}{d\epsilon}\right|_{\epsilon=0}(h_t+\epsilon\alpha)x_1\vee\cdots\vee (h_t+\epsilon\alpha)x_k\\
    =&\sum_{i=1}^kh_t(x_1)\vee\cdots\vee h_t(x_{i-1})\vee \alpha(x_i)\vee h_t(x_{i+1})\vee\cdots\vee h_t(x_k).
    \end{align*}
In other words,
 $$\left.\frac{d}{d\epsilon}\right|_{\epsilon=0}S^k(h_t+\epsilon\alpha)=\sum_{i=1}^kh_t\vee\cdots\vee \alpha\vee\cdots\vee h_t,$$
and hence \eqref{eq:eq1curvaturevariation} reads
 
\begin{align*}
     \left.\frac{d}{d\epsilon}\right|_{\epsilon=0}\hat{\lambda}\circ c_{t,\epsilon}=&\sum_{k}\hat{\lambda}_k\left(\sum_{i=1}^kh_t\vee\cdots\vee \alpha\vee\cdots\vee h_t\right)-\hat{\lambda}_1\circ \alpha\circ \mathrm{pr}_{\mathfrak{g}[1]}.
\end{align*}
Therefore, 

\begin{align*}
     \left.\frac{d}{d\epsilon}\right|_{\epsilon=0}K_{h_t+\epsilon\alpha}=& A+B\\
     =&\hat{\lambda}_1\circ \alpha-\alpha\circ \lambda+\sum_{k}\hat{\lambda}_k\left(\sum_{i=1}^kh_t\vee\cdots\vee \alpha\vee\cdots\vee h_t\right)\\
     &-\hat{\lambda}_1\circ \alpha\circ \mathrm{pr}_{\mathfrak{g}[1]}\\
     =&\sum_{k}\left(\sum_{i=1}^k\hat{\lambda}_k (h_t\vee\cdots\vee \alpha\vee\cdots\vee h_t)\right)-\alpha\circ\lambda.
\end{align*}

Summing up, one has

\begin{equation}\label{eq:eq2curvaturevariation}
\left.\frac{d}{d\epsilon}\right|_{\epsilon=0}K_{h_t+\epsilon\alpha}=\sum_{k}\left(\sum_{i=1}^k\hat{\lambda}_k (h_t\vee\cdots\vee \alpha\vee\cdots\vee h_t)\right)-\alpha\circ\lambda.
\end{equation}

The first term of the left hand side of \eqref{eq:eq2curvaturevariation} can be written as

\begin{equation}\label{eq:eq3curvaturevariation}
\sum_{k}\left(\sum_{i=1}^k\hat{\lambda}_k (h_t\vee\cdots\vee \alpha\vee\cdots\vee h_t)\right)=\hat{\lambda}\circ\hat{\mu}\circ (S(h_t)\otimes \hat\alpha)\circ\Delta_S,
\end{equation}
where $\hat\alpha:=h_1-h_0$. Note that this immediately implies the statement of the proposition, namely:

$$\frac{d}{dt}K_{h_t}=D_{\rho_t}\alpha.$$
Indeed, the variation of the curvatures has the expression

\begin{align*}
    \frac{d}{dt}K_{h_t}=\left.\frac{d}{d\epsilon}\right|_{\epsilon=0}K_{h_t+\epsilon\alpha}=&\sum_{k}\left(\sum_{i=1}^k\hat{\lambda}_k (h_t\vee\cdots\vee \alpha\vee\cdots\vee h_t)\right)-\alpha\circ\lambda\\
    =&\rho_t\circ(\id_{S(\mathfrak{g}[1])}\otimes \alpha)\circ \Delta_S-(-1)^{|\alpha|}\alpha\circ d_{\lambda}\\
    =&D_{\rho_t}\alpha.
\end{align*}

\noindent It only remains to show \eqref{eq:eq3curvaturevariation}. To see this, recall that $\alpha=\hat\alpha\circ \mathrm{pr}_{\mathfrak{g}[1]}$, therefore on homogeneous elements $x_1,...,x_k\in S(\mathfrak{g}[1])$ the following holds

\begin{align*}
\hat{\lambda}\circ\hat{\mu}\circ(S(h_t)\otimes\hat\alpha)\Delta_S(x_1\vee\cdots \vee x_k)=&\hat{\lambda}\circ\hat{\mu}\circ(S(h_t))\left(\sum_{\sigma\in Sh(k-1,1)}\epsilon(\sigma)x_{\sigma(1)}\vee\cdots\vee x_{\sigma(k-1)}\otimes \alpha(x_{\sigma(k)})\right)\\
    =&\hat{\lambda}\circ \hat{\mu}\left(\sum_{\sigma\in Sh(k-1,1)}\epsilon(\sigma)h_t(x_{\sigma(1)})\vee\cdots\vee h_t(x_{\sigma(k-1)})\otimes \alpha(x_{\sigma(k)})\right)\\
    =&\hat{\lambda}\left(\sum_{\sigma \in Sh(k-1,1)}\epsilon(\sigma)h_t(x_{\sigma(1)})\vee\cdots\vee h_t(x_{\sigma(k-1)})\vee \alpha(x_{\sigma(k)})\right).
\end{align*}

\noindent
Note that since

$$Sh(k-1,1)=\left\lbrace \sigma \in \Sigma_k\,|\,\sigma(1)<\sigma(2)<\cdots<\sigma(k-1)\right\rbrace$$ has precisely $k$ elements, then
 
\begin{align*}
\hat{\lambda}\circ\hat{\mu}\circ(S(h_t)\otimes\hat\alpha)\Delta_S(x_1\vee\cdots \vee x_k)=&\sum_{k}\hat{\lambda}_k\left(\sum_{i=1}^k\epsilon(\sigma)h_t(x_{1})\vee\cdots\vee \alpha(x_i)\vee\cdots h_t(x_{k})\right)\\
     =&\sum_{k}\left(\sum_{i=1}^k\hat{\lambda}_k(h_t\vee\cdots\vee \alpha\vee\cdots\vee h_t)\right)(x_1\vee\cdots\vee x_k).
\end{align*}

\noindent which shows \eqref{eq:eq3curvaturevariation}.

\end{proof}

%%%%%%%%%%%%%%%%%%%%%%%%%%%%%%%%%%%%%
%%%%%%%%%%%%%%%%%%%%%%%%%%%%%%%%%%%%%
%%%%%%%%%%%%%%%%%%%%%%%%%%%%%%%%%%%%%

\subsection{Chern-Weil-Lecomte map}

Let us consider an extension of $L_{\infty}$-algebras

\begin{center}
    \begin{tikzcd}
    0\arrow[r]&\mathfrak{n}\arrow[r,"\iota"]&\hat{\mathfrak{g}}\arrow[r,"\pi"]&\mathfrak{g}\arrow[r]&0
    \end{tikzcd}
\end{center}
as in \eqref{extensionLinfinity}. Assume that $(\mathbb{V},\rho)$ is a representation up to homotopy of $\mathfrak{g}$. Since $\pi:\hat{\mathfrak{g}}\to \mathfrak{g}$ is an $L_{\infty}$-algebra morphism, one gets a pullback ruth $(\mathbb{V},\pi^*\rho)$ of $\hat{\mathfrak{g}}$ on $\mathbb{V}$. As we observed in Example \ref{ex:idealextension}, the subspace $\mathfrak{n}[1]\subseteq \hat{\mathfrak{g}}[1]$ is an ideal and hence the adjoint ruth of $\hat{\mathfrak{g}}$ restricts to a ruth $\ad|_{\mathfrak{n}}:S(\hat{\mathfrak{g}}[1])\otimes \mathfrak{n}[1]\to \mathfrak{n}[1]$, which extends to a ruth of $\hat{\mathfrak{g}}$ on $\wedge^k\mathfrak{n}[1]$ for every $k\geq 1$. This induced representation is denoted $(\wedge^k\mathfrak{n}[1],\wedge^k\ad|_{\mathfrak{n}})$ for every $k\geq 1$. In this way, we end up with two ruths

$$(\mathbb{V},\pi^*\rho), (\wedge^k\mathfrak{n}[1],\wedge^k\ad|_{\mathfrak{n}})\in \mathrm{Rep}_{\infty}(\hat{\mathfrak{g}}).$$

\noindent The corresponding space of ruths morphisms is denoted by

\begin{equation}\label{eq:equivariantpolynomials}
\underline{\Hom}^{\bullet}(\wedge^k\mathfrak{n}[1],\mathbb{V})^{\hat{\mathfrak{g}}}:=\{ f \in \underline{\Hom}^{\bullet}(\wedge^k\mathfrak{n}[1],\mathbb{V})\,|\, f \text{ ruth map}\}.
\end{equation}

One easily observes that if $f\in \underline{\Hom}^{\bullet}(\wedge^k\mathfrak{n}[1],\mathbb{V})^{\hat{\mathfrak{g}}}$ is a ruth map, then $f$ is $\mathfrak{g}$-equivariant in the sense that

\begin{equation}\label{eq:equivariantnotruth}
f\circ h^*(\ad^{\otimes k}|_{\mathfrak{n}})=\rho\circ (\id_{S(\mathfrak{g}[1])}\otimes f),
\end{equation}

\noindent where $h^*(\ad^{\otimes k}|_{\mathfrak{n}}):=(\ad^{\otimes k}|_{\mathfrak{n}})\circ (S(h)\otimes \id_{\mathfrak{n}[1]})$. It is worth noticing that, although $h:\mathfrak{g}\to \hat{\mathfrak{g}}$ is not an $L_{\infty}$-algebra morphism, the map $h^*(\ad^{\otimes k}|_{\mathfrak{n}})$ is well defined but it is not a representation of $\mathfrak{g}$. Even though, the map $h^*(\ad^{\otimes k}|_{\mathfrak{n}})$ induces a degree 1 derivation $D_{h^*(\ad^{\otimes k}|_{\mathfrak{n}})}:C(\mathfrak{g},\wedge^k\mathfrak{n}[1])\to C(\mathfrak{g},\wedge^k\mathfrak{n}[1])$ which does not square to zero. In spite of this, $f:\wedge^k\mathfrak{n}[1]\to \mathbb{V}$ yields a pullback map $f^*:C(\mathfrak{g},\mathbb{V})\to C(\mathfrak{g},\wedge^k\mathfrak{n}[1])$ and property \eqref{eq:equivariantnotruth} implies that the computations in Proposition \ref{equivariance} hold true, so 

\begin{equation}\label{eq:intertwiningnotruth}
D_{\rho}\circ f^*=f^*\circ D_{h^*(\ad^{\otimes k}|_{\mathfrak{n}})}.
\end{equation}

We will see that elements in $\underline{\Hom}^{\bullet}(\wedge^k\mathfrak{n}[1],\mathbb{V})^{\hat{\mathfrak{g}}}$ applied to the curvature of a section $h:\mathfrak{g}\to \hat{\mathfrak{g}}$ of an extension \eqref{extensionLinfinity} define $L_{\infty}$-cocycles in the cochain complex $(C(\mathfrak{g},\mathbb{V}),D_{\rho})$. For that, let $K_h\in C^1(\mathfrak{g},\mathfrak{n}[1])$ be the curvature of a section $h:\mathfrak{g}\to \hat{\mathfrak{g}}$ of \eqref{extensionLinfinity}. As seen in Example \ref{ex:Idcochain}, the identity map $\id:\mathfrak{n}[1]\otimes \mathfrak{n}[1]\to \mathfrak{n}[1]\otimes \mathfrak{n}[1]$ defines a cochain product $\wedge_{\id}:C(\mathfrak{g},\mathfrak{n}[1])\otimes C(\mathfrak{g},\mathfrak{n}[1])\to C(\mathfrak{g},\mathfrak{n}[1]\otimes \mathfrak{n}[1])$. In particular, when applied to the curvature one gets an element

$$K^{\wedge^k}_h:=K_h\wedge_{\id}\cdots \wedge_{\id} K_h\in C^k(\mathfrak{g},\mathfrak{n}[1]^{\otimes k}),$$
for every $k\geq 1$. Now, if $f\in \underline{\Hom}^{\bullet}(\wedge^k\mathfrak{n}[1],\mathbb{V})_{\bullet}^{\hat{\mathfrak{g}}}$, then

\begin{equation}\label{eq:CWLcocycle}
f_{K_h}:=f\circ K^{\wedge^k}_h\in   C^{k+\bullet}(\mathfrak{g},\mathbb{V}).
\end{equation}

Now we state the $L_{\infty}$-version of the Chern-Weil-Lecomte morphism, which is the main result of this section.

  \begin{theorem}\label{thm:lecomte}

  Let $K_h\in C^1(\mathfrak{g},\mathfrak{n}[1])$ be the curvature of a section $h:\mathfrak{g}\to \hat{\mathfrak{g}}$ of an extension of $L_{\infty}$-algebras

  \begin{center}
 \begin{tikzcd}
    0\arrow[r]&\mathfrak{n}\arrow[r,"\iota"]&\hat{\mathfrak{g}}\arrow[r,"\pi"]&\mathfrak{g}\arrow[r]&0,
 \end{tikzcd}
\end{center}

\noindent and $(\mathbb{V},\rho)\in \mathrm{Rep}_{\infty}(\mathfrak{g})$. If $f\in \underline{\Hom}^{\bullet}(\wedge^k\mathfrak{n}[1],\mathbb{V})^{\hat{\mathfrak{g}}}$, then

\begin{enumerate}

\item the element $f_{K_h}$ given by \eqref{eq:CWLcocycle} is a cocycle in the complex $(C(\mathfrak{g},\mathbb{V}),D_{\rho})$,

\item the cohomology class $[f_{K_h}]\in H^{\bullet+k}(\mathfrak{g},\mathbb{V})$ does not depend of the section $h$.

\end{enumerate}

  \end{theorem}

\begin{proof}

First we show that $f_{K_h}$ is a cocycle, that is, $D_{\rho}f_{K_h}=0$. To see this, since $f:\wedge^k\mathfrak{n}[1]\to \mathbb{V}$ is a ruth map, the induced pullback $f^*:C(\mathfrak{g},\mathbb{V})\to C(\mathfrak{g},\wedge^k\mathfrak{n}[1])$ commutes with the corresponding derivations, see \eqref{eq:intertwiningnotruth}. Also, by definition $f_{K_h}=f^*(\wedge^kK_h)$ then 

$$D_{\rho}f_{K_h}=f^*D_{h^*(\ad^{\otimes k}|_{\mathfrak{n}})}(\wedge^kK_h).$$

Using the Leibniz rule (see Prop. \ref{prop:leibnizcochain}) together with the Bianchi identity (see Theorem \ref{thm:bianchi}) yields $D_{h^*(\ad^{\otimes k}|_{\mathfrak{n}})}(\wedge^kK_h)=0$, which implies that $D_{\rho}f_{K_h}=0$.

Now we proceed to check \emph{(2)}. We will show that if $h_0,h_1:\mathfrak{g}\to \hat{\mathfrak{g}}$ are sections, then $f_{K_{h_1}}, f_{K_{h_0}}$ are cohomologous. For that, consider the family of sections $h_t=h_0+t(h_1-h_0)$ with $t\in [0,1]$ together with the corresponding cocycle $f_{K_{h_t}}\in C(\mathfrak{g},\mathbb{V})$ as in \eqref{eq:CWLcocycle}.  To see that $[f_{K_{h_1}}]= [f_{K_{h_0}}]$, it suffices to find $\beta_t\in C^{k-1+\bullet}(\mathfrak{g},\mathbb{V})$ satisfying

\begin{equation}\label{eq:eq1prooflecomte}
\frac{d}{dt}f_{K_{h_t}}=D_{\rho}\beta_t.
\end{equation}
\noindent Indeed, if such a $\beta_t$ exists, then 

\begin{align*}
f_{K_{h_1}}- f_{K_{h_0}}=&\int_{0}^{1}\frac{d}{dt}f_{K_{h_t}}dt\\
=& \int_{0}^{1}D_{\rho}\beta_tdt\\
=& D_{\rho}\big{(}\int_{0}^{1}\beta_tdt\big{)}
\end{align*}

We will show that \eqref{eq:eq1prooflecomte} holds. In fact, by the curvature variation formula of Proposition \ref{prop:curvaturevariation} and the Leibniz rule in Proposition \ref{prop:leibnizcochain}, one has

\begin{equation}\label{eq:eq2prooflecomte}
\frac{d}{dt}f_{K_{h_t}}=f^*(\frac{d}{dt}K^{\wedge^k}_{h_t})\\
=\sum f^*(K_{h_t}\wedge_{\id}\cdots \wedge_{\id} D_{\rho_t}\alpha\wedge_{\id}\cdots \wedge_{\id} K_{h_t}).
\end{equation}

A direct computation using the skew-symmetry of $f$ gives rise to 

$$f^*(K_{h_t}\wedge_{\id}\cdots \wedge_{\id} D_{\rho_t}\alpha\wedge_{\id}\cdots \wedge_{\id} K_{h_t})=f^*(D_{\rho_t}\alpha\wedge_{\id}K_{h_t}^{\wedge^{(k-1)}}).$$

\noindent As a consequence, identity \eqref{eq:eq2prooflecomte} reads

\begin{equation}\label{eq:eq3prooflecomte}
\frac{d}{dt}f_{K_{h_t}}=kf^*(D_{\rho_t}\alpha\wedge_{\id}K_{h_t}^{\wedge^{(k-1)}}).
\end{equation}

\noindent Define now $\beta_t:=kf^*(\alpha\wedge_{\id}K^{\wedge(k-1)}_{h_t})$. Then, using $D_{\rho}\circ f^*=f^*\circ D_{\rho_t}$ and the Leinbiz rule for $D_{\rho_t}$ as in Proposition \eqref{prop:leibnizcochain}, one gets

\begin{align*}
D_{\rho}\beta_t=&kf^*\big{(}(D_{\rho_t}\alpha)\wedge_{\id}K^{\wedge^{(k-1)}}_{h_t}+\alpha \wedge_{\id} D_{\rho_t}K^{\wedge^{(k-1)}}_{h_t}\big{)}\\
=&\frac{d}{dt}f_{K_{h_t}},
\end{align*}

\noindent where in the last identity we have used \eqref{eq:eq3prooflecomte} together with the Bianchi identity as in Theorem \ref{thm:bianchi}. This proves \eqref{eq:eq1prooflecomte} and hence statement \emph{(2)} as desired.
\end{proof}

As a consequence of Theorem \ref{thm:lecomte}, given an extension of $L_{\infty}$-algebras 

 \begin{center}
 \begin{tikzcd}
    0\arrow[r]&\mathfrak{n}\arrow[r,"\iota"]&\hat{\mathfrak{g}}\arrow[r,"\pi"]&\mathfrak{g}\arrow[r]&0,
 \end{tikzcd}
\end{center}

\noindent together with a representation up to homotopy $(\mathbb{V},\rho)\in \mathrm{Rep}_{\infty}(\mathfrak{g})$, there is a well defined map

\begin{align}\label{CWLmap}
cw:\underline{\Hom}^{\bullet}(\wedge^k\mathfrak{n}[1],&\mathbb{V})^{\hat{\mathfrak{g}}}  \to H^{\bullet+k}(\mathfrak{g},\mathbb{V})\\
& f\mapsto  [f_{K_h}], \nonumber
\end{align}
where $h:\mathfrak{g}\to \hat{\mathfrak{g}}$ is any section. The map \eqref{CWLmap} is referred to as the \textbf{Chern-Weil-Lecomte} map of the extension.

%\textcolor{red}{Additionally, if $(\mathbb{V},m)$ is a graded algebra, it is natural to expect the Chern-Weil-Lecomte map to have an extra property. Indeed, the cohomology $H^{\bullet}(\mathfrak{g},\mathbb{V})$ is a graded algebra as well, see Corollary \ref{cor:cochaincohomologygradedalgebra}. Also, the graded vector space} 

%$$\mathrm{Skew}(\mathfrak{n}[1],\mathbb{V})\cong \underline{\Hom}^{\bullet}(\wedge^{\bullet}\mathfrak{n}[1],\mathbb{V}),$$

%\noindent inherits the structure of a graded algebra with product

%$$(\alpha\cdot_m\beta)(x_1,...,x_{k+l}):=\sum_{\sigma\in S_{k+l}}\chi(\sigma,x_1,...,x_{k+l})m\big{(}\alpha(x_{\sigma(1)},...,x_{\sigma(k)})\otimes \beta(x_{\sigma(k+1)},...,x_{\sigma(k+l)})\big{)},$$

%\noindent where $\alpha\in \mathrm{Skew}^k(\mathfrak{n}[1],\mathbb{V})$ and $\beta\in \mathrm{Skew}^l(\mathfrak{n}[1],\mathbb{V})$.

%We will see that the Chern-Weil-Lecomte map is actually a morphism of graded algebras.

%\begin{proposition}\label{prop:cwlmorphism}

%The map $cw:\underline{\Hom}^{\bullet}(\wedge^k\mathfrak{n}[1],\mathbb{V})^{\hat{\mathfrak{g}}}  \to H^{\bullet+k}(\mathfrak{g},\mathbb{V})$ is a graded algebra morphism.

%\end{proposition}

%\begin{proof}
%\textcolor{red}{Check this up.}
%\end{proof}

%In this case, we call \eqref{CWLmap} the \textbf{Chern-Weil-Lecomte morphism} of the extension.

%%%%%%%%%%%%%%%%%%%%%%%%%%%%%%%%%%%%%%%%%%%%%%%%%
%%%%%%%%%%%%%%%%%%%%%%%%%%%%%%%%%%%%%%%%%%%%%%%%%

\subsection{Naturality}

In this section we show the naturality of the Chern-Weil-Lecomte morphism. More precisely, we will see that the pullback construction in the category of $L_{\infty}$-algebras induces a notion of pullback of extensions of $L_{\infty}$-algebras. If additionally, extensions come with representations up to homotopy, we study how the Chern-Weil-Lecomte morphism behaves with respect to these pullbacks.

As shown in \cite[Thm 5.9]{Rogers} (see also \cite[Thm 2.1]{Vallette}), given two $L_{\infty}$-algebra morphisms $T:\mathfrak{h}\to \mathfrak{g}$ and $T':\mathfrak{h}'\to \mathfrak{g}$ one has a pullback $L_{\infty}$-algebra $\mathfrak{h}\times_{\mathfrak{g}}\mathfrak{h}'$ which is an $L_{\infty}$-subalgebra of the direct sum $\mathfrak{h}\oplus \mathfrak{h}'$. We observe that this pullback construction extends to $L_{\infty}$-algebra extensions. Indeed, given $L_{\infty}$-algebras $\mathfrak{g}$ and $\mathfrak{n}$, extensions \eqref{extensionLinfinity} of $\mathfrak{g}$ by $\mathfrak{n}$  form a category $\mathrm{Ext}(\mathfrak{g},\mathfrak{n})$ in which morphisms are given by commutative diagrams of $L_{\infty}$-morphisms:

\begin{center}
 \begin{tikzcd}
    0\arrow[r]&\mathfrak{n}\arrow[d,"\id"]\arrow[r,]&\hat{\mathfrak{g}}\arrow[d,"\varphi"]\arrow[r,]&\mathfrak{g}\arrow[d,"\id"]\arrow[r]&0\\
    0\arrow[r]&\mathfrak{n}\arrow[r,]&\tilde{\mathfrak{g}}\arrow[r,]&\mathfrak{g}\arrow[r]&0.\\

 \end{tikzcd}
\end{center}

\noindent We simply say that $\varphi:\hat{\mathfrak{g}}\to \tilde{\mathfrak{g}}$ is a morphism of extensions $\hat{\mathfrak{g}},\tilde{\mathfrak{g}}\in \mathrm{Ext}(\mathfrak{g},\mathfrak{n})$. It is clear that every arrow in $\mathrm{Ext}(\mathfrak{g},\mathfrak{n})$ is necessarily invertible. 

Any $L_{\infty}$-algebra morphism $T:\mathfrak{h}\to \mathfrak{g}$ defines a pullback functor $T^*:\mathrm{Ext}(\mathfrak{g},\mathfrak{n})\to \mathrm{Ext}(\mathfrak{h},\mathfrak{n})$ defined on objects by $T^*\hat{\mathfrak{g}}:=\hat{\mathfrak{h}}$ where $\hat{\mathfrak{h}}:=\mathfrak{h}\times_{\mathfrak{g}}\hat{\mathfrak{g}}$ is the pullback defined by the maps $\pi:\hat{\mathfrak{g}}\to \mathfrak{g}$ and $T:\mathfrak{h}\to \mathfrak{g}$. The pullback extension is depicted as in the next commutative diagram of $L_{\infty}$-algebras:

\begin{center}\label{eq:pullbackextension}
 \begin{tikzcd}
    0\arrow[r]&\mathfrak{n}\arrow[d,"\id"]\arrow[r,"\iota_{\bullet}"]&\hat{\mathfrak{h}}\arrow[d,"T_{\bullet}"]\arrow[r,"\pi_{\bullet}"]&\mathfrak{h}\arrow[d,"\id"]\arrow[r]&0\\
    0\arrow[r]&\mathfrak{n}\arrow[r,"\iota"]&\hat{\mathfrak{g}}\arrow[r,"\pi"]&\mathfrak{g}\arrow[r]&0.\\

 \end{tikzcd}
\end{center}

As for any extension of $L_{\infty}$-algebras, the graded vector space $\mathfrak{n}$ in \eqref{eq:pullbackextension} is an ideal of both $\hat{\mathfrak{g}}$ and its pullback $\hat{\mathfrak{h}}$ via $T:\mathfrak{h}\to \mathfrak{g}$. In particular, one can see $\wedge^k\mathfrak{n}[1]; k\geq 1$ as a ruth in two different ways, namely $\wedge^k\mathfrak{n}[1]\in \mathrm{Rep}_{\infty}(\hat{\mathfrak{g}})$ and $\wedge^k\mathfrak{n}[1]\in \mathrm{Rep}_{\infty}(\hat{\mathfrak{h}})$ via the adjoint action. Additionally, if $(\mathbb{V},\rho)\in \mathrm{Rep}_{\infty}(\mathfrak{g})$ and $(\mathbb{W},\tau)\in \mathrm{Rep}_{\infty}(\mathfrak{h})$ are representations up to homotopy, then there are pullback ruths $(\mathbb{V},\pi^*\rho)\in \mathrm{Rep}_{\infty}(\hat{\mathfrak{g}})$ and $(\mathbb{W},\pi^*_{\bullet}\tau)\in \mathrm{Rep}_{\infty}(\hat{\mathfrak{h}})$. It is reasonable to require $T:\mathfrak{h}\to \mathfrak{g}$ to be part of a ruth morphism, so we assume now that $(T,t):(\mathbb{W},\tau) \to (\mathbb{V},\rho)$ is a ruths morphism, i.e. $T:\mathfrak{h}\to \mathfrak{g}$ is an $L_{\infty}$-algebra morphism and $t:\mathbb{V}\to \mathbb{W}$ is a map of graded vector spaces satisfying condition \eqref{eq:ruthmorphism}. One can easily see that there is a well defined map

\begin{equation}\label{eq:tmapnaturality}
t^{\sharp}:\underline{\Hom}^{\bullet}(\wedge^k\mathfrak{n}[1],\mathbb{V})^{\hat{\mathfrak{g}}} \to \underline{\Hom}^{\bullet}(\wedge^k\mathfrak{n}[1],\mathbb{W})^{\hat{\mathfrak{h}}}; f\mapsto t\circ f
\end{equation}

\noindent for every $k\geq 1$.

It is clear that every section $h:\mathfrak{g}\left[1\right]\to \hat{\mathfrak{g}}\left[1\right]$ of $\pi$ induces a section $\overline{h}:\mathfrak{h}[1]\to \hat{\mathfrak{h}}[1]$ of the pullback extension \eqref{eq:pullbackextension}, with the property: 

\begin{equation}\label{eq:pullbacksection}
T_{\bullet}\circ S(\overline{h})=S(h)\circ T.
\end{equation}

\noindent In particular, for every $k\geq 1$ there are Chern-Weil-Lecomte morphisms

$$cw:\underline{\Hom}^{\bullet}(\wedge^k\mathfrak{n}[1],\mathbb{V})^{\hat{\mathfrak{g}}}  \to H^{\bullet+k}(\mathfrak{g},\mathbb{V}) \quad \text{ and } \quad cw:\underline{\Hom}^{\bullet}(\wedge^k\mathfrak{n}[1],\mathbb{W})^{\hat{\mathfrak{h}}}  \to H^{\bullet+k}(\mathfrak{h},\mathbb{W}),$$

\noindent which are related according to the next result.

\begin{theorem}\label{thm:naturality}
Let $\hat{\mathfrak{g}}\in \mathrm{Ext}(\mathfrak{g},\mathfrak{n})$ be an extension of $\mathfrak{g}$ by $\mathfrak{n}$. Assume that $(\mathbb{V},\rho)\in \mathrm{Rep}_{\infty}(\mathfrak{g})$ are $(\mathbb{W},\tau)\in \mathrm{Rep}_{\infty}(\mathfrak{h})$ and $(T,t):(\mathbb{W},\tau) \to (\mathbb{V},\rho)$ is a ruths morphism. Let $\hat{\mathfrak{h}}\in \mathrm{Ext}(\mathfrak{h},\mathfrak{n})$ be the associated pullback extension \eqref{eq:pullbackextension}. Then the following diagram commutes

\begin{equation}\label{naturality1}
    \begin{tikzcd}
    \underline{\Hom}^{\bullet}(\wedge^k\mathfrak{n}\left[1\right],\mathbb{V})^{\hat{\mathfrak{g}}}\arrow[r,"cw"]\arrow[d,"t^{\sharp}"]&H^{k+\bullet}(\mathfrak{g},\mathbb{V})\arrow[d,"T^*"]\\
      \underline{ \Hom}^{\bullet}(\wedge^k\mathfrak{n}\left[1\right],\mathbb{W})^{\hat{\mathfrak{h}}}\arrow[r,"cw"] &H^{k+\bullet}(\mathfrak{h},\mathbb{W}).
    \end{tikzcd}
\end{equation}

\noindent where $t^{\sharp}$ is given by \eqref{eq:tmapnaturality} and $T^*:H^{k+ \bullet}(\mathfrak{g},\mathbb{V})\to H^{k+\bullet}(\mathfrak{h},\mathbb{W})$ denotes the pullback map in $L_{\infty}$-cohomology.

\end{theorem}

\begin{proof}

Let $h:\mathfrak{g}\left[1\right]\to \hat{\mathfrak{
g}}\left[1\right]$ be a linear section of $\pi:\hat{\mathfrak{g}}\to \mathfrak{g}$ with curvature $K_{h}\in C^1(\mathfrak{g},\mathfrak{n}[1])$. As observed above, there is an induced section $\overline{h}:\mathfrak{h}\left[1\right]\to \hat{\mathfrak{h}}\left[1 \right]$ of $\pi_{\bullet}:\hat{\mathfrak{h}}\to \mathfrak{h}$ satisfying \eqref{eq:pullbacksection}. Let $K_{\overline{h}}\in C^1(\mathfrak{h},\mathfrak{n}[1])$ be its curvature. A simple computation shows that $T^*K_h=K_{\overline{h}}$.

For $f\in \underline{\Hom}^{\bullet}(\wedge^k\mathfrak{n}[1],\mathbb{V})^{\hat{\mathfrak{g}}}$, let us consider the cocycle $f_{K_h}:=f\circ K^{\wedge^k}_h\in   C^{k+\bullet}(\mathfrak{g},\mathbb{V})$ as in \eqref{eq:CWLcocycle}. Then:

\begin{align*}
    T^{*}f_{K_h}=&f\circ K_{h}^{\wedge^k}\circ T\\
    =&f\circ K_{h}\otimes\cdots \otimes K_{h}\circ \Delta^{k-1}\circ T\\
    =&f\circ K_{h}\otimes\cdots \otimes K_{h}\circ T^{\otimes k}\circ \Delta^{k-1}\\
\end{align*}

\noindent where in the last identity we have used that $T$ is a coalgebra morphism. Due to the interchange law  \eqref{eq:interchangelaw}, one has

\begin{align*}
K_{h}\otimes\cdots \otimes K_{h}\circ T^{\otimes k}=&(K_{h}\circ T)^{\otimes k}\\
=& (T^*K_h)^{\otimes k}\\
=& K^{\otimes k}_{\overline{h}}.
\end{align*}    
    
Hence we conclude   
    
 \begin{align*}   
    T^{*}f_{K_h}=&f\circ K^{\otimes k}_{\overline{h}}\circ \Delta^{k-1}\\
     =&f\circ K_{\overline{h}}^{\wedge_{\otimes}k}\\
   =& f_{K_{\overline{h}}}.
\end{align*}

As a consequence
\begin{align*}
    \left(T^*\circ cw\right)(f)=&\left[T^*(f_{K_h})\right]\\
    =&\left[t\circ f_{K_h}\circ T\right]\\
    =&\left[t\circ T^*f_{K_h}\right]\\
    =&\left[t\circ f_{K_{\overline{h}}}\right]\\
    =&\left[t\circ f \circ K_{\overline{h}}^{\wedge_{\otimes}k}\right]\\
    =&\left[t^{\sharp}(f)\circ K_{\overline{h}}^{\wedge_{\otimes}k}\right]\\
    =&cw(t^{\sharp}(f))\\
    =&(cw\circ t^{\sharp})(f),\\
\end{align*}
which shows that $T^*\circ cw= cw\circ t^{\sharp}$, as desired.
\end{proof}

Suppose now that $T:\mathfrak{h}\to \mathfrak{g}$ is an $L_{\infty}$-morphism and $(\mathbb{V},\rho)\in \mathrm{Rep}_{\infty}(\mathfrak{g})$. The pullback ruth $(T^*\rho,\mathbb{V})\in \mathrm{Rep}_{\infty}(\mathfrak{h})$ defined in \eqref{eq:pullbackruth} comes with a canonical ruth morphism $(T,\id_{\mathbb{V}}):(\mathbb{V},T^*{\rho})\to (\mathbb{V},\rho)$. In particular, if $\hat{\mathfrak{g}}\in \mathrm{Ext}(\mathfrak{g,\mathfrak{n}})$ and $\hat{\mathfrak{h}}\in \mathrm{Ext}(\mathfrak{h,\mathfrak{n}})$ is the pullback extension, then Theorem \ref{thm:naturality} gives rise to the following.

\begin{corollary}\label{cor:naturality}

If $T:\mathfrak{h}\to \mathfrak{g}$ is an $L_{\infty}$-quasi isomorphism, then the isomorphism $T^*:H^{\bullet}(\mathfrak{g},\mathbb{V})\to H^{\bullet}(\mathfrak{h},T^*\mathbb{V})$ intertwines Chern-Weil-Lecomte maps:

\begin{equation}\label{naturality1}
    \begin{tikzcd}
    \underline{\Hom}^{\bullet}(\wedge^k\mathfrak{n}\left[1\right],\mathbb{V})^{\hat{\mathfrak{g}}}\arrow[r,"cw"]\arrow[d,"\id"]&H^{k+\bullet}(\mathfrak{g},\mathbb{V})\arrow[d,"T^*"]\\
       \underline{\Hom}^{\bullet}(\wedge^k\mathfrak{n}\left[1\right],\mathbb{W})^{\hat{\mathfrak{h}}}\arrow[r,"cw"] &H^{k+\bullet}(\mathfrak{h},\mathbb{W}).
    \end{tikzcd}
\end{equation}

\end{corollary}

In other words, quasi-isomorphic $L_{\infty}$-algebras have the same characteristic classes defined via the Chern-Weil-Lecomte construction.

%%%%%%%%%%%%%%%%%%%%%%%%%%%%%%%
%%%%%%%%%%%%%%%%%%%%%%%%%%%%%%%%

\subsection{Examples}

In what follows we give several examples of characteristic classes which can be described by means of the Chern-Weil-Lecomte map associated to an extension of $L_{\infty}$-algebras.

\begin{example}[\textbf{Extensions of Lie algebras}]

Let us consider an extension of Lie algebras
 
$$ \begin{tikzcd}
    0\arrow[r]&\mathfrak{n}\arrow[r,"\iota"]&\hat{\mathfrak{g}}\arrow[r,"\pi"]&\mathfrak{g}\arrow[r]&0,
 \end{tikzcd}
 $$

\noindent together with a representation $\rho:\mathfrak{g}\to \mathfrak{gl}(V)$ where $V$ is a vector space. In this case, the Chern-Weil-Lecomte map defined by a splitting $h:\mathfrak{g}\to \hat{\mathfrak{g}}$ coincides with the ordinary Lecomte map \cite{Lecomte}. 
\end{example}

\begin{remark}\label{rmk:CWL2algebras}

In the special case of extensions of semistrict Lie $2$-algebras

 \begin{center}
 \begin{tikzcd}
    0\arrow[r]&\mathfrak{n}\arrow[r,"\iota"]&\hat{\mathfrak{g}}\arrow[r,"\pi"]&\mathfrak{g}\arrow[r]&0,
 \end{tikzcd}
\end{center}

\noindent with a splitting $h:\mathfrak{g}\to \hat{\mathfrak{g}}$, the curvature $K_h\in C^1(\mathfrak{g},\mathfrak{n}[1])$ has only four possibly non-zero components. These can be explicitly described in terms of higher skew-symmetric brackets and the underlying complexes $\partial:\mathfrak{h}\to \mathfrak{g}$ and $\hat{\partial}:\hat{\mathfrak{h}}\to \hat{\mathfrak{g}}$ corresponding to $\mathfrak{g}$ and $\hat{\mathfrak{g}}$, respectively. More concretely, one has

\begin{enumerate}

\item  $K_h:\mathfrak{g}_{-1}\times \mathfrak{g}_0\to \mathfrak{n}_{-1}$, with

$$K_h(u,x)= h_{-1}(D_xu)-D_{h_0(x)}h_{-1}(u).$$

\item  $K_h:\wedge^3\mathfrak{g}_0\to \mathfrak{n}_{-1}$, given by

$$K_h(x,y,z)= [h_0(x),h_0(y),h_0(z)]^3 - h_{-1}([x,y,z]^3).$$

\item  $K_h:\mathfrak{g}_{-1}\to \mathfrak{n}_{0}$, defined as

$$K_h(u)= (h_0\circ d-\hat{d}\circ h_{-1})(u).$$

\item  $K_h:\wedge^2\mathfrak{g}_{0}\to \mathfrak{n}_{0}$, where

$$K_h(x,y)= h_{0}[x,y]-[h_0(x),h_0(y)].$$

\end{enumerate}

Note also that if both $\mathfrak{g}$ and $\hat{\mathfrak{g}}$ are strict 2-term $L_{\infty}$-algebras, then the brackets of order 3 vanish and hence the curvature component $K_h:\wedge^3\mathfrak{g}_0\to \mathfrak{n}_{-1}$ above is necessarily zero as well.

\end{remark}

In what follows, we use Remark \ref{rmk:CWL2algebras} to give more examples of characteristic classes which can be described my means of Theorem \ref{thm:lecomte}.

\begin{example}[\textbf{Characteristic class of semistrict Lie 2-algebras}]

Since every semistrict 2-term $L_{\infty}$-algebra is quasi-isomorphic to a minimal one, it suffices to consider a semistrict 2-term $L_{\infty}$-algebra $\mathfrak{g}=\mathfrak{g}_{-1}\oplus \mathfrak{g}_0$ which is minimal. As discussed in Example \ref{ex:2termminimal}, there is an associated characteristic class $[\theta]\in H^3_{CE}(\mathfrak{g}_{0},\mathfrak{g}_{-1})$ where $\theta=[\cdot]^3$ and $\mathfrak{g}_0$ acts on $\mathfrak{g}_{-1}$ via the bracket of order 2. We will see now that $[\theta]$ can be described by means of the Chern-Weil-Lecomte approach. Let $\mathfrak{n}\subseteq \mathfrak{g}$ be the 2-term $L_{\infty}$-algebra with underlying complex $\mathfrak{g}_{-1}\to 0$. Similarly, the Lie algebra $\mathfrak{g}_0$ can be seen as an $L_{\infty}$-algebra with underlying complex $0\to \mathfrak{g}_0$. It turns out that $\mathfrak{g}$ is an extension of $\mathfrak{g}_0$ by $\mathfrak{n}$, which is given in terms of the underlying 2-term complexes, as follows:

\[
\begin{tikzcd}
  0 \arrow[r] & \mathfrak{g}_{-1} \arrow[d, "d=0"] \arrow[r, "\id"] & \mathfrak{g}_{-1} \arrow[d, "d=0"] \arrow[r, "0"] & 0 \arrow[d, "d=0"] \arrow[r] & 0 \\
  0 \arrow[r] & 0 \arrow[r, "0"] & \mathfrak{g}_0 \arrow[r, "\id"] & \mathfrak{g}_0 \ar[r] & 0.
\end{tikzcd}
\]

\noindent The representation $\mathfrak{g}_0\to \mathrm{End}(\mathfrak{g}_{-1})$ can be regarded as a ruth of $\mathfrak{g}_0$ on the cochain complex $0\to \mathfrak{g}_{-1}$. There is a canonical splitting given by $h_{-1}:=0$ and $h_{0}:=\id:\mathfrak{g}_0\to \mathfrak{g}_0$, which according to Remark \ref{rmk:CWL2algebras}, has curvature $K_h:S(\mathfrak{g}_0[1])\to \mathfrak{n}[1]$ with only one non-trivial component, namely

\begin{align*}
K^3_h:\wedge^3\mathfrak{g}_0&\to \mathfrak{g}_{-1}\\ 
(x,y,&z)\mapsto \theta(x,y,z).
\end{align*}

\noindent  It is clear now that $cw(\id)=[\theta]\in H^3_{CE}(\mathfrak{g}_{0},\mathfrak{g}_{-1})$ where $\id\in\underline{\mathrm{Hom}}^0(\mathfrak{n}[1],\mathfrak{n}[1])$.

\end{example}

\begin{example}[\textbf{The Atiyah Lie 2-algebra}]\label{ex:AtiyahLie2algebra}

Let $H\in \Omega^3_{cl}(M)$ be a closed 3-form on a manifold $M$. The \textbf{Atiyah Lie 2-algebra} is the minimal 2-term $L_{\infty}$-algebra $\mathfrak{a}(M,H)$ whose underlying complex is $0:C^{\infty}(M)\to \mathfrak{X}(M)$, the brackets of order 2 are given by $[\cdot,\cdot]:\mathfrak{X}(M)\times \mathfrak{X}(M)\to \mathfrak{X}(M)$ and $\mathfrak{X}(M)\times C^{\infty}(M)\to C^{\infty}(M)$ defined as the Lie derivative. The bracket of order 3 is 

$$\mathfrak{X}(M)^{\otimes 3}\to C^{\infty}(M); (X_1,X_2,X_3)\mapsto -H(X_1,X_2,X_3).$$

\noindent  From the previous example one concludes that every degree 3 cohomology class in $H^3_{dR}(M)$ can be described as the characteristic class of a minimal semi-strict $L_{\infty}$-algebra, and hence as an element in the image of its Chern-Weil-Lecomte map. The Atiyah Lie 2-algebra is related to symmetries of $\mathbb{S}^1$-bundle gerbes. Indeed, every $\mathbb{S}^1$-bundle gerbe with connective structure defines a quasi-isomorphism between the Lie 2-algebra of infinitesimal symmetries of the gerbe and $\mathfrak{a}(M,H)$, see \cite{DjounvounaKrepski}. We will discuss more about gerbes in Section \ref{sec:CWL2bundles}.

\end{example}

\begin{example}[\textbf{Abelian extensions}]\label{ex:abelianextensions}

An extension of $L_{\infty}$-algebras

$$
\begin{tikzcd}
    0\arrow[r]&\mathfrak{n}\arrow[r,"\iota"]&\hat{\mathfrak{g}}\arrow[r,"\pi"]&\mathfrak{g}\arrow[r]&0,
 \end{tikzcd}
$$

\noindent is called \textbf{abelian} if $\ad(x\otimes y)=0$ for every $x,y\in S(\mathfrak{n}[1])$. Here $\ad$ denotes the adjoint representation up to homotopy of $\mathfrak{n}$, see \ref{ex:adjointruth}. Let $h:\mathfrak{g}\to \hat{\mathfrak{g}}$ a splitting with curvature $K_h=\sum_{j}K^j_h:S(\mathfrak{g}[1])\to \mathfrak{n}[1]$ where $K^j_h:\mathfrak{g}[1]^{\otimes j}\to \hat{\mathfrak{n}}[1]$. There is a well-defined representation up to homotopy $\rho:S(\mathfrak{g}[1])\otimes \mathfrak{n}[1]\to \mathfrak{n}[1]$ given by $\rho(x\otimes y):=\ad(h(x)\otimes y)$. In particular, for $k=1$ in \eqref{CWLmap} there is a Chern-Weil-Lecomte map 

$$cw:\underline{\Hom}^{\bullet}(\mathfrak{n}[1],\mathfrak{n}[1])^{\hat{\mathfrak{g}}}  \to H^{\bullet+1}(\mathfrak{g},\mathfrak{n}[1]).$$

\noindent The identity $\id\in \underline{\Hom}^0(\mathfrak{n}[1],\mathfrak{n}[1])^{\hat{\mathfrak{g}}}$ yields the cohomology class $cw(\id)=[K_h]\in H^1(\mathfrak{g},\mathfrak{n}[1])$ of the curvature of $h$. In the special case of central extensions of \emph{Lie algebras}, the curvature $K_h=K^2_h:\wedge^2\mathfrak{g}\to \mathfrak{n}$ since the symmetric brackets $\lambda_j=0$ for every $j\neq 2$. Hence, the class $[K_h]\in H^2_{CE}(\mathfrak{g},\mathfrak{n})$ is the ordinary Chevalley-Eilenberg cohomology class classifying a central extension of Lie algebras. 

\end{example}

\begin{example}[\textbf{Kostant-Souriau central extension of a 2-plectic manifold}]\label{ex:CWL2plectic}

A \textbf{2-plectic structure} \cite{rogers1} on a manifold $M$ is a closed 3-form $\omega\in \Omega^3_{cl}(M)$ which is non-degenerate, i.e. $i_X\omega=0$ if and only if $X=0$. A 1-form $\alpha\in \Omega^1(M)$ is Hamiltonian if there exists a vector field $X_{\alpha}\in \mathfrak{X}(M)$ with $d\alpha=-i_{X_{\alpha}}\omega$. We denote by $\Omega^1_{\mathrm{Ham}}(M)$ and $\mathfrak{X}_{\mathrm{Ham}}(M)$ the sets of Hamiltonian 1-forms and Hamiltonian vector fields, respectively. The vector field $X_{\alpha}$ is unique and called the Hamiltonian vector field of $\alpha$. In particular, there is a well-defined map 

$$p:\Omega^1_{\mathrm{Ham}}(M)\to \mathfrak{X}_{\mathrm{Ham}}(M); \alpha \mapsto X_{\alpha}.$$

\noindent The bracket of $\alpha,\beta\in \Omega^1_{\mathrm{Ham}}(M)$ is defined as the 1-form $\{\alpha,\beta\}_{\omega}:=i_{X_{\beta}}i_{X_{\alpha}}\omega\in \Omega^1_{\mathrm{Ham}}(M)$. There is a semistrict Lie 2-algebra $L_{\infty}(M,\omega)$ whose underlying complex is given by the de Rham differential $d:C^{\infty}(M)\to \Omega^1_{\mathrm{Ham}}(M)$, the only non-zero skew-symmetric brackets are $\{\cdot,\cdot\}_{\omega}:\Omega^1_{\mathrm{Ham}}(M)\times \Omega^1_{\mathrm{Ham}}(M)\to \Omega^1_{\mathrm{Ham}}(M)$ and 

\begin{align*}
[\cdot]^3:\Omega^1_{\mathrm{Ham}}(M)\times \Omega^1_{\mathrm{Ham}}(M)\times \Omega^1_{\mathrm{Ham}}&(M)\to C^{\infty}(M)\\
(\alpha,\beta,\gamma)\mapsto & \qquad \omega(X_{\alpha},X_{\beta},X_{\gamma}).
\end{align*}

\noindent The Lie algebra $\mathfrak{X}_{\mathrm{Ham}}(M)$ is regarded as a Lie 2-algebra with underlying complex $0\to \mathfrak{X}_{\mathrm{Ham}}(M)$. As shown in \cite{rogers1}, one can see $L_{\infty}(M,\omega)$ as a central extension of $\mathfrak{X}_{\mathrm{Ham}}(M)$ by the abelian Lie 2-algebra $\mathfrak{n}$ whose underlying complex is $d:C^{\infty}(M)\to \Omega^1_{cl}(M)$. The extension is given by

\begin{equation}\label{eq:extension2plectic}
\begin{tikzcd}
  0 \arrow[r] & C^{\infty}(M) \arrow[d, "d"] \arrow[r, "\id"] & C^{\infty}(M) \arrow[d, "d"] \arrow[r, "0"] & 0 \arrow[d, "d=0"] \arrow[r] & 0 \\
  0 \arrow[r] & \Omega^1_{cl}(M) \arrow[r, "\iota"] & \Omega^{1}_{\mathrm{Ham}}(M) \arrow[r, "p"] & \mathfrak{X}_{\mathrm{Ham}}(M) \ar[r] & 0.
\end{tikzcd}
\end{equation}

\noindent A splitting $h_0:\mathfrak{X}_{\mathrm{Ham}}(M)\to \Omega^1_{\mathrm{Ham}}(M)$ induces a splitting $h$ of \eqref{eq:extension2plectic} with $h_{-1}:=0$. Hamiltonian vector fields act on $C^{\infty}(M)\to \Omega^1_{cl}(M)$ by the Lie derivative, hence as seen in Example \ref{ex:abelianextensions}, one has a Chern-Weil-Lecomte map 

$$cw:\underline{\Hom}^{\bullet}(\mathfrak{n}[1],\mathfrak{n}[1])^{\hat{\mathfrak{g}}}  \to H^{\bullet+1}(\mathfrak{X}_{\mathrm{Ham}}(M),\mathfrak{n}[1]),$$

\noindent with $cw(\id)=[K_h]$. There are only two non-vanishing components of the curvature $K_h$, namely the ordinary curvature of the splitting $h_0$

$$K_h:\wedge^2\mathfrak{X}_{\mathrm{Ham}}(M)\to \Omega^1_{\mathrm{Ham}}(M),$$

\noindent and $K^{-3}_h:\wedge^3\mathfrak{X}_{\mathrm{Ham}}(M)\to C^{\infty}(M)$ which is given by the 2-plectic structure

$$K^{-3}_h(X,Y,Z)=\omega(X,Y,Z).$$

\end{example}

\begin{example}[\textbf{Courant algebroids}]\label{ex:CWLsevera}

Let $E\to M$ be a Courant algebroid with pairing $\langle\cdot,\cdot\rangle:E\times E\to \R$, skew-symmetric bracket $\Cour{\cdot,\cdot}:\Gamma(E)\times \Gamma(E)\to \Gamma(E)$ and anchor map $\rho:E\to TM$, see \cite{LWX}. We say that $E$ is \textbf{exact} if the sequence of vector bundles

\begin{equation}\label{eq:exactCourant}
\begin{tikzcd}
    0\arrow[r]&T^*M\arrow[r,"\rho^*"]&E\arrow[r,"\rho"]&TM\arrow[r]&0,
 \end{tikzcd}
\end{equation}

\noindent is exact. It is well known that an exact Courant algebroid has an isotropic splitting $\nabla:TM\to E$, which induces a closed 3-form $H\in \Omega^3_{cl}(M)$ defined as 

$$H(X,Y,Z):=\langle\Cour{\nabla X,\nabla Y},\nabla Z\rangle,$$

\noindent for every $X,Y,Z\in \mathfrak{X}(M)$. The cohomology class $[H]\in H^3_{dR}(M)$ is called the \textbf{\v{S}evera class} of the exact Courant algebroid $E$. We will see that the \v{S}evera class can be described by means of an appropriate Chern-Weil-Lecomte map. Indeed, as shown in \cite{RoytenbergWeinstein}, every Courant algebroid has an associated semistrict 2-term $L_{\infty}$-algebra $\hat{\mathfrak{g}}$ with underlying complex 

$$\mathcal{D}:C^{\infty}(M)\to \Gamma(E); \langle \mathcal{D}f,e\rangle=\frac{1}{2}\Lie_{\rho(e)}f,$$

\noindent and higher skew-symmetric brackets 

$$[e_1,e_2]^2:=\Cour{e_1,e_2}; \quad [e,f]^2:=\langle e,\mathcal{D}e\rangle; \quad [e_1,e_2,e_3]^3:=-T(e_1,e_2,e_3),$$

\noindent where $T(e_1,e_2,e_3):=\frac{1}{3}\langle \Cour{e_1,e_2},e_3\rangle + c.p.$

\noindent Let $\mathfrak{g}:=\mathfrak{X}(M)$ be the Lie algebra of vector fields on $M$ viewed as a strict Lie 2-algebra with complex $0\to \mathfrak{X}(M)$.  Note that $\hat{\mathfrak{g}}$ is an extension of $\mathfrak{X}(M)$ by the \emph{abelian} Lie 2-algebra $\mathfrak{n}$ whose underlying complex is $\frac{1}{2}d:C^{\infty}(M)\to \Omega^1(M)$. In terms of the corresponding 2-term complexes, the extension is given by

\begin{equation}\label{eq:extensionCourant}
\begin{tikzcd}
  0 \arrow[r] & C^{\infty}(M) \arrow[d, "\frac{1}{2}d"] \arrow[r, "\id"] & C^{\infty}(M) \arrow[d, "\mathcal{D}"] \arrow[r, "0"] & 0 \arrow[d, "d=0"] \arrow[r] & 0 \\
  0 \arrow[r] & \Omega^1(M) \arrow[r, "\rho^*"] & \Gamma(E) \arrow[r, "\rho"] & \mathfrak{X}(M) \ar[r] & 0.
\end{tikzcd}
\end{equation}

\noindent  An isotropic splitting $\nabla:TM\to E$ of \eqref{eq:exactCourant} induces a splitting $h:\mathfrak{g}\to \hat{\mathfrak{g}}$ of \eqref{eq:extensionCourant} whose homogeneous components are $h_{-1}:=0$ and $h_0:=\nabla:\mathfrak{X}(M)\to \Gamma(E)$. In particular, there is an induced representation of vector fields $\mathfrak{X}(M)$ on $C^{\infty}(M)\to \Omega^{1}(M)$ given by the Lie derivative. As observed in Example \ref{ex:abelianextensions}, one has a Chern-Weil-Lecomte map

$$cw:\underline{\Hom}^{\bullet}(\mathfrak{n}[1],\mathfrak{n}[1])^{\hat{\mathfrak{g}}}  \to H^{\bullet+1}(\mathfrak{g},\mathfrak{n}[1]),$$

\noindent which maps $\id\in \underline{\Hom}^0(\mathfrak{n}[1],\mathfrak{n}[1])^{\hat{\mathfrak{g}}}$ to the cohomology class $cw(\id)=[K_h]\in H^1(\mathfrak{g},\mathfrak{n}[1])$ of the curvature of $h$. Due to Remark \ref{rmk:CWL2algebras}, the curvature $K_h\in C^1(\mathfrak{g},\mathfrak{n}[1])$ has only two non-zero components, namely the ordinary curvature of the isotropic splitting

$$K_h:\wedge^2\mathfrak{X}(M)\to \Omega^1(M); K_h(X,Y)=\nabla[X,Y]-\Cour{\nabla X, \nabla Y},$$

\noindent and the curvature component coming from the bracket of order 3

$$K_h:\wedge^3\mathfrak{X}(M)\to C^{\infty}(M); K_h(X,Y,Z)=[\nabla X, \nabla Y, \nabla Z]^3.$$

\noindent It follows from the definition of the bracket $[\cdot,\cdot]^3$, that

\begin{align*}
[\nabla X, \nabla Y, \nabla Z]^3=&-T(\nabla X, \nabla Y, \nabla Z)\\
=& -\frac{1}{3}\langle\Cour{\nabla X, \nabla Y}, \nabla Z \rangle + c.p.\\
=& -\frac{1}{3}H(X,Y,Z) + c.p.\\
=& -H(X,Y,Z).
\end{align*}

\noindent In particular, the \v{S}evera class $[H]\in H^3_{dR}(M)$ appears as one of the components of the cohomology class $[K_h]\in H^1(\mathfrak{g},\mathfrak{n}[1])$.
\end{example}

\begin{remark}[\textbf{2-plectic structures vs Courant algebroids}]

Examples \ref{ex:CWL2plectic} and \ref{ex:CWLsevera} are related by \v{S}evera's classification of exact Courant algebroids. Every 2-plectic structure yields an exact Courant algebroid (unique up to isomorphism) together with an isotropic splitting and whose \v{S}evera class is precisely the 2-plectic structure \cite{rogers1}. Hence, the Chern-Weil-Lecomte maps of Examples \ref{ex:CWL2plectic} and \ref{ex:CWLsevera} commute with the construction of  the exact Courant algebroid defined by a 2-plectic manifold.

\end{remark}

%%%%%%%%%%%%%%%%%%%%%%%%
%%%%%%%%%%%%%%%%%%%%%%%%%
%%%%%%%%%%%%%%%%%%%%%%%%

\section{Chern-Weil map for principal 2-bundles}\label{sec:CWL2bundles}

In this section we show how the Chern-Weil-Lecomte of an $L_{\infty}$-algebra extension can be applied to introduce a Chern-Weil map for principal 2-bundles over Lie groupoids.

%-short definition, quote Waldorf
%- connections (quote waldorf and people from india)
%- atiyah sequence
%- Chern-Weil
%- morita invariance
%- examples: usual chern-weil, equivariant case, etale case and Camille, G-gerbes?

\subsection{Principal 2-bundles over a Lie groupoid}

We start by introducing some terminology that shall be used throughout the whole section. For a Lie 2-group we use the notation $\mathbb{G}=(G_1\rightrightarrows G_0)$ and similarly a Lie groupoid will be denoted by $\mathbb{X}=(X_1\rightrightarrows X_0)$.

\begin{definition}
Let $\mathbb{G}$ be a Lie 2-group and $\mathbb{X}$ be a Lie groupoid. A \textbf{right action} of $\mathbb{G}$ on $\mathbb{X}$ is a Lie groupoid morphism $a:\mathbb{X}\times \GG\to \mathbb{X}$ which defines Lie group actions $a_1:X_1\times G_1\to X_1$ and $a_0:X_0\times G_0\to X_0$ on arrows and objects, respectively. Left actions are defined in a similar manner.
\end{definition}

Right and left actions are denoted $\mathbb{X}\curvearrowleft \GG$ and $\GG\curvearrowright \mathbb{X}$, respectively. Similarly,  $xg$ and $gx$ denote the corresponding actions of $g\in \GG$ on $x\in \mathbb{X}$. 

\begin{remark}\label{rmk:structuralequivariance}
Given a right action of $\mathbb{G}$ on $\mathbb{X}$,  the structural maps of $\mathbb{X}$ are equivariant with respect to the structural maps of $\mathbb{G}$. That is, if $g\in G_1, g_0\in G_0$ and $x\in X_1, x_0\in X_0$, then

$$s_X(xg)=s_X(x)s_G(g),\quad t_X(xg)=t_X(x)t_G(g),\quad 1_{x_0g_0}=1_{x_0}1_{g_0}$$

\noindent and for every composable arrows $(x,y)\in X_2$ and $(g,g') \in G_2$ the following interchange law is fulfilled

 \begin{equation}\label{eq:2interchangelaw}
     (x*y)(g*h)=(xg)*(yh)
 \end{equation} where $x*y$ and $g*h$ denote the groupoid multiplication in $\mathbb{X}$ and $\GG$, respectively. 
\end{remark}

\begin{remark}\label{rmk:2groupaut}

It is worth observing that a right action $\mathbb{X}\curvearrowleft \GG$ is not necessarily given by Lie groupoid automorphism. However, the Lie group $G_0$ acts on $\mathbb{X}$ by Lie groupoid automorphisms via $1_{g_0}:\mathbb{X}\to \mathbb{X}$ for every $g_0\in G_0$. Also, every arrow $g\in G_1$ defines a map $\tau_g:X_0\to X_1; x_0\mapsto 1_{x_0}g$ which is a natural isomorphism $\tau_g:1_{s(g)}\Rightarrow 1_{t(g)}$. Hence $\mathbb{X}\curvearrowleft \GG$ can be seen as a 2-group morphism $\GG\to 2\mathrm{Aut}(\mathbb{X})$ where $2\mathrm{Aut}(\mathbb{X})$ denotes the 2-group of automorphisms of $\mathbb{X}$.

\end{remark}

Let us see some examples of Lie 2-group actions.

\begin{example}
A usual action of a Lie group $G$ on a manifold $X$ can be seen as an action of the Lie 2-group $G\rightrightarrows G$ on the unit groupoid $X\rightrightarrows X$.
\end{example}

\begin{example}
Every Lie 2-group $\mathbb{G}$ acts on itself by right translations. More generally,  any Lie 2-subgroup $\mathbb{H}$ of $\mathbb{G}$ acts on $\mathbb{G}$ by right translations.
\end{example}

\begin{example}\label{ex:2adjoint}
Any Lie 2-group $\mathbb{G}$ acts on itself on the left by conjugation at both the level of arrows and the level of objects. This induces an action $\mathbb{G}\curvearrowright (\mathfrak{g}_1\rightrightarrows \mathfrak{g}_0)$ on its Lie 2-algebra, called the \textbf{adjoint representation} of $\GG$.
\end{example}

\begin{example}\label{ex:tangent-lifting-2-action}

A right action of $\mathbb{X}\curvearrowleft \GG$ induces a right action $T\mathbb{X}\curvearrowleft\mathbb{G}$ on \textbf{the tangent groupoid} $T\mathbb{X}=(TX_1\rightrightarrows TX_0)$ given by the tangent lift of the action of $G_i$ on $X_i$, for $i=0,1$. 

\end{example}

\begin{definition}
Let $\mathbb{X}=(X_1\rightrightarrows X_0)$ and $\mathbb{E}=(E_1\rightrightarrows E_0)$ be two Lie groupoids. A Lie groupoid morphism $\pi: \mathbb{E}\to \mathbb{X}$ is said to be  a \textbf{fibration} if it satisfies the following conditions

\begin{enumerate}
    \item the base map $\pi_0:E_0\to X_0$ is a surjective submersion, and 
    \item the map $\tilde{\pi}:E_1\to X_1\pullback{s}{\pi_0}E_0, e\mapsto (\pi_1(e),s_{E}(e))$ is a surjective submersion.
\end{enumerate}

\end{definition}

\begin{definition}\label{def:2bundle}

 Let $\pi:\p\to \mathbb{X}$ be a groupoid fibration together with a right action $\mathbb{P}\curvearrowleft \GG$. We say that the action is  \textbf{principal along $\pi$} if

\begin{equation}\label{eq:principality}
   \mathbb{P}\times \GG\to \mathbb{P}\times_{\mathbb{X}}\mathbb{P}; (p,g)\mapsto (p,pg)
   \end{equation}

\noindent is a Lie groupoid isomorphism. In this case, we say that $\mathbb{P}\to \mathbb{X}$ is a \textbf{principal $\GG$-2-bundle}. 
 \end{definition}

 Our main motivation to consider principal 2-bundles as in Definition \ref{def:2bundle} comes from the study of principal actions on differentiable stacks as in \cite{BNZ}, where principal actions of stacky Lie groupoids are studied. Since every Lie groupoid $\mathbb{X}=(X_1\to X_0)$ defines a quotient stack $[X_0/X_1]$ \cite{behrend-xu}, principal 2-bundles $\mathbb{P}\curvearrowleft \mathbb{G}\to \mathbb{X}$ in the sense of Definition \ref{def:2bundle} define principal actions of the stacky Lie group $[G_0/G_1]$ along the stack morphism $[P_0/P_1]\to [X_0/X_1]$ induced by $\pi:\p\to \mathbb{X}$. In \cite{BNZ}, a stacky Lie group is not necessarily strict and actions are weak. Hence, the notion of principal 2-bundle considered here can be seen as a strict version of a stacky principal bundle as in \cite{BNZ}. Several notions of principal 2-bundles have been studied by different authors, in \cite{Wockel} weak 2-groups act in a principal manner along a map $\mathcal{P}\to M$ between smooth 2-spaces. Here the base $M$ is a manifold seen as a smooth 2-space. Similarly, in \cite{waldorf1}, Waldorf studies principal 2-bundles as in Definition \ref{def:2bundle} with $\mathbb{X}$ being the unit Lie groupoid of a manifold, however principality is required in a weaker sense. Also, principal 2-bundles as in Definition \ref{def:2bundle} are studied in \cite{CCP}.

%\noindent As in the case of principal bundles, the fibers of a principal 2-bundle are isomorphic to the Lie 2-group $\mathbb{G}$. For a principal 2-bundle $(\p,\pi,\mathbb{X},\mathbb{G})$ we call $\p$  \textbf{the total space}, $\mathbb{X}$  \textbf{the base space} and $\mathbb{G}$ \textbf{the structural 2-group}. 

 \begin{example}[\textbf{Ordinary principal bundles}]\label{ex:2bundlebyusualbundle}
 Let $P\to X$ be a usual principal $G$-bundle over a manifold. Then, viewing manifolds as unit groupoids yields a principal 2-bundle $\mathbb{P}\curvearrowleft \GG\to \mathbb{X}$.
 \end{example}

 \begin{example}[\textbf{Principal 2-bundles over manifolds}]\label{2bundleovermanifold}

A \textbf{principal 2-bundle over a manifold} is a principal 2-bundle $\mathbb{P}\curvearrowleft \GG \to (X\rr X)$ where the base space is the unit groupoid of a manifold $X$. A more general notion of principal 2-bundle over a manifold was introduced in \cite{Wockel}, where principality is defined in a weak sense, i.e. the map \eqref{eq:principality} is a Morita map rather than an isomorphism.

  \end{example}
 
 \begin{example}[\textbf{Principal bundles over Lie groupoids}]\label{bundleovergroupoid}
 Let $\pi_0:P_0\to X_0$ be a principal $G$-bundle  with a left action of a Lie groupoid $\mathbb{X}=(X_1\rightrightarrows
  X_0)$ along $\pi_0$ which commutes with the right action of $G$. The transformation groupoid $\p:=X_1\ltimes P_0\rightrightarrows P_0$ is a principal $\GG$-2-bundle over $\mathbb{X}$ with $\GG=(G\rr G)$ defined as the unit groupoid of $G$. These kind of principal bundles are studied in \cite{LGTX}. 
 \end{example}

 \begin{example}[\textbf{Homogeneous Lie groupoids}]
 Let $\mathbb{G}$ be a Lie 2-group and $\mathbb{H}$ a Lie 2-subgroup such that $H_i\subseteq G_i$ is a closed subgroup for $i=0,1$. Then $\mathbb{G}$ is a principal $\mathbb{H}$-2-bundle over the Lie groupoid, $G_1/H_1\rightrightarrows G_0/H_0$.
 \end{example}

 \begin{example}[\textbf{Pullbacks}]
 
 Let $\phi:\mathbb{Y}\to \mathbb{X}$ be a Lie groupoid morphism. If $\mathbb{P}\curvearrowleft \GG\to \mathbb{X}$ is a principal 2-bundle, then the pullback bundles $\phi^*_1P_1\to Y_1$ and $\phi^*_0P_0\to X_0$ define a principal 2-bundle $\phi^*\mathbb{P}\curvearrowleft \GG\to \mathbb{Y}$ called the \textbf{pullback} 2-bundle of $\mathbb{P}$ by $\phi:\mathbb{Y}\to \mathbb{X}$.
 
\end{example}

\begin{example}[\textbf{$\mathbb{S}^1$-gerbes}]\label{ex:2bundlesgerbes}

Every surjective submersion $q:Y\to X$ defines a Lie groupoid $Y\times_{X}Y\to Y$ called the \textbf{submersion groupoid}. The nerve of the submersion groupoid yields a simplicial manifold $(Y_{n})_n$ whose face maps are denoted by $d_i:Y_n\to Y_{n-1}, i=0,1,...,n$. An \textbf{$\mathbb{S}^1$-gerbe} over a manifold $X$ is given by a triple $\mathcal{G}=(Y,P,\mu)$ where $q:Y\to X$ is a surjective submersion, $\pi:P\to Y\times_{X}Y$ is a principal $\mathbb{S}^1$-bundle and $\mu:d^*_0P\otimes d^*_2P\to d^*_1P$ is a bundle isomorphism over $Y_{3}$ called the \textbf{gerbe multiplication}, satisfying 

$$d^*_1\mu\circ(d^*_3\mu \otimes \id)=d^*_2\mu \circ (\id \otimes d^*_0\mu),$$

\noindent over $Y_{4}$. One can see that an $\mathbb{S}^1$-gerbe determines a principal 2-bundle with 2-group $\mathbb{S}^1\rr \{*\}$. Indeed, there is a natural Lie groupoid $P\rr Y$ with source and target maps $s:=d_0\circ \pi$, $t:=d_1\circ \pi$ and multiplication $m:P\times_YP\to P$ defined by 

$$\mu((x,y,z),p_1\otimes p_2):=((x,y,z),m(p_1,p_2)).$$

\noindent One can see that the action of $\mathbb{S}^1$ on $P$ defines an action of the Lie 2-group $\mathbb{S}^1\rr \{*\}$ on $P\rr Y$, making $\p:=(P\rr Y)$ a principal 2-bundle over the submersion groupoid $Y\times_{X}Y\to Y$. We refer to $\p\to Y\times_{X}Y$ as the \textbf{principal 2-bundle underlying} the $\mathbb{S}^1$-gerbe $\mathcal{G}$.

\end{example}

 %%%%%%%%%%%%%%%%
 %%%%%%%%%%%%%%%%
 
 \subsection{Equivariant VB-groupoids and LA-groupoids}

 Let $\mathbb{X}=(X_1\rr X_0)$ be a Lie groupoid. A \textbf{VB-groupoid} $\pi:\mathbb{E}\to \mathbb{X}$ is a Lie groupoid $\mathbb{E}=(E_1\rr E_0)$ where $E_1\to X_1$ and $E_0\to X_0$ are vector bundles and all the structural maps of $\mathbb{E}$ are vector bundle morphisms over the corresponding structural maps of $\mathbb{X}$. An \textbf{LA-groupoid} is a VB-groupoid $\pi:\mathbb{E}\to \mathbb{X}$ where $E_1\to X_1$ and $E_0\to X_0$ are Lie algebroids and all the structural maps are Lie algebroid morphisms. See \cite{mackenzie} for more details.
 
Let $\GG$ be a Lie 2-group. A VB-groupoid $\pi:\mathbb{E}\to \mathbb{X}$ is called \textbf{$\GG$-equivariant} if there is a left action $\GG\curvearrowright \mathbb{E}$ by vector bundle automorphisms. As a consequence, $E_1\to X_1$ and $E_0\to X_0$ are $G_1$-equivariant and $G_0$-equivariant, respectively. It follows from Remark \ref{rmk:2groupaut} that $G_0\curvearrowright \mathbb{E}$ is an action by VB-groupoid automorphisms. In a similar manner, one can introduce $\GG$-equivariant LA-groupoids and in this case $G_0\curvearrowright \mathbb{E}$ is an action by LA-groupoid automorphisms.

 \begin{example}\label{ex:equivarianttangentVB}
 
 Let $\GG \curvearrowright \mathbb{X}$ be a left action. Then the tangent LA-groupoid $T\mathbb{X}\to \mathbb{X}$ is $\GG$-equivariant with the action defined in Example \ref{ex:tangent-lifting-2-action}. 
 \end{example}
 
 \begin{example}\label{ex:preadjointVB}
 Let $\p\curvearrowleft \GG\to \mathbb{X}$ be a principal 2-bundle. Let $\mathbb{g}=(\mathfrak{g_1}\rr \mathfrak{g_0})$ be the Lie 2-algebra of $\GG$ equipped with the adjoint representation of Example \ref{ex:2adjoint}. The trivial VB-groupoid $\p\times \mathbb{g}\to \p$ is $\GG$-equivariant. Viewing $P_1\times \mathfrak{g}_1\to P_1$ and $P_0\times \mathfrak{g}_0\to P_0$ as bundles of Lie algebras, then $\p\times \mathbb{g}\to \p$ is actually a $\GG$-equivariant LA-groupoid.
 \end{example}

 Let $\pi:\mathbb{E}\to \mathbb{X}$ be a $\GG$-equivariant VB-groupoid. If $G_1\curvearrowright E_1$ and $G_0\curvearrowright E_0$ are principal actions, then $G_1\curvearrowright X_1$ and $G_0\curvearrowright X_0$ are principal as well and both $E_1/G_1\to X_1/G_1$ and $E_0/G_0\to X_0/G_0$ are vector bundles. It is easy to see that in this case, all groupoid structures also descend to the quotient, yielding a quotient VB-groupoid $\overline{\pi}:\mathbb{E}/\GG\to \mathbb{X}/\GG$. A similar result holds for $\GG$-equivariant LA-groupoids.
 
The following examples of quotients LA-groupoids associated to $\GG$-equivariant ones are particularly important in this work. See also \cite{CCP}, where both examples are discussed.
 
 \begin{example}\label{ex:adjointVB}
 
 Let us consider the $\GG$-equivariant LA-groupoid $\p\times \mathbb{g}\to \p$ of Example \ref{ex:preadjointVB}. The corresponding quotient LA-groupoid is denoted $\Ad(\p)\to \mathbb{X}$ and referred to as the \textbf{adjoint LA-groupoid} of $\p$.
 
 \end{example}
 
 \begin{example}\label{ex:atiyahVB}
 
 Let $\p\curvearrowleft \GG\to \mathbb{X}$ be a principal 2-bundle. Consider the tangent bundle $T\p\to \p$ equipped with the $\GG$-equivariant LA-groupoid structure of Example \ref{ex:equivarianttangentVB}. The corresponding quotient LA-groupoid is denoted $\At(\p)\to \mathbb{X}$ and called the \textbf{Atiyah LA-groupoid} of $\p$. 
 
 \end{example}

As shown in \cite{OrtizWaldron}, every VB-groupoid has an associated 2-vector space of sections, which in the case of an LA-groupoid it is also a Lie 2-algebra. We proceed now to recall the definition of the Lie 2-algebra of sections of an LA-groupoid following \cite{OrtizWaldron}.
 
 Let $\pi:\mathbb{E}\to \mathbb{X}$ be a VB-groupoid. A \textbf{multiplicative section} of $\mathbb{E}$ is a Lie groupoid morphism $e:\mathbb{X}\to \mathbb{E}$ satisfying $\pi\circ e=\id_{\mathbb{X}}$. The space of multiplicative sections is denoted by $\Gamma_{mult}(\mathbb{E})$. Multiplicative sections form a category in which morphisms are natural transformations $\tau:e\Rightarrow e'$ satisfying $1_{\pi} \bullet \tau=1_{\id_{\mathbb{X}}}$, where $\bullet$ denotes the horizontal composition of natural transformations:
 
\[\begin{tikzcd}
\mathbb{X} \arrow[rr, bend left=50,""{name=U, below}, "{e'}" description]
\arrow[rr, bend right=50,""{name=D}, "{e}" description]
&&
\mathbb{E}\arrow[Rightarrow, from=D,to=U,"\tau"]\arrow[rr, bend left=50,""{name=D, below}, "{\pi}" description]
\arrow[rr, bend right=50,""{name=U}, "{\pi}" description]&&\mathbb{X} \arrow[Rightarrow,"1_{\pi}", from=U, to=D]&=& \mathbb{X}\arrow[rr, bend left=50, "{\id_{\mathbb{X}}}" description,""{name=U, below}]
\arrow[rr, bend right=50, "{\id_{\mathbb{X}}}" description,""{name=D}]&&\mathbb{X}.\arrow[Rightarrow,"1_{\id_{\mathbb{X}}}", from=D, to=U]
\end{tikzcd}
\]
% 
%\[\xymatrix@!@C=50pt{\mathbb{X} \ar@/_1.1pc/[r]|{\overset{ }{\, e\, }}="a"  \ar@/^1.1pc/[r]|{\, %{e'}_{\,} \, }="b"
%                  & \mathbb{E} \ar@/_1.1pc/[r]|{\overset{ }{\, q_{\mathbb{E}}\, }}="c"  %\ar@/^1.1pc/[r]|{\, {\pi} \, }="d" 
%							  	  & \mathbb{X} \ar@{=>}^-{\tau}"a";"b"-<0pt,4pt>\ar@{=>}^-%{1_{\pi}}"c";"d"-<0pt,4pt>}
%\quad = \quad
%\xymatrix@!@C=50pt{\mathbb{X} \ar@/_1.1pc/[r]|{{ }{\, \id_{\mathbb{X}}\, }}="a"  \ar@/^1.1pc/[r]|{\, \id_{\mathbb{X}} \, }="c" & \mathbb{X} \ar@{=>}^{1_{\id_{\mathbb{X}}}}"a";"c"-<0pt,4pt>}\]

The \textbf{core bundle} of $\mathbb{E}$ is the vector bundle over $X_0$ defined as the pullback $C:=\ker(s:E_1\to X_1)|_{X_0}$. The \textbf{core complex} is the 2-term complex of vector bundles $\partial:C\to E_0$ with $\partial(c)=t(c)$ where $t:\mathbb{E}\to \mathbb{X}$ is the target map. 

The \textbf{complex of multiplicative sections} is the 2-term complex $C^{\bullet}_{mult}(\mathbb{E})$ defined by

\begin{equation}\label{eq:complexmultiplicative}
\delta:\Gamma(C)\to \Gamma_{mult}(\mathbb{E}); c\mapsto c^r-c^l,
\end{equation}

\noindent where $c^r(x):=c_{t(x)}\cdot 0_x$ and $c^l(x):=-0_x\cdot (c_{s(x)})^{-1}$ for every $x\in X_1$. Here $0:\mathbb{X}\to \mathbb{E}$ denotes the zero section and $\cdot$ the groupoid multiplication in $\mathbb{E}$. As shown in \cite{OrtizWaldron}, the category of multiplicative sections of $\mathbb{E}$ is isomorphic to the 2-vector space $\Gamma(C)\oplus \Gamma_{mult}(\mathbb{E})\rr \Gamma_{mult}(\mathbb{E})$ defined by the complex of multiplicative sections. 

Let us consider now an LA-groupoid $\pi:\mathbb{E}\to \mathbb{X}$. In this case, the core bundle $C\to X_0$ is naturally a Lie algebroid and the complex of multiplicative sections is part of a crossed module of Lie algebras $\Gamma(C)\to \Gamma_{mult}(\mathbb{E})\to \mathrm{Der}(\Gamma(C))$ where the last map $e\mapsto D_e$ is given by 

\begin{equation}\label{eq:multiplicativecrossedmodule}
D_e(c)=[e,c^r]|_{X_0}. 
\end{equation}

\noindent In other words, the category of multiplicative sections of $\mathbb{E}$ inherits a natural structure of strict Lie 2-algebra which we denote by $\mathbb{L(E)}$. For more details see \cite{OrtizWaldron}.

\begin{example}\label{ex:crossedmultiplicativevectorfields}
Let $\mathbb{X}$ be a Lie groupoid with Lie algebroid $A_{\mathbb{X}}$. As shown in \cite{OrtizWaldron} (see also \cite{BEL}), the category of multiplicative sections of the tangent LA-groupoid $T\mathbb{X}\to \mathbb{X}$ is given by the crossed module $\Gamma(A_{\mathbb{X}})\to \mathfrak{X}_{mult}(\mathbb{X})\to \mathrm{Der}(\Gamma(A_{\mathbb{X}}))$, where $\mathfrak{X}_{mult}(\mathbb{X})$ denotes the Lie algebra of multiplicative vector fields on $\mathbb{X}$ of Mackenzie and Xu \cite{mackenzieXu}.
\end{example}

\begin{example}\label{ex:crossedmultiplicativefunctions}
Let $\mathbb{X}=(X_1\rr X_0)$ be a Lie groupoid and $\mathbb{V}:=V_{-1}\oplus V_0$ the 2-vector defined by a 2-term cochain complex $\partial:V_{-1}\to V_0$. The trivial VB-groupoid $\mathbb{X}\times \mathbb{V}\to \mathbb{X}$ has core bundle $C=X_0\times V_{-1}\to X_0$, the space of multiplicative sections is $\mathrm{Hom}_{gpd}(\mathbb{X},\mathbb{V})$ and the complex of multiplicative sections is given by 

$$C^{\infty}(X_0,V_{-1})\to \mathrm{Hom}_{gpd}(\mathbb{X},\mathbb{V}); f\mapsto (t^*f-s^*f)\oplus s^*(\partial\circ f).$$

The corresponding 2-vector space is denoted $\mathbb{F(X,V)}$ and referred to as the 2-vector space of \textbf{$\mathbb{V}$-valued multiplicative functions} on $\mathbb{X}$.
\end{example}

Let us consider now a $\GG$-equivariant LA-groupoid $\pi:\mathbb{E}\to \mathbb{X}$ with core bundle $C\to X_0$. Since $G_0$ acts on $\pi:\mathbb{E}\to \mathbb{X}$ by LA-groupoid automorphisms, the core $C\to X_0$ inherits the structure of a $G_0$-equivariant Lie algebroid. Also, since $t:\mathbb{E}\to \mathbb{X}$ is $\GG$-equivariant (Remark \ref{rmk:structuralequivariance}), the core complex $\partial:C\to E_0$ is a 2-term complex of $G_0$-equivariant vector bundles, i.e. $\partial$ is $G_0$-equivariant. It is straightforward to check the following.

\begin{proposition}\label{prop:coreequivariant}

Let $\pi:\mathbb{E}\to \mathbb{X}$ be a $\GG$-equivariant LA-groupoid with core $C\to X_0$. If the action is principal, then $G_0\curvearrowright C$ is also principal and quotient LA-groupoid $\mathbb{E}/\GG\to \mathbb{X}/\GG$ has core bundle $C/G_0\to X_0/G_0$ with core anchor map $\overline{\partial}:C/G_0\to E_0/G_0; [c]\mapsto [\partial(c)]$. Additionally, the space of multiplicative sections of $\mathbb{E}/\GG$ identifies with $\Gamma_{mult}(\mathbb{E})^\GG$ and the corresponding crossed module of multiplicative sections is isomorphic to

$$\Gamma(C)^{G_0}\to \Gamma_{mult}(\mathbb{E})^\GG \to \mathrm{Der}(\Gamma(C)^{G_0}),$$

\noindent where the maps are respectively given the restrictions of \eqref{eq:complexmultiplicative} and \eqref{eq:multiplicativecrossedmodule} to invariant sections.

\end{proposition}
 
This allows to describe the crossed module of multiplicative sections of both the adjoint LA-groupoid and the Atiyah LA-groupoid of a principal 2-bundle.

\begin{example}\label{ex:coreadjoint}

Let $\p\curvearrowleft \GG\to \mathbb{X}$ be a a principal 2-bundle and $\Ad(\p)\to \mathbb{X}$ be the adjoint LA-groupoid of Example \ref{ex:adjointVB}. If $\mathbb{g}=(\mathfrak{g}_1\rr \mathfrak{g}_0)$ is the Lie 2-algebra of $\GG$, let $\partial:\mathfrak{h}\to \mathfrak{g}_0$ be the 2-term complex defined by the crossed module associated to $\mathbb{g}$. The core bundle of $\Ad(\p)$ is  the Lie algebra bundle $P_{0}\times_{G_0} \mathfrak{h}\to X_0$ hence sections of the core bundle are elements in $C^{\infty}(P_0,\mathfrak{h})^{G_0}$. The space of multiplicative sections of $\Ad(\p)$ is given by

$$\Gamma_{mult}(\Ad(\p))\cong \Hom_{gpd}(\p,\mathbb{g})^{\GG},$$

\noindent where the right hand side denotes the space of $\GG$-equivariant groupoid morphisms $F:\p\to \mathbb{g}$. By $\GG$-equivariant we mean

$$F_1(pg)=\Ad_{g^{-1}}^{G_1}f_1(p),\quad F_0(p_0g_0)=\Ad_{g^{-1}_0}^{G_0}f_0(p_0),$$

\noindent for every $p\in P_1, p_0\in P_0$ and $g\in G_1,g_0\in G_0$. As a consequence of Proposition \ref{prop:coreequivariant} and Example \ref{ex:crossedmultiplicativefunctions}, the complex of multiplicative sections of $\Ad(\p)$ is

$$C^{\infty}(P_0,\mathfrak{h})^{G_0}\to \Hom_{gpd}(\p,\mathbb{g})^{\GG}; f\mapsto (t^*f-s^*f,s^*(\partial \circ f)),$$

\noindent where we have used the decomposition $\mathbb{g}=\mathfrak{h}\oplus \mathfrak{g}_1$ in terms of crossed module components.

\end{example}

\begin{example}\label{ex:coreatiyah}

Let $\At(\p)\to \mathbb{X}$ be the Atiyah LA-groupoid of a principal 2-bundle $\p\curvearrowleft \GG\to \mathbb{X}$. In this case, the core bundle of $T\p\to \p$ is the Lie algebroid $A_{\p}$ of $\p$, hence the core bundle of $\At(\p)$ is $A_{\p}/G_0$. The space of multiplicative sections of the Atiyah LA-groupoid is given by

$$\Gamma_{mult}(\At(\p))\cong \mathfrak{X}_{mult}(\p)^{\GG},$$

\noindent where the right hand side is the space of $\GG$-invariant multiplicative vector fields on $\p$. It follows from Proposition \ref{prop:coreequivariant} and Example \ref{ex:crossedmultiplicativevectorfields} that the complex of multiplicative sections of $\At(\p)$ is given by

$$\Gamma(A_{\p})^{G_0}\to \mathfrak{X}_{mult}(\p)^{\GG}; a\mapsto a^r-a^l.$$

\end{example}

Just as a usual principal bundle has an associated Atiyah sequence, for a  given principal 2-bundle the adjoint LA-groupoid and the Atiyah LA-groupoid fit into an exact sequence in the category of LA-groupoids \cite{CCP}.

\begin{definition}\label{def:atiyahsequence}

Let $\pi:\p\curvearrowleft \GG\to \mathbb{X}$ be a principal 2-bundle. The \textbf{Atiyah sequence} of $\p$ is the exact sequence of LA-groupoids

$$\begin{tikzcd}
    0\arrow[r]&\Ad(\mathbb{P})\arrow[r,"\iota"]\arrow[rd]&\At(\mathbb{P})\arrow[d]\arrow[r,"\overline{T\pi}"]&T\mathbb{X}\arrow[dl]\arrow[r]&0.\\
&&\mathbb{X}&&\end{tikzcd}$$

\end{definition}

As shown in \cite{OrtizWaldron}, every LA-groupoid morphism covering the identity induces a Lie 2-algebra morphism between the corresponding categories of multiplicative sections. In particular, the Atiyah sequence of $\p$ determines an extension of Lie 2-algebras

\begin{equation}\label{eq:Lie2atiyah}
\begin{tikzcd}
    0\arrow[r]&\mathbb{L}(\Ad(\mathbb{P}))\arrow[r,"\iota"]&\mathbb{L}(\At(\mathbb{P}))\arrow[r,"\overline{T\pi}"]&\mathbb{L}(T\mathbb{X})\arrow[r]&0,\\
\end{tikzcd}
\end{equation}

\noindent referred to as the \textbf{Atiyah extension} of $\pi:\p\curvearrowleft \GG\to \mathbb{X}$ and denoted by $\mathbb{At}(\p)$.

%%%%%%%%%%%%%%%%%%%%%%%%
%%%%%%%%%%%%%%%%%%%%%%%

\subsection{Connections on principal 2-bundles}

Let $\mathbb{X}=(X_1\rr X_0)$ be a Lie groupoid and consider its nerve $N\mathbb{X}=(X_{n})$ with its usual structure of simplicial manifold. Here $X_n$ denotes the space of $n$-simplices, which are strings of arrows $(x_1,...,x_n)$ with $s(x_i)=t(x_{i+1}), i=1,...,n-1$ and face maps $d_i:X_{n+1}\to X_n; i=1,...,n+1$. The simplicial structure gives rise to a complex $\delta:\Omega^{\bullet}(X_n)\to \Omega^{\bullet}(X_{n+1})$ defined as the alternating sum of pullbacks by the face maps. This differential does not change the form degree, but it increases the simplicial degree by +1. For instance, $\delta\eta=t^*\eta-s^*\eta$ for every $\eta\in \Omega^{k}(X_0)$ and $\delta\omega=pr^*_1\omega-m^*\omega+pr^*_2\omega$ if $\omega\in\Omega^{k}(X_1)$, where $m,pr_1,pr_2:X_1\times_{X_0}X_1\to X_1$ denote the groupoid multiplication and the canonical projections, respectively.

The \textbf{Bott-Shulman-Stasheff complex} of $\mathbb{X}$ is the double complex $\Omega^{\bullet}(X_{\bullet})$ with vertical differential given by the usual de Rham differential of forms and horizontal differential defined by the simplicial structure. The \textbf{de Rham complex} of $\mathbb{X}$ is the total complex:

$$C^{p}_{dR}(\mathbb{X}):=\bigoplus_{i+j=p}\Omega^i(X_j),$$

\noindent with differential $D\omega:=d\omega+(-1)^{j+1}\delta\omega$ for every $\omega \in \Omega^{i}(X_j)$.

A \textbf{multiplicative $k$-form} on  $\mathbb{X}$ is a $k$-form $\omega\in \Omega^k(X_1)$ which satisfies $\delta\omega=0$. That is,

\begin{equation}\label{eq:multiplicativeform}
m^*\omega=pr^*_1\omega+pr^*_2\omega.
\end{equation}

\noindent As shown in \cite{bursztyn-cabrera, BCO}, a $k$-form $\omega\in \Omega^k(X_1)$ is multiplicative if and only if the induced bundle map $\omega:\oplus^{k-1}T\mathbb{X}\to T^*\mathbb{X}$ is a Lie groupoid morphism.

Let $\eta\in \Omega^{k+1}_{cl}(X_0)$ be a closed form. We say that a multiplicative form $\omega\in \Omega^k(X_1)$ is \textbf{closed} with respect to $\eta$ if $d\omega=t^*\eta-s^*\eta$. In other words, the element $(\omega+\eta)\in C^{k+1}_{dR}(\mathbb{X})$ is a cocycle in the de Rham complex of $\mathbb{X}$.

\begin{remark}
A multiplicative $k$-form on $\mathbb{X}$ should be thought of as a $k$-form of degree 1 on the quotient stack $[X_0/X_1]$ since it yields a morphism of complexes $\bigoplus^{k-1} T_{\mathbb{X}}\to T^*_{\mathbb{X}}[1]$, where $T_{\mathbb{X}}:=(A_{\mathbb{X}}\to TX_0)$ is the tangent complex and $T^*_{\mathbb{X}}:=(T^*X_0\to A^*_{\mathbb{X}})$ is the cotangent complex. Similarly, a multiplicative $k$-form on $\mathbb{X}$ which is closed with respect to $\eta\in \Omega^{k+1}_{cl}(X_0)$ has to be thought of as a closed $k$-form on the quotient stack $[X_0/X_1]$. The special case of 2-forms arises in the context of Dirac structures \cite{BCWZ} as well as in shifted symplectic geometry \cite{ptvv}, see also \cite{calaque}.
\end{remark}

 A more general notion of multiplicative form with values in a VB-groupoid is considered in  \cite{Drummond-Egea}. We are interested in the particular case of $\mathbb{g}$-valued multiplicative $k$-forms on a Lie groupoid $\mathbb{X}$ where $\mathbb{g}=(\mathfrak{g}_1\rr \mathfrak{g}_0)$ is a Lie 2-algebra.

 \begin{definition}\label{def:gvaluedmultiplicative}
 
 A \textbf{$\mathbb{g}$-valued multiplicative $k$-form} on a Lie groupoid $\mathbb{X}=(X_1\rr X_0)$ is a $k$-form $\omega_1\in \Omega^k(X_1,\mathfrak{g}_1)$ such that $\omega_1:\oplus^kTX_1\to \mathfrak{g}_1\times X_1$ is a Lie groupoid morphism covering $\omega_0:\oplus^kTX_0\to \mathfrak{g}_0\times X_0$ where $\omega_0\in \Omega^k(X_0,\mathfrak{g}_0)$.  
 
 \end{definition}
 
 If $\partial:\mathfrak{h}\to \mathfrak{g}_0$ is the cochain complex associated with $\mathbb{g}$, then a $\mathbb{g}$-valued multiplicative $k$-form on $\mathbb{X}$ is given by a pair of $k$-forms $\bar\omega\in \Omega^k(X_1,\mathfrak{h})$ and $\omega_0\in \Omega^k(X_0,\mathfrak{g}_0)$ with

\begin{equation}\label{eq:splitmultiplicativeform}
\omega_1=\bar\omega + s^*\omega_0,
\end{equation}
\noindent and satisfying the following conditions:

\begin{align}\label{eq:conditionssplitforms}
m^*\bar\omega=&pr^*_1\bar\omega+pr^*_2\bar\omega\\
\partial \circ \bar\omega=&t^*\omega_0-s^*\omega_0. \nonumber
\end{align}

\noindent For more details, see \cite{Drummond-Egea}.

\begin{definition}\label{def:gvaluedclosed}
Let $\mathbb{X}=(X_1\rr X_0)$ be a Lie groupoid and $\eta\in\Omega^{k+1}_{cl}(X_0,\mathfrak{h})$ be a closed form. We say that a $\mathbb{g}$-valued multiplicative $k$-form $\omega_1=(\bar\omega,s^*\omega_0)$ is \textbf{closed} with respect to $\eta$ if the following conditions hold:

\begin{align}\label{eq:gvaluedclosed}
d\omega_0=&\partial\circ \eta\\
d\bar\omega=&(t^*\eta-s^*\eta)\nonumber
\end{align}

\end{definition}

\begin{remark}
A $\mathbb{g}$-valued multiplicative $k$-form on $\mathbb{X}$ should be thought of as a $k$-form of degree 1 on the stack $[X_0/X_1]$ with values in the stacky Lie algebra $[\mathfrak{g}_0/\mathfrak{g}_1]$. A similar interpretation holds for $\mathbb{g}$-valued multiplicative $k$-forms which are closed with respect to $\eta\in\Omega^{k+1}_{cl}(X_0,\mathfrak{h})$.
\end{remark}

Just as multiplicative $k$-forms with respect to closed $(k+1)$-form are $(k+1)$-cocycles in the de Rham complex of a Lie groupoid, closed $\mathbb{g}$-valued multiplicative $k$-form are cocycles in the differential graded Lie algebra given by the de Rham complex of a Lie groupoid \emph{with values} in a Lie 2-algebra \cite{waldorf1}. This will be discussed below.

Let $\mathbb{g}$ be a Lie 2-algebra with associated 2-term cochain complex $\partial:\mathfrak{h}\to \mathfrak{g}_0$. Due to notation purposes, it is useful to see this 2-term complex as $\partial_k:\mathfrak{v}_k\to \mathfrak{v}_{k+1}$ where $\partial_{-1}:=\partial:\mathfrak{h}\to \mathfrak{g}_0$ and $\mathfrak{v}_k=0$ for every $k\neq -1,0$. Let $\mathbb{X}=(X_1\rr X_0)$ be a Lie groupoid and consider the simplicial manifold given by the nerve $N\mathbb{X}=(X_{n})$. There is a total complex

$$C^p(\mathbb{X},\mathbb{g}):=\bigoplus_{i+j+k=p}\Omega^i(X_j,\mathfrak{v}_k),$$

\noindent with differential $C^p(\mathbb{X},\mathbb{g})\to C^{p+1}(\mathbb{X},\mathbb{g})$ given by $d+(-1)^i\delta+(-1)^{i+j}\partial_k$ on each component $\Omega^i(X_j,\mathfrak{v}_k)$. Here $d$ denotes the usual de Rham differential of forms and  $\delta$ is the simplicial differential on forms. As explained in \cite{waldorf1}, if one considers elements in $C^p(\mathbb{X},\mathbb{g})$ whose form degree is at least $p$, then the total differential does not preserve such an elements. However, one can co-restrict the total differential yielding a map $\mathbb{D}:C^p_{dR}(\mathbb{X},\mathbb{g})\to C^{p+1}_{dR}(\mathbb{X},\mathbb{g})$ which actually defines a differential. An element $\Phi\in C^p_{dR}(\mathbb{X},\mathbb{g})$ is given by a triple of differential forms: 

$$\Phi^a\in \Omega^p(X_0,\mathfrak{g}_0); \quad \Phi^b\in \Omega^p(X_1,\mathfrak{h}); \quad \Phi^c\in \Omega^{p+1}(X_0,\mathfrak{h}),$$

\noindent such that

\begin{align}\label{eq:waldrofsplitforms}
m^*\Phi^b=&pr^*_1\Phi^b+pr^*_2\Phi^b\\
\partial \circ \Phi^b=&t^*\Phi^a-s^*\Phi^a. \nonumber
\end{align}

\noindent The component $\Phi^c$ controls the differential of $\Phi$. More concretely, the differential $\mathbb{D}\Phi\in C^{p+1}_{dR}(\mathbb{X},\mathbb{g})$ is given by the triple $((\mathbb{D}\Phi)^a,(\mathbb{D}\Phi)^b,(\mathbb{D}\Phi)^c)$ with

\begin{align}\label{eq:waldorfsplitdifferential}
(\mathbb{D}\Phi)^a=&d\Phi^a-(-1)^p\partial\circ \Phi^c\\
(\mathbb{D}\Phi)^b=&d\Phi^b-(-1)^p\delta_0\circ \Phi^c\\
(\mathbb{D}\Phi)^c=&d\Phi^c
\end{align}

\begin{remark}\label{rmk:waldorfvsmultiplicative}

Note that $\Phi\in C^p_{dR}(\mathbb{X},\mathbb{g})$ if and only if $\Phi^b+s^*\Phi^a$ is a $\mathbb{g}$-valued multiplicative form on $\mathbb{X}$ together with conditions on $\Phi^c$ \eqref{eq:waldorfsplitdifferential} which encodes information about derivatives. In \cite{waldorf1}, elements in $C^{\bullet}_{dR}(\mathbb{X},\mathbb{g})$ are called $\mathbb{g}$-valued differential forms on $\mathbb{X}$. It is important to observe that in this paper we will always consider \emph{multiplicative} $\mathbb{g}$-valued forms in the sense of Definition \ref{def:gvaluedmultiplicative}, which is a particular case of \cite{waldorf1}. Also, note that $\Phi=(\Phi^a,\Phi^b,\Phi^c)\in C^p_{dR}(\mathbb{X},\mathbb{g})$ being a cocycle, i.e. $\mathbb{D}\Phi=0$, is equivalent to saying that $\Phi^b+s^*\Phi^a$ is a $\mathbb{g}$-valued multiplicative $p$-form on $\mathbb{X}$ which is closed with respect to $\Phi^c$, see Definition \ref{def:gvaluedclosed}.
\end{remark}

We are particularly concerned with the case of $\mathbb{g}$-valued multiplicative \emph{connection} 1-forms.

\begin{definition}\label{def:2connection}

Let $\p\curvearrowleft \GG\to \mathbb{X}$ be a principal 2-bundle and let $\mathbb{g}$ be the Lie 2-algebra of $\GG$. A $\mathbb{g}$-valued multiplicative 1-form $\theta$ on $\p$ is called a \textbf{2-connection} if it is a connection on both morphisms and objects.

\end{definition}

In other words, a 2-connection is given by a pair $\theta=(\theta_1,\theta_0)$ of usual principal connections $\theta_i\in \Omega^1(P_i,\mathfrak{g}_i), i=1,2$ such that $\theta_1:TP_1\to \mathfrak{g}_1\times P_1$ is a Lie groupoid morphism covering $\theta_0:TP_0\to \mathfrak{g}_0 \times P_0$. As observed in \eqref{eq:splitmultiplicativeform},  one can write $\theta_1=\bar\theta+s^*\theta_0$ where $\bar\theta\in\Omega^1(P_1,\mathfrak{h})$ satisfies the identities in \eqref{eq:conditionssplitforms}. From now on a 2-connection will be written $\theta=\bar\theta+s^*\theta_0$.

\begin{remark}\label{rmk:relationwaldorf}
Connections on principal 2-bundles were introduced in \cite{waldorf1} as a triple $(\theta^a,\theta^b,\theta^c)$ of forms $\theta^a\in \Omega^1(P_0,\mathfrak{g}_0), \theta^b\in \Omega^1(P_1,\mathfrak{h})$ and $\theta^c\in\Omega^2(P_0,\mathfrak{h})$ with $\theta^a$ and $\theta^b$ satisfying identities \eqref{eq:waldrofsplitforms} and \eqref{eq:waldorfsplitdifferential} together with a suitable notion of invariance. Every 2-connection $\theta=\bar\theta+s^*\theta_0$ in our sense defines such a triple as $(\theta_0,\bar\theta,0)$, however the notion of invariance is different from that of \cite{waldorf1}. Regarding the term $\theta^c$, as shown in \cite{waldorf1} it only plays a role when introducing the curvature of a connection, once connections in \cite{waldorf1} are defined in order to have a well-defined parallel transport not only along paths but also along surfaces called bigons. In this paper, 2-connections are defined as a model for connections on principal bundles over differentiable stacks. We also should mention that the notion of 2-connection of our work coincides with the one in \cite{CCP} .
\end{remark}

Let us see some examples of principal 2-bundles with 2-connections. For more examples, see \cite{waldorf1}.

\begin{example}

Every usual principal $G$-bundle $P\to X$ can be seen as a principal 2-bundle $\p=(P\rr P)\to \mathbb{X}=(X\rr X)$ as in Example \ref{ex:2bundlebyusualbundle}. Every connection 1-form $\theta_0$ on $P$ can be seen as a 2-connection $\bar\theta + s^*\theta_0$ on $\p$ with $\bar\theta=0$. 
\end{example}
\begin{example}

Let $\GG=(G_1\rr G_0)$ be the Lie 2-group defined by a crossed module $H\to G_0$. The Maurer-Cartan form $\theta_{MC}\in\Omega^1(\GG,\mathbb{g})$ defines a 2-connection on $\GG$ seen as a principal 2-bundle $\GG\curvearrowleft \GG\to \{*\}$. One has $\theta_{MC}=(\bar\theta, \theta_0)$ where $\theta_0$ is the Maurer-Cartan form on $G_0$ and $\bar\theta\in\Omega^1(H\ltimes G, \mathfrak{h})$ is defined by

$$\bar\theta_{(h,g)}=\Ad_{g^{-1}}(\theta^{H}_{MC})_h,$$

\noindent where $\theta^{H}_{MC}$ is the Maurer-Cartan form on $H$. See \cite{waldorf1} for details.

\end{example}

\begin{example}

Let $\mathcal{G}=(Y,P,\mu)$ be an $\mathbb{S}^1$-gerbe over $X$, see Example \ref{ex:2bundlesgerbes}. A \textbf{connection} on $\mathcal{G}$ is a connection 1-form $\theta$ on the principal $\mathbb{S}^1$-bundle $P\to Y\times_{X}Y$ which is invariant by the gerbe multiplication. That is,

\begin{equation}\label{eq:gerbeconnectioninvariance}
\mu^*(\hat{d_1}^*\theta)=(\hat{d_2}^*\theta)\otimes (\hat{d_0}^*\theta).
\end{equation}

\noindent Here $\hat{d}_i:d^*_iP\to P$ is the natural map covering the face map $d_i$ on $Y_{3}$. One easily observes that condition \eqref{eq:gerbeconnectioninvariance} is equivalent to $\theta$ being a multiplicative 1-form on the Lie groupoid $P\rr Y$. Hence, connections on $\mathbb{S}^1$-gerbes can be regarded as 2-connections on its underlying principal 2-bundle. For more details see \cite{KrepskiVaughan}.

\end{example}

It is well-known that connections on a usual principal bundle are in one-to-one correspondence with splittings of its Atiyah sequence. One immediately observes that the proof can carry over verbatim in the case of a principal 2-bundle. Hence, there is a one-to-one correspondence between 2-connections on a principal 2-bundle $\p\curvearrowleft \GG\to \mathbb{X}$  and splittings 
of the Atiyah sequence of Definition \ref{def:atiyahsequence}. By a splitting we mean a VB-groupoid morphism $h:T\mathbb{X}\to \At(\mathbb{P})$ with $\tilde{d\pi}\circ h=\id_{T\mathbb{X}}$. Such a splitting will be referred to as a \textbf{multiplicative horizontal lift} of $\At(\mathbb{P})$.

%%%%%%%%%%%%%%%%
%%%%%%%%%%%%%%%%

\subsection{The curvature of a 2-connection}\label{sec:curvature-of-2-connection}

Our main goal now is to define the curvature of a 2-connection. For that, we start by fixing the notation in the case of ordinary principal bundles.

Let $P\curvearrowleft G \to X$ be a principal bundle with connection $\theta\in\Omega^1(P,\mathfrak{g})$. Let $\Omega(P,V)$ be the space of $V$-valued differential forms on $P$ where $V$ is a vector space. The \textbf{covariant derivative} of $\omega\in \Omega(P,V)$  is defined as $\nabla^{\theta}\omega:=(d\omega)^h$, the horizontal part of the usual exterior derivative. The \textbf{curvature} of a connection $\theta$ is the 2-form $\Omega:=\nabla^{\theta}\theta\in \Omega^2(P,\mathfrak{g})$. The curvature 2-form can be written as

\begin{equation}\label{eq:MCcurvature}
\Omega=d\theta+\frac{1}{2}[\theta,\theta].
\end{equation}

If $\rho:G\to GL(V)$ a representation, we say that $\omega\in \Omega^k(P,V)$ is \textbf{basic} if it is horizontal and invariant in the sense that $R^*_g\omega=\rho_{g^{-1}}\omega$ for every $g\in G$. We denote by $\Omega^k_{\rho}(P,V)$ the space of basic $k$-forms on $P$ with values in $V$. The space $\Omega^k_{\rho}(P,V)$ identifies canonically with $\Omega^k(X,P\times_GV)$, the space of $k$-forms on $X$ with values in the associated bundle $P\times_GV\to X$. The curvature of a connection is basic with respect to the adjoint representation, hence it can be seen as a 2-form on $X$ with values in the adjoint bundle $\Ad(P)$.

For basic forms $\omega\in \Omega^k_{\rho}(P,V)$ the covariant derivative has the following explicit formula:

\begin{equation}\label{eq:covariantbasic}
\nabla^{\theta}\omega=d\omega + \rho_{*}(\theta\wedge \omega).
\end{equation}

\noindent Here $\rho_{*}(\theta\wedge \omega)\in \Omega^{k+1}(P,V)$ is defined as

$$\rho_{*}(\theta\wedge \omega)(U_1,...,U_{k+1})=\sum_{\sigma \in \Sigma_{k+1}}(-1)^{\sigma}\rho_{*}(\theta(U_{\sigma(1)}))\omega(U_{\sigma(2)},...,U_{\sigma(k+1)}),$$

\noindent where $\rho_{*}:\mathfrak{g}\to \mathfrak{gl}(V)$ is the induced Lie algebra representation.

We proceed now to define the curvature of a 2-connection on a principal 2-bundle.

\begin{definition}\label{def:2curvature}

Let $\theta=\bar\theta+s^*\theta_0\in \Omega^1(\p,\mathbb{g})$ be a 2-connection on a principal 2-bundle $\p\curvearrowleft \GG\to \mathbb{X}$. The \textbf{curvature} of $\theta$ is the 2-form $\Omega:=d\theta+\frac{1}{2}[\theta,\theta]\in \Omega^2(\p,\mathbb{g})$.
\end{definition}

\begin{remark}
It is reasonable to define the curvature of a 2-connection $\theta=\bar\theta+s^*\theta_0$ by looking at the expression \eqref{eq:MCcurvature} of each of its components $\bar\theta$ and $\theta_0$. However, the 2-form $d\bar\theta+\frac{1}{2}[\bar\theta,\bar\theta]\in \Omega^2(P_1,\mathfrak{h})$ fails to be multiplicative, hence it does not provide a good notion of curvature in the setting of principal bundles over Lie groupoids.
\end{remark}

Note that the curvature $\Omega\in \Omega^2(\p,\mathbb{g})$ as in Definition \ref{def:2curvature} is a multiplicative 2-form. This follows immediately from the expression $\Omega=(d\theta)^h$ combined with the fact that $d\theta$ is multiplicative and the horizontal projection $TP\to H:=\Ker(\theta)$ is a VB-groupoid morphism. As a consequence, one can write the curvature as

\begin{equation}\label{eq:2curvaturecomponents}
\Omega=\bar\Omega+s^*\Omega_0,
\end{equation}

\noindent where $\bar\Omega\in\Omega^2(P_1,\mathfrak{h})$ and $\Omega_0\in\Omega^2(P_0,\mathfrak{g}_0)$. More concretely, it follows from Definition \ref{def:2curvature} that $\Omega_0=d\theta_0+\frac{1}{2}[\theta_0,\theta_0]$ is just the curvature of $\theta_0$, and

$$\bar\Omega=d\bar\theta+\frac{1}{2}[\bar\theta,\bar\theta]+[s^*\theta_0,\bar\theta].$$

Now we describe $\bar\Omega$  in \eqref{eq:2curvaturecomponents} in terms of the covariant derivative. For that, recall that the source map $P_1\to P_0$ is a morphism of principal bundles with connections, hence the corresponding covariant derivatives are related by $\nabla^{\theta}\circ s^*=s^*\circ \nabla^{\theta_0}$. Since $\Omega=\nabla^{\theta}\theta$ and $\theta=\bar\theta+s^*\theta_0$, one has

\begin{align*}
\Omega=&\nabla^{\theta}\bar\theta+s^*\nabla^{\theta_0}\theta_0\\
=&\nabla^{\theta}\bar\theta + s^*\Omega_0.
\end{align*}

\noindent The uniqueness of the decomposition \eqref{eq:2curvaturecomponents} implies $\bar\Omega:=\nabla^{\theta}\bar\theta$. In particular, as a consequence of Bianchi identity $\nabla^{\theta}\Omega=0$, one has the following Bianchi identity for $\bar\Omega$:

\begin{equation}\label{eq:Bianchibaromega}
 \nabla^{\theta}\bar\Omega=0.
\end{equation}

\begin{remark}\label{rmk:basiccurvature2connection}

Let $\Omega^2(\p,\mathbb{g})$ be the curvature of a 2-connection $\theta\in \Omega^1(\p,\mathbb{g})$. Since $\Omega=\bar\Omega+ s^*\Omega_0$ is basic with respect to $\pi:\p\curvearrowleft \GG\to \mathbb{X}$, there exists a unique \emph{multiplicative} 2-form $\omega\in \Omega^2(\mathbb{X},\Ad(\p))$ with $\pi^*\omega=\Omega$. As shown in \cite{Drummond-Egea}, a VB-groupoid valued form is completely determined by its core and side components, which in the case of $\omega\in \Omega^2(\mathbb{X},\Ad(\p))$ such a components are $\bar\omega\in\Omega^2(X_1,t^*(P_0\times_{G_0} \mathfrak{h}))$ and $\omega_0\in \Omega^2(X_0,P_0\times_{G_0}\mathfrak{g}_0)$, and

$$\omega=\bar\omega+s^*\omega_0,$$

\noindent Since $\pi^*\omega=\Omega$, then $\pi^*_0\omega_0=\Omega_0$ and $\pi^*_1\bar\omega=\bar\Omega$. Also, due to Bianchi identity \eqref{eq:Bianchibaromega} $\nabla^{\theta}\bar\Omega=0$, one has that $\nabla_0\bar\omega=0$ where $\nabla_0$ is a covariant exterior derivative defined by the induced connection on $t^*(P_0\times_{G_0} \mathfrak{h})\to X_1$.

\end{remark}

%\begin{remark}

%A 2-connection $\theta=\bar\theta+s^*\theta_0$ on $\p$ can be seen as an element $\theta:=(\theta^a,\theta^b,0)\in C^1_{dR}(\mathbb{P},\mathbb{g})$ with $\theta^a=\theta_0$ and $\theta^b=\bar\theta$, see Remark \ref{rmk:relationwaldorf}. It was shown in \cite{waldorf1} that for every Lie groupoid $\mathbb{X}$ the set $(C^{\bullet}_{dR}(\mathbb{X},\mathbb{g}),\mathbb{D},[\cdot\wedge\cdot])$ is a differential graded Lie algebra. In terms of the dgla associated to the Lie groupoid $\p$, the curvature of a 2-connection as in Definition \ref{def:2curvature} reads

%$$\Omega=\mathbb{D}\theta+\frac{1}{2}[\theta\wedge \theta]\in C^2_{dR}(\mathbb{P},\mathbb{g}).$$ 

%\end{remark}

It is useful to consider 2-connections with prescribed curvature. This is formalized by introducing the notion of a curving 2-form. Let $G_0\to H$ be the crossed module associated to the Lie 2-group $\GG$. One of the components of the crossed module is a Lie group morphism $\alpha:G_0\to \Aut(H)$ which yields a Lie group representation $\tau:G_0\to GL(\mathfrak{h})$. In particular, one can look at basic forms $\Omega_{\tau}(P_0,\mathfrak{h})$ with respect to the representation $\tau$.

\begin{definition}\label{def:curving}
Let $\theta=\bar\theta+s^*\theta_0$ be a 2-connection on a principal 2-bundle $\p\curvearrowleft\GG\to \mathbb{X}$ with curvature $\Omega=\bar\Omega + s^*\Omega_0$.  A \textbf{curving} for $\theta$ is a basic 2-form $B\in\Omega^2_{\tau}(P_0,\mathfrak{h})$ with:

\begin{enumerate}

\item $\Omega_0+\partial\circ B=0$, and
\item $\bar\Omega+(t^*-s^*)(B)=0$.
 
\end{enumerate}

\end{definition}

A \textbf{connective structure} on a principal 2-bundle $\p\curvearrowleft\GG\to \mathbb{X}$ is a pair $(\theta,B)$ where $\theta$ is a 2-connection on $\p$ and $B$ is a curving for $\theta$. The \textbf{curvature 3-form} of a connective structure $(\theta,B)$ is the 3-form $\nabla^c_{\theta_0}B\in \Omega^3(P_0,\mathfrak{h})$. Due to \eqref{eq:covariantbasic}, the curvature 3-form has the following expression

\begin{equation}\label{eq:curvature3form}
\nabla^c_{\theta_0}B=dB+\tau_{*}(\theta_0\wedge B),
\end{equation}

\noindent where $\tau_{*}:\mathfrak{g}_0\to \mathfrak{gl}(\mathfrak{h})$ is the Lie algebra representation induced by $\tau:G_0\to GL(\mathfrak{h})$.

The terminology is motivated by the following example.

\begin{example}\label{ex:connectivegerbe}

Let $\mathcal{G}=(Y,P,\mu)$ be an $\mathbb{S}^1$-gerbe over a manifold $X$. A connective structure $(\theta,B)$ on the underlying principal 2-bundle associated to $\mathcal{G}$ is just a connection 1-form on $P\to Y\times_XY$ and $B\in \Omega^2(Y)$ such that $\Omega=-(d^*_1-d^*_0)B$, where $\Omega$ is the curvature of $\theta$. This recovers the notion of connective structure on an $\mathbb{S}^1$-gerbe, \cite{KrepskiVaughan}. In this case, the curvature 3-form is $dB\in \Omega^3(Y)$. Since $dB$ is basic with respect to $q:Y\to X$, there exists a unique closed 3-form $\chi\in\Omega^3_{cl}(X)$ with $dB=q^*\chi$. The 3-form $\chi$ is called the \textbf{3-curvature} of the connective structure $(\theta,B)$ on the gerbe $\mathcal{G}=(Y,P,\mu)$.

\end{example}

\begin{remark} 

As observed in Example \ref{ex:connectivegerbe}, the curvature 3-form of a connective structure on a gerbe $\mathcal{G}=(Y,P,\mu)$ is a closed 3-form on the manifold $X$, which is the base of the submersion $q:Y\to X$. Regarding $X$ as the quotient stack of of the submersion groupoid $Y\times_XY\rr Y$, the curvature 3-form can be seen as a 3-form on the stack $X=[Y/Y\times_XY]$. For a general principal 2-bundle $\p\curvearrowleft \GG\to \mathbb{X}$, the curvature 3-form of a connective structure is an element in $\Omega^3(P_0,\mathfrak{h})$. However, one can see the curvature 3-form as a suitable multiplicative 3-form on $\mathbb{X}$ with values in the VB-groupoid $\Ad(\p)$, which in the case of a regular groupoid $\mathbb{X}$ yields a 3-form on the quotient stack $[X_0/X_1]$ with values in  $\Ad(\p)$. We will explain this below.

\end{remark}

As shown in \cite{Drummond-Egea}, $k$-forms on a Lie groupoid $\mathbb{Y}$ with values in a $VB$-groupoid $\mathbb{E}\to \mathbb{Y}$ determine a cochain complex $\delta_{\mathbb{E}}:C^{\bullet,k}(\mathbb{Y},\mathbb{E})\to C^{\bullet,k}(\mathbb{Y},\mathbb{E})$ whose 1-cocycles are precisely multiplicative $k$-forms. If $C\to Y_0$ is the core bundle of $\mathbb{E}\to \mathbb{Y}$, then $C^{0,k}(\mathbb{Y},\mathbb{E}):=\Omega^k(Y_0,C)$. A degree zero cocycle is called \textbf{$\mathbb{Y}$-invariant}.

In the case of a principal 2-bundle $\p\curvearrowleft \GG \to \mathbb{X}$ and the trivial VB-groupoid $\p\times \mathbb{g}$, one has the complex $\delta_{\p\times\mathbb{g}}:C^{p,2}(\p,\mathbb{g})\to C^{p+1,2}(\p,\mathbb{g})$ of $\mathbb{g}$-valued 2-forms on $\p$ whose 1-cocycles are $\mathbb{g}$-valued multiplicative $2$-forms on $\p$.  A connective structure $\tilde\theta:=(\theta,B)$ on $\p$ allows us to write the curvature 

$$\Omega=-((t^*B-s^*B+s^*(\partial\circ B))=\delta_{\p\times \mathbb{g}} B$$

\noindent as a \emph{multiplicatively exact} 2-form with values in $\mathbb{g}$. The groupoid morphism $\pi:\p\to \mathbb{X}$ defines a cochain map $\pi^*:C^{\bullet,2}(\mathbb{X},\Ad(\p))\to C^{\bullet,2}(\p,\mathbb{g})$, hence the multiplicative 2-form $\omega\in\Omega^2(\mathbb{X},\Ad(\p))$ described in Remark \ref{rmk:basiccurvature2connection} is also multiplicatively exact. Indeed,  $\omega=\delta_{\Ad(\p)}B_0$ where $B_0\in \Omega^2(X_0,P_0\times_{G_0}\mathfrak{h})$ is the unique 2-form with $\pi^*_0B_0=B$.

Let $\nabla^c_{0}:\Omega(X_0,P_0\times_{G_0}\mathfrak{h})\to \Omega(X_0,P_0\times_{G_0}\mathfrak{h})$ the covariant exterior derivative defined by the induced connection on $P_0\times_{G_0}\mathfrak{h}\to X_0$. We claim that $\nabla^c_{0}B_0\in \Omega^3(X_0,P_0\times_{G_0}\mathfrak{h})$ is $\mathbb{X}$-invariant, that is, it is a zero cocycle in the complex $C^{\bullet,3}(\mathbb{X},\Ad(\p))$. To see this, it suffices to check that $\pi^*(\delta_{\Ad(\p)}\nabla^c_{0}B_0)=0$ for the cochain map $\pi^*:C^{\bullet,3}(\mathbb{X},\Ad(\p))\to C^{\bullet,3}(\p,\mathbb{g})$. Indeed,

\begin{align*}
\pi^*(\delta_{\Ad(\p)}\nabla^c_{0}B_0)=&\delta_{\p\times \mathbb{g}}(\pi^*_0\nabla^c_{0}B_0)\\
=& \delta_{\p\times \mathbb{g}}(\nabla^c_{\theta_{0}}\pi^*_0B_0)\\
=& \delta_{\p\times \mathbb{g}}(\nabla^c_{\theta_{0}} B)
\end{align*}

\noindent where $\nabla^c_{\theta_{0}}$ is the induced covariant exterior derivative on $P_0\times \mathfrak{h}$. But $\delta_{\p\times \mathbb{g}}(\nabla^c_{\theta_{0}} B)=\nabla (\delta_{\p\times \mathbb{g}}B)=0$ because of Bianchi identity and $\delta_{\p\times \mathbb{g}}B=\Omega$. The corresponding cohomology class $[\nabla^c_{0}B_0]\in H^{0,3}(\mathbb{X},\Ad(\p))$ is an invariant of principal 2-bundles with connective structures.

\begin{example}[\textbf{3-curvature of a gerbe}]

Let $(\theta,B)$ a connective structure on the underlying principal 2-bundle associated to an $\mathbb{S}^1$-gerbe $\mathcal{G}$ over $X$. In this case, the covariant exterior derivative $\nabla^c_0B=dB$ is the ordinary exterior derivative and being a cocycle $\delta_{\Ad(\p)}\nabla^c_{0}B=0$ is equivalent to saying that $dB\in \Omega^3(Y)$ is basic with respect to the surjective submersion $q:Y\to X$. Hence there exists a unique closed 3-form $\chi\in \Omega^3(X)$ with $q^*\chi=dB$. The 3-form $\chi$ is called the \textbf{3-curvature} of the connective structure $(\theta,B)$ on $\mathcal{G}$.

\end{example}

\subsubsection{Flatness}

A connective structure $(\theta,B)$ on $\p\curvearrowleft \GG\to \mathbb{X}$ is called \textbf{weakly flat} if its curvature 3-form vanishes, i.e. $\nabla^c_{\theta_0}B=0$, where $\nabla^{c}_{\theta_0}:\Omega(P_0,\mathfrak{h})\to \Omega(P_0,\mathfrak{h})$ is the induced a covariant derivative. Since $B\in \Omega^2(P_0,\mathfrak{h})$ is basic, it follows from \eqref{eq:curvature3form} that being weakly flat is equivalent to 

\begin{equation}\label{eq:flatcurving}
dB+\tau_{*}(\theta_0\wedge B)=0,
\end{equation}

\noindent where $\tau_{*}:\mathfrak{g}_0\to \mathfrak{gl}(\mathfrak{h})$ is the Lie algebra representation induced by $\tau:G_0\to GL(\mathfrak{h})$. In the special case $B=0$, we say that $\theta$ is a \textbf{flat} 2-connection on $\p$. In this case, conditions (1) and (2) in Definition \ref{def:curving} imply that the curvature $\Omega=\bar{\Omega}+s^*\Omega_0$ vanishes.

In what follows, we give a geometric interpretation of these notions of flatness in terms of horizontal lifts. For that, let $\p\curvearrowleft \GG\to \mathbb{X}$ be a principal 2-bundle and consider the corresponding Atiyah sequence
 
$$\begin{tikzcd}
    0\arrow[r]&\Ad(\mathbb{P})\arrow[r,"\iota"]\arrow[rd]&\At(\mathbb{P})\arrow[d]\arrow[r,"\overline{T\pi}"]&T\mathbb{X}\arrow[dl]\arrow[r]&0.\\
&&\mathbb{X}&&\end{tikzcd}$$

\noindent As observed in \eqref{eq:Lie2atiyah}, a multiplicative horizontal lift $h:T\mathbb{X}\to \At(\mathbb{\p})$ defines a splitting of the Atiyah extension

\begin{equation}\label{eq:atiyahextension}
\begin{tikzcd}
    0\arrow[r]&\mathbb{L}(\Ad(\mathbb{P}))\arrow[r,"\iota"]&\mathbb{L}(\At(\mathbb{P}))\arrow[r,"\overline{T\pi}"]&\mathbb{L}(T\mathbb{X})\arrow[r]&0.\\
\end{tikzcd}
\end{equation}

We spell this out in terms of the corresponding crossed modules of multiplicative sections. Since by definition $h:T\mathbb{X}\to \At(\mathbb{\p})$ is a VB-groupoid map, it induces a vector bundle morphism $h_c:A_{\mathbb{X}}\to A_{\p}/G_0$ between the core bundles as well as a map between multiplicative sections $h_m:\mathfrak{X}_{mult}(\mathbb{X})\to \Gamma_{mult}(\At(\mathbb{P})); (\xi,u)\mapsto(h\circ \xi,h_0\circ u),$ where $h_0:TX_0\to \At(P_0)$ is the map induced by $h$ on the units. The fact that $h:T\mathbb{X}\to \At(\mathbb{\p})$ is a VB-groupoid morphism implies that it induces a cochain map between the underlying complexes of multiplicative sections, i.e. $\delta\circ h_c=h_m\circ\delta$.

Using the explicit description of the crossed modules of multiplicative sections as in Examples \ref{ex:coreadjoint} and \ref{ex:coreatiyah}, a multiplicative horizontal lift yields to a diagram

\begin{equation}\label{eq:splittingAtiyah2extension}
    \begin{tikzcd}
       0\arrow[r]&C^{\infty}(P_0,\mathfrak{h})^{G_0}\arrow[d]\arrow[r,"\iota_c"]&\Gamma(A_{\mathbb{P}}/G_0)\arrow[d]\arrow[r,"\pi_c"]& \Gamma(A_{\mathbb{X}})\arrow[l,bend right,dashed,"h_c"']\arrow[d]\arrow[r]&0\\
   0\arrow[r]& \mathrm{Hom}_{gpd}(\p,\mathbb{g})^{\GG}\arrow[r,"\iota_m"]&  \mathfrak{X}_{mult}(\mathbb{P})^{\GG}\arrow[r,"\pi_m"]&\mathfrak{X}_{mult}(\mathbb{X})\arrow[l,bend right,dashed,"h_m"']\arrow[
   r]&0
    \end{tikzcd}
\end{equation}

\noindent where $\pi_c:\Gamma(A_{\p}/G_0)\to \Gamma(A_{\mathbb{X}})$ and $\pi_m:\mathfrak{X}_{mult}(\p)^{\GG}\to \mathfrak{X}_{mult}(\mathbb{X})$ are maps induced by the VB-groupoid morphism $\overline{T\pi}:\At(\p)\to T\mathbb{X}$ between core and multiplicative sections, respectively.

The induced map $h:\mathbb{L}(T\mathbb{X})\to \mathbb{L}(\At(\p))$ is not necessarily a Lie 2-algebra morphism. Being a weak Lie 2-algebra morphism amounts to the existence of a map $\Omega_m:\wedge^2\mathfrak{X}_m(\mathbb{X})\to C^{\infty}(P_0,\mathfrak{h})^{G_0}$ such that for every $\xi_1,\xi_2,\xi_3\in \mathfrak{X}_{mult}(\mathbb{X})$ and $a\in\Gamma(A_{\mathbb{X}})$, the following conditions are satisfied:\\

\begin{itemize}
   \item[WF1)] $h_m[\xi_1,\xi_2]-[h_m(\xi_1),h_m(\xi_2)]=\delta\Omega_m(\xi_1,\xi_2)$,
    \item[WF2)] $h_{c}(D_{\xi_1} a)-D_{h_m(\xi_1)}(h_c(a))=\Omega_m(\xi_1,\delta(a))$,
    \item[WF3)] $\big{(}\Omega_m([\xi_1,\xi_2],\xi_3)-D_{h_m(\xi_1)}(\Omega_m(\xi_2,\xi_
    3))\big{)}+c.p.=0$

\end{itemize}

\noindent where $c.p.$ in the last condition stands for cyclic permutations of $\xi_1,\xi_2,\xi_3$. Here $D_{\xi}(a):=[\xi,a^r]|_{X_0}$ denotes the action of $\xi\in\mathfrak{X}_{mult}(\mathbb{X})$
by derivations of $\Gamma(A_{\mathbb{X}})$ defined by the crossed module of multiplicative vector fields, see \eqref{eq:multiplicativecrossedmodule}. 

The next result establishes a relation between connective structures which are weakly flat and weak morphisms of Lie 2-algebras.

\begin{theorem}\label{thm:flatconnectivestructure}

Let $(\theta, B)$ be a connective structure on a principal 2-bundle $\p\curvearrowleft \GG\to \mathbb{X}$. If $(\theta,B)$ is weakly flat (resp. flat), then the multiplicative horizontal lift $h:\mathbb{L}(T\mathbb{X})\to \mathbb{L}(\At(\mathbb{\p}))$ defined by the 2-connection $\theta$ is a weak (resp. strict) morphism of 2-term $L_{\infty}$-algebras.
\end{theorem}

\begin{proof}

Let $(\xi_1,u_1),(\xi_2,u_2)$ be multiplicative vector fields on $\mathbb{X}$ and $u^h_1,u^h_2\in \mathfrak{X}(P_0)$ the corresponding horizontal lifts with respect to the connection $\theta_0$. Let us consider

$$\Omega_m(\xi,\eta):=-B(u^h,v^h).$$

\noindent Since the curving $B$ is invariant, one has a well defined map $\Omega_m:\bigwedge^2\mathfrak{X}_{mult}(\mathbb{X})\to C^{\infty}(P_0,\mathfrak{h})^{G_0}$ and we will show that $\Omega_m$ satisfies conditions WF1), WF2) and WF3). As observed in Example \ref{ex:coreadjoint}, the Lie 2-algebra of multiplicative sections of the adjoint 2-bundle is given by  the crossed module

 $$C^{\infty}(P_0,\mathfrak{h})^{G_0}\stackrel{\delta}{\longrightarrow}\mathrm{Hom}_{gpd}(\p,\mathbb{g})^{\GG}\stackrel{D}{\longrightarrow}\textrm{Der}(\Gamma(P_0\times_G\mathfrak{h})),$$

 \noindent where $D$ is defined by \eqref{eq:multiplicativecrossedmodule} and $\delta(f)=(t^*f-s^*f,s^*(\partial\circ f))$. Recall that $B$ being a weakly flat curving amounts to the next three equations:
 
\begin{align}
\Omega_0+\partial \circ B=&0\label{eq:curving1};\\
    \bar\Omega+(t^*-s^*)B=&0\label{eq:curving2};\\
    dB+\tau_{*}(\theta_0\wedge B)=&0,\label{eq:curvingflat}
\end{align}
  
\noindent where $\Omega=\bar\Omega+s^*\Omega_0$ is the curvature of $\theta$.
  
\noindent   For every $\xi_1,\xi_2\in\mathfrak{X}_{mult}(\mathbb{X})$ covering $u_1,u_2\in\mathfrak{X}(X_0)$, respectively, one has

 $$\Omega(\xi^h_1,\xi^h_2)=h_m[\xi_1,\xi_2]-[h_m\xi_1,h_m\xi_2],$$

\noindent which combined with equations \eqref{eq:curving1} and \eqref{eq:curving2} and the fact that $\xi^h_i$ is both $t$-related and $s$-related with $u^h_i, i=1,2$, imply
   \begin{align*}
       h_m[\xi_1,\xi_2]-[h_m\xi_1,h_m\xi_2]=&(t^*-s^*)(\Omega_m(\xi_1,\xi_2))+s^*(\partial(\Omega_m(\xi_1,\xi_2)))\\
       =&\delta(\Omega_m(\xi_1,\xi_2))\in \mathrm{Hom}_{gpd}(\p,\mathbb{g})^{\GG},
   \end{align*}
   
\noindent showing   WF1). To see WF2), let $a \in \Gamma(A_{\mathbb{X}})$ and $\xi\in\mathfrak{X}_{mult}(\mathbb{X})$ and note that 
   \begin{align}\label{eq:eq1thmflatconnectivestructure}
       h_c(D_{\xi}a)-D_{h_m(\xi)}(h_c(a))=&h\circ [\xi,a^r]|_{X_0}-[\xi^h,h_c(a)^r]|_{P_0}\\ \nonumber
       =&([\xi,a^r]^h-[\xi^h,h_c(a)^r])|_{P_0}\\ \nonumber
       =&\Omega(\xi^h,h_c(a)^r)|_{P_0}\\ \nonumber
       =&\bar\Omega(\xi^h,h_c(a)^r)+s^*\Omega_0(\xi^h,h_c(a)^r).\\ \nonumber
   \end{align}
   
On the right hand side of the previous identity, one observes that the second term vanishes once $h_c(a)^r$ is right-invariant, hence $s$-related to zero.  Regarding the first term, due to \eqref{eq:curving2}, one has

\begin{align}\label{eq:eq2thmflatconnectivestructure}
\bar\Omega(\xi^h,h_c(a)^r)|_{X_0}=&-(t^*-s^*)(B)(\xi^h,h_c(a)^r)|_{X_0}\\ \nonumber
=&-B(u^h,\rho(h_c(a))).
\end{align}

\noindent Since $h_m\circ \delta=\delta\circ h_c$, the multiplicative vector field  $\delta(a)^h=\delta(h_c(a))$ covers $\rho(h_c(a))$, hence $\Omega_m(\xi,\delta(a))=-B(u^h,\rho(h_c(a)))$. Combining this with \eqref{eq:eq1thmflatconnectivestructure} and \eqref{eq:eq2thmflatconnectivestructure}, one concludes

$$h_c(D_{\xi}a)-D_{h_m(\xi)}(h_c(a))=\Omega_m(\xi,\delta(a)),$$

\noindent which is WF2).

   Finally, we show WF3). For $i=1,2,3$ let $\xi_i\in \mathfrak{X}_{mult}(\mathbb{X})$ be a multiplicative vector field covering $u_i\in\mathfrak{X}(X_0)$. It follows from equation \eqref{eq:curvingflat}  that 
   \begin{align}\label{eq:eq3thmflatconnectivestructure}
       0=&(dB+\tau_{*}(\theta_0\wedge B))(u^h_1,u^h_2,u^h_3)\\ \nonumber
       =&\big{(}u^h_1(B(u^h_2,u^h_3))-B([u^h_1,u^h_2],u^h_3)+c.p.\big{)},\\ \nonumber
     \end{align}  

\noindent since $\alpha_*(\theta_0\wedge B)(u^h_1,u^h_2,u^h_3)=0$ once $\theta_0$ vanishes on horizontal vector fields. Also,   $B([u^h_1,u^h_2],u^h_3)=B([u_1,u_2]^h,u^h_3)$ since $B$ is a horizontal form.  Replacing this in \eqref{eq:eq3thmflatconnectivestructure} yields
       
       \begin{align}
       0=&\big{(}\Omega_m([\xi_1,\xi_2],\xi_3)-[\xi^h_1,\Omega_m(\xi_2,\xi_3)]|_{P_0}\big{)}+c.p.\\
       =&\big{(}\Omega_m([\xi_1,\xi_2],\xi_3)-D_{\xi^h_1}(\Omega_m(\xi_2,\xi_3))\big{)}+ c.p..
   \end{align}
   This shows WF3) and therefore $(h_c,h_m,\Omega_m)$ is a weak morphism of Lie 2-algebras, as claimed. 
\end{proof}

%%%%%%%%%%%%%%%%%%%%%
%%%%%%%%%%%%%%%%%%%%%

\begin{example}

Let $\mathcal{G}=(Y,P,\mu)$ be an $\mathbb{S}^1$-gerbe over a manifold $X$. A weakly flat connective structure $(\theta,B)$ on the underlying principal 2-bundle associated to $\mathcal{G}$ is just a connection 1-form on $P\to Y\times_XY$ with curving $B\in \Omega^2_{cl}(Y)$  a closed 2-form. This says that the 3-curvature $q^*\chi:=dB \in\Omega^3(X)$ of the connective structure vanishes. It follows from Theorem \ref{thm:flatconnectivestructure} that the corresponding horizontal lift is a weak morphism of Lie 2-algebras. This recovers a result by Krepski and Vaughan \cite{KrepskiVaughan} on $\mathbb{S}^1$-gerbes with flat 3-curvature.

\end{example}

\subsection{Chern-Weil-Lecomte map}

In what follows, we will state and prove the main result of this section, namely the construction of the Chern-Weil morphism for principal 2-bundles with 2-connections. This relies on the general set up of Theorem \ref{thm:lecomte} which gives rise to characteristic classes of extensions of $L_{\infty}$-algebras together with a representation up to homotopy. 

Let $\p\curvearrowleft \GG\to \mathbb{X}$ be a principal 2-bundle with associated Atiyah extension:

\begin{equation}\label{eq:atiyahextension2}
\begin{tikzcd}
    0\arrow[r]&\mathbb{L}(\Ad(\mathbb{P}))\arrow[r,"\iota"]&\mathbb{L}(\At(\mathbb{P}))\arrow[r,"\overline{T\pi}"]&\mathbb{L}(T\mathbb{X})\arrow[r]&0.\\
\end{tikzcd}
\end{equation}

Consider the multiplicative horizontal lift $h:T\mathbb{X}\to \At(\p)$ defined by a 2-connection and look at the associated splitting $h:\mathbb{L}(T\mathbb{X})\to \mathbb{L}(\At(\p))$ of \eqref{eq:atiyahextension2}. In order to apply Theorem \ref{thm:lecomte} to the extension \ref{eq:atiyahextension2}, it remains to define a representation up to homotopy of $\mathbb{L}(T\mathbb{X})$ on a cochain complex. We will see that such a representation up to homotopy can be defined in a canonical way. For that, consider the 2-vector space $\mathbb{F}(\mathbb{X})$ of real valued multiplicative functions on $\mathbb{X}$ defined by the 2-term cochain complex $\delta:C^{\infty}(X_0)\to C^{\infty}_{mult}(\mathbb{X}); f\mapsto t^*f-s^*f$, see Example \ref{ex:crossedmultiplicativefunctions}. As observed in Example \ref{ex:crossedmultiplicativevectorfields}, the Lie 2-algebra of multiplicative vector fields $\mathbb{L}(T\mathbb{X})$ is defined by the crossed module $\Gamma(A_{\mathbb{X}})\to \mathfrak{X}_{mult}(\mathbb{X}); a\mapsto a^r-a^l$, where $A_{\mathbb{X}}$ is the Lie algebroid of $\mathbb{X}$. Elements in $\mathbb{L}(T\mathbb{X})$ are pairs $(a,(\xi,u))$ with $a\in\Gamma(A_{\mathbb{X}})$ and $(\xi,u)$ a multiplicative vector field on $\mathbb{X}$. Similarly, elements in $\mathbb{F(X)}$ are pairs $(f,F)$ where $f\in C^{\infty}(X_0)$ and $F\in C^{\infty}_{mult}(\mathbb{X})$. 

There is a canonical degree zero map $\rho:\mathbb{L}(T\mathbb{X})\otimes \mathbb{F(X)}\to \mathbb{F(X)}$ defined by

\begin{equation}\label{eq:canonicalruthderivations}
\rho_{\mathbb{X}}((a,(\xi,u))\otimes (f,F)):=(\Lie_uf+(\Lie_{a^r}F)|_{X_0},\Lie_{\xi}F)
\end{equation}

It is possible to check that \eqref{eq:canonicalruthderivations} is a representation up to homotopy of $\mathbb{L}(T\mathbb{X})$ on the cochain complex $C_m(\mathbb{X})$ of multiplicative functions. More details will appear in \cite{HOW}.

The representation up to homotopy \eqref{eq:canonicalruthderivations} together with the splitting $h:\mathbb{L}(T\mathbb{X})\to \mathbb{L}(\At(\p))$ of \eqref{eq:atiyahextension2}, allows to apply Theorem \ref{thm:lecomte}, yielding the following result.

\begin{theorem}\label{thm:CWL2bundles}
Let $\p\curvearrowleft\GG\to \mathbb{X}$ be a principal 2-bundle equipped with a 2-connection. Let $h:\mathbb{L}(T\mathbb{X})\to \mathbb{L}(\At(\p))$ the induced splitting of the Atiyah extension of Lie 2-algebras \ref{eq:atiyahextension2}. For every $k\geq 0$ there exists a natural map

\begin{align}\label{eq:CWL2bundles}
cw:\underline{\Hom}^{\bullet}(\wedge^k\mathbb{L}(\Ad(\p))&[1],\mathbb{F(X)})^{\mathbb{L}(\mathrm{At}(\mathbb{P}))}\to H^{\bullet+k}(\mathbb{L}(T\mathbb{X}),\mathbb{F(X)})\\ \nonumber
&f\mapsto [f\circ K^k_h],
\end{align}

\noindent where $K^k_h=K_h\wedge_{\id}\cdots \wedge_{\id}K_h$ and $K_h\in C^1(\mathbb{L}(T\mathbb{X}),\mathbb{L}(\Ad(\p)))[1])$ is the curvature of $h$. The map $cw$ does not depend of the 2-connection.
\end{theorem}

One observes that since \ref{eq:atiyahextension2} is an extension of \emph{strict} Lie 2-algebras and the horizontal lift $h:T\mathbb{X}\to \mathrm{At}(\p)$ is a VB-groupoid morphism, then Remark \ref{rmk:CWL2algebras} implies that the curvature of the induced splitting $h:\mathbb{L}(T\mathbb{X})\to \mathbb{L}(\mathrm{At}(\p))$ has only two possibly non-zero components, namely:

\begin{enumerate}

\item  $K_h:\Gamma(A_{\mathbb{X}})\times \mathfrak{X}_{mult}(\mathbb{X})\to C^{\infty}(P_0,\mathfrak{h})^{G_0}$, with

$$K_h(a,\xi)= h_{c}([\xi,a^r]|_{X_0})-[h_m(\xi),(h_c(a))^r]|_{P_0}.$$

\item  $K_h:\wedge^2\mathfrak{X}_{mult}(\mathbb{X})\to \Hom_{gpd}(\p,\mathbb{g})^{\GG}$, where

$$K_h(\xi_1.\xi_2)= \Omega(\xi_1,\xi_2),$$

\end{enumerate}

\noindent where $\Omega=\bar{\Omega}+s^*\Omega_0$ is the curvature of the 2-connection associated to $h:T\mathbb{X}\to \At(\p)$.

Note that in this case, the component $K_h:\wedge^3\mathfrak{g}_0\to \mathfrak{n}_{-1}$ (see (2) in \ref{rmk:CWL2algebras}) is \emph{always zero} since all the brackets of order 3 are zero in a strict Lie 2-algebra.

\begin{remark}\label{rmk:CWLcurving}

Suppose that $B\in \Omega^2_{\tau}(P_0,\mathfrak{h})$ is a curving for a 2-connection $\theta\in \Omega^1(\p,\mathbb{\mathfrak{g}})$. Then, the second component of the curvature $K_h:\wedge^2\mathfrak{X}_{mult}(\mathbb{X})\to \Hom_{gpd}(\p,\mathbb{g})^{\GG}$ is given by

$$K_h=-(t^*B-s^*B)-\partial\circ B.$$

The curvature 3-form $\nabla^c_{\theta_0}B\in \Omega^3(P_0,\mathfrak{h})$ is not necessarily zero unless the horizontal lift $h:\mathbb{L}(T\mathbb{X})\to \At(\mathbb{P})$ is a weak $L_{\infty}$-algebra morphism (see Theorem \ref{thm:flatconnectivestructure}). In particular, since the component $K_h:\wedge^3\mathfrak{g}_0\to \mathfrak{n}_{-1}$ always vanishes, one cannot recover the curvature 3-form as a component of the curvature $K_h\in C^1(T\mathbb{X},\mathbb{F}(\mathbb{X}))$.

\end{remark}

Let us discuss some examples.

\begin{example}[\textbf{Ordinary principal bundles}]

Let $P\to X$ be an ordinary principal G-bundle viewed as a 2-bundle as in Example \ref{ex:2bundlebyusualbundle}. In this case, the map \eqref{eq:CWL2bundles} recovers the ordinary Chern-Weil map of a principal bundle with connection. See \cite{Lecomte}.

\end{example}

%\begin{example}[\textbf{\'Etale case}]

%\end{example}

\begin{example}[\textbf{$\mathbb{S}^1$-gerbes}]

Let $\mathcal{G}=(Y,P,\mu)$ be an $\mathbb{S}^1$-gerbe over a manifold $X$ equipped with a connective structure $(\theta,B)$. Let $H\in \Omega^3_{cl}(X)$ be the corresponding curvature 3-form. One would expect to recover the cohomology class $[H]\in H^3_{dR}(X)$ from the Chern-Weil-Lecomte map defined by the induced connection on the principal 2-bundle $\mathbb{P}\to Y\times_XY\rr Y$ underlying the gerbe $\mathcal{G}$.  However, as observed in Remark \ref{rmk:CWLcurving}, this is not possible because each of the Lie 2-algebras involved in the Atiyah sequence is strict. In spite of this, one can still describe the curvature 3-form via the Chern-Weil-Lecomte map associated to a suitable exact Courant algebroid. Indeed, as shown in \cite{Hitchin1, Hitchin2}, the choice of a connection and curving on an $\mathbb{S}^1$-gerbe over $X$ with curvature 3-form $H\in\Omega^3_{cl}(X)$ defines an exact Courant algebroid $E\to X$ with a splitting whose \v{S}evera class is $[H]\in H^3_{dR}(X)$. The class $[H]$ can be recovered as a characteristic class via Chern-Weil-Lecomte as in Example \ref{ex:CWLsevera}. Alternatively, as shown in \cite{DjounvounaKrepski, KrepskiVaughan}, every $\mathbb{S}^1$-gerbe over $X$ equipped with a connective structure and 3-curvature $H\in \Omega^3_{cl}(X)$ has an associated Lie 2-algebra of infinitesimal symmetries which is quasi-isomorphic to the Poisson-Lie 2-algebra of $(X,H)$. Hence the class $[H]$ can also be described as in Example \ref{ex:CWL2plectic}.

%the curvature 3-form can be described as the characteristic class of a Lie 2-algebra naturally associated with the pair $(\mathcal{G},B)$, namely its Lie 2-algebra of infinitesimal symmetries. As shown in \cite{KrepskiVaughan}, there exists a Lie 2-algebra $\mathbb{L}(\mathcal{G},\theta,B)$ whose objects are multiplicative vector fields $(\xi,v)$ on $\mathbb{P}\rr Y$ which preserve the connective structure $(\theta,B)$ in the sense that there exists $\alpha\in \Omega^1(Y)$ with

%$$\Lie_{\xi}\theta=t^*\alpha-s^*\alpha \quad \text{ and } \quad \Lie_{v}B=d\alpha,$$

%\noindent where $s,t$ are the source and target maps of the Lie groupoid $\mathbb{P}\rr Y$. It turns out that $\mathbb{L}(\mathcal{G},\theta,B)$ is a Lie 2-subalgebra of the Lie 2-algebra of multiplicative vector fields $\mathbb{L}(T\mathbb{P})$, see Proposition 3.16 of \cite{KrepskiVaughan}. On the one hand, since the curvature 3-form $H\in \Omega^3_{cl}(X)$ is closed, it defines an exact Courant algebroid structure on $TX\oplus T^*X$ whose \v{S}evera class is $[H]\in H^3_{dR}(X)$. In particular, as shown by Roytenberg and Weinstein \cite{RoytenbergWeinstein}, there is a semi-strict Lie 2-algebra with underlying complex $\frac{1}{2}d:C^{\infty}(X)\to \mathfrak{X}(X)\oplus \Omega^1(X)$. It follows from Example \ref{ex:CWLsevera} that this semi strict Lie 2-algebra has characteristic class 

\end{example}

%%%%%%%%%%%%%%%%%%%%%%%%%%%%%%%%%%%%%%%%%%%%%%%%%
\subsubsection{Morita invariance}

We finish this section by describing how does the Chern-Weil-Lecomte map behave under pullbacks by Morita maps. For that, we briefly recall some facts about projectable sections of VB-groupoids, see \cite{OrtizWaldron} for more details. Let $\Phi:E\to E'$ be a vector bundle morphism covering $\phi:M\to M'$. Consider the set $\Gamma(E)^{\Phi}$ of \textbf{$\Phi$-projectable sections}, that is, sections $e\in\Gamma(E)$ for which there exists $e'\in\Gamma(E')$ with $\Phi\circ e=e'\circ \phi$. The section $e'$ is not necessarily unique, but uniqueness holds whenever $\phi:M\to M'$ is surjective. In this case, there are maps

\begin{equation}\label{eq:projectablesections}
    \begin{tikzcd}
   \Gamma(\mathbb{E}) &\arrow[l,hook',"\iota"']\Gamma(\mathbb{E})^{\Phi}\arrow[r,"\Phi_*"]&\Gamma(\mathbb{E}'),
    \end{tikzcd}
\end{equation}

\noindent where $\Phi_{*}:\Gamma(E)^{\Phi}\to \Gamma(E'); e\mapsto e'$ and $\Gamma(E)^{\Phi}\hookrightarrow \Gamma(E)$ is the canonical inclusion.

Let $\Phi:\mathbb{E}\to \mathbb{E}'$ be a VB-groupoid morphism covering $\phi:\mathbb{X}\to \mathbb{X}'$ and inducing $\Phi_c:C\to C'$ between the core bundles.  The complex of multiplicative projectable sections $C^{\bullet}_{mult}(\mathbb{E})^{\Phi}$ is the subcomplex of $C^{\bullet}_{mult}(\mathbb{E})$ (\eqref{eq:complexmultiplicative}) given by

$$\delta:\Gamma(C)^{\Phi_c}\to \Gamma_{mult}(\mathbb{E})^{\Phi}; c\mapsto c^r-c^l.$$

\noindent If $\Phi:\mathbb{E}\to \mathbb{E}'$ is surjective on objects, then diagram \eqref{eq:projectablesections} can be upgraded to the level of complexes of multiplicative sections, yielding

\begin{equation}\label{eq:multiplicativeprojectablesections}
    \begin{tikzcd}
   C^{\bullet}_{mult}(\mathbb{E}) &\arrow[l,hook',"\iota"']C^{\bullet}_{mult}(\mathbb{E})^{\Phi}\arrow[r,"\Phi_*"]&C^{\bullet}_{mult}(\mathbb{E}').
    \end{tikzcd}
\end{equation}

\noindent It was proven in \cite{OrtizWaldron} that if $\Phi:\mathbb{E}\to \mathbb{E}'$ is a VB-Morita map which is also surjective on objects, then the maps in \eqref{eq:multiplicativeprojectablesections} are quasi-isomorphisms. Additionally, if $\mathbb{E},\mathbb{E}'$ are LA-groupoids and $\Phi:\mathbb{E}\to \mathbb{E}'$ is an LA-groupoid morphism covering $\phi:\mathbb{X}\to \mathbb{X}'$, then $\Gamma(C)^{\Phi_c}\to \Gamma_{mult}(\mathbb{E})^{\Phi}$ is actually a crossed sub-module of the crossed module of multiplicative sections. In particular, there is a Lie 2-subalgebra $\mathbb{L}(\mathbb{E})^{\Phi}\subseteq \mathbb{L}(\mathbb{E})$ of $\Phi$-projectable sections. Moreover, if $\Phi:\mathbb{E}\to \mathbb{E}'$ is an LA-Morita map which is surjective on objects, then the maps in \eqref{eq:multiplicativeprojectablesections} are quasi-isomorphisms of crossed modules, see Theorem 7.3. in \cite{OrtizWaldron}. In particular, one has quasi-isomorphisms of Lie 2-algebras

\begin{equation}\label{eq:quasiLie2algebras}
    \begin{tikzcd}
   \mathbb{L}(\mathbb{E}) &\arrow[l,hook',"\iota"']\mathbb{L}(\mathbb{E})^{\Phi}\arrow[r,"\phi_*"]&\mathbb{L}(\mathbb{E}').
    \end{tikzcd}
\end{equation}

\noindent As a consequence, an LA-Morita equivalence $\begin{tikzcd}
   \mathbb{E} &\arrow[l,"\Phi"']\tilde{\mathbb{E}}\arrow[r,"\Psi"]&\mathbb{E}'
    \end{tikzcd}$ induces a zig-zag of Lie 2-algebra quasi-isomorphisms
    
\begin{equation}
\begin{tikzcd}
    \mathbb{L}(\mathbb{E})&\mathbb{L}(\tilde{\mathbb{E}})^{\Phi}\arrow[l,"\Phi_*"']\arrow[r,hook]&\mathbb{L}(\tilde{\mathbb{E}})&\mathbb{L}(\tilde{\mathbb{E}})^{\Psi}\arrow[l,hook']\arrow[r,"\Psi_*"]&\mathbb{L}(\mathbb{E}').
    \end{tikzcd}
\end{equation}

In what follows, we specialize the previous discussion in the case of the VB-groupoid map $T\phi:T\mathbb{X}\to T\mathbb{X}'$ associated with a Lie groupoid morphism  $\phi:\mathbb{X}\to \mathbb{X}'$. Our main goal is to show that the cohomology $H(\mathbb{L}(T\mathbb{X}),\mathbb{F}(\mathbb{X}))$ associated with the ruth $\rho_{\mathbb{X}}:\mathbb{L}(\mathbb{X})\otimes \mathbb{F}(\mathbb{X})\to \mathbb{F}(\mathbb{X})$ as in \eqref{eq:canonicalruthderivations}, is invariant by Morita equivalence. This follows from the invariance of the $L_{\infty}$-cohomology in the derived category of ruths, see Corollary \ref{cor:moritainvarianceLinftycohomology}.

Indeed, let $\phi:\mathbb{X}\to \mathbb{X}'$ be a Morita map which is surjective on objects. Then, the LA-Morita map $T\phi:T\mathbb{X}\to T\mathbb{X}'$ is surjective on objects, hence there are quasi-isomorphisms of Lie 2-algebras

\begin{equation*}
    \begin{tikzcd}
   \mathbb{L}(T\mathbb{X}) &\arrow[l,hook',"\iota"']\mathbb{L}(T\mathbb{X})^{T\phi}\arrow[r,"(T\phi)_*"]&\mathbb{L}(T\mathbb{X}').
    \end{tikzcd}
\end{equation*}

\noindent

The next result is proven in \cite{HOW} using the language of differential graded modules.

\begin{theorem}[\cite{HOW}]
Let $\mathbb{X}, \mathbb{X}'$ be Morita equivalent Lie groupoids. Then the representations up to homotopy $(\mathbb{F(X)},\rho_{\mathbb{X}})$ and $(\mathbb{F(X')},\rho_{\mathbb{X}'})$ are quasi-isomorphic.
\end{theorem}

We actually need the following quasi-isomorphism. First, it is well-known that the pullback of functions defines a quasi-isomorphism of complexes $\phi^*:\mathbb{F}(\mathbb{X}')\to \mathbb{F}(\mathbb{X})$. Also, one can see that the map $((T\phi)_{*},\phi^*):(\mathbb{F}(\mathbb{X}),\iota^*\rho_{\mathbb{X}})\to (\mathbb{F}(\mathbb{X}'),\rho_{\mathbb{X}'})$ is a quasi-isomorphism of ruths, more details will appear in \cite{HOW}. As a consequence of Corollary \ref{cor:moritainvarianceLinftycohomology} one has

\begin{equation}\label{eq:Moritainvariancecohomologyfunctions}
H(\mathbb{L}(\mathbb{X}),\mathbb{F}(\mathbb{X}))\cong H(\mathbb{L}(\mathbb{X}'),\mathbb{F}(\mathbb{X}')).
\end{equation}
The following result establishes the Morita invariance of the $L_{\infty}$-cohomology of the Lie 2-algebra of multiplicative vector fields on a Lie groupoid.

\begin{theorem}\label{thm:Lie2cohomologyinvariance}

Let $\mathbb{X},\mathbb{X}'$ be Morita equivalent Lie groupoids. Then, there is an isomorphism

$$H(\mathbb{L}(T\mathbb{X}),\mathbb{F}(\mathbb{X}))\cong H(\mathbb{L}(T\mathbb{X}'),\mathbb{F}(\mathbb{X}')).$$

\end{theorem}

\begin{proof}

Let us consider Morita maps $\begin{tikzcd}\mathbb{X} &\arrow[l,"\Phi"']\tilde{\mathbb{Y}}\arrow[r,"\Psi"]&\mathbb{X}'\end{tikzcd}$ which are surjective on objects. This induces an LA-Morita equivalence $\begin{tikzcd}T\mathbb{X} &\arrow[l,"T\Phi"']T\tilde{\mathbb{Y}}\arrow[r,"T\Psi"]&T\mathbb{X}'\end{tikzcd}$ which is surjective on objects as well. The argument used to conclude \eqref{eq:Moritainvariancecohomologyfunctions} can be applied to each map of this Morita equivalence, giving rise to isomorphisms:

$$H(\mathbb{L}(T\mathbb{X}),\mathbb{F}(\mathbb{X}))\cong H(\mathbb{L}(T\tilde{\mathbb{Y}}),\mathbb{F}(\tilde{\mathbb{Y}}))\cong H(\mathbb{L}(T\mathbb{X}'),\mathbb{F}(\mathbb{X}')),$$

\noindent which finishes the proof.

\end{proof}

The $L_{\infty}$-cohomology $H(\mathbb{L}(T\mathbb{X}),\mathbb{F}(\mathbb{X}))$ is the target space of the Chern-Weil-Lecomte \eqref{eq:CWL2bundles} associated with a principal 2-bundle over $\mathbb{X}$. Due to Theorem \ref{thm:Lie2cohomologyinvariance}, one is naturally led to study the behavior of the Chern-Weil-Lecomte map under pullbacks by Morita maps, which is explained below. 

Let $\phi:\mathbb{X}\to \mathbb{X}'$ be a Lie groupoid morphism which is surjective on objects and let $\p\curvearrowleft \GG \to \mathbb{X}'$ be a principal 2-bundle. The pullback 2-bundle $(\phi^*\p)\curvearrowleft \GG \to \mathbb{X}$ is given by $\phi^*\p=(\phi^*_1P_1\rr \phi^*_0P_0)$. The natural map $\Phi:\phi^*\p\to \p$ is both a Lie groupoid morphism and a bundle map covering $\phi:\mathbb{X}\to \mathbb{X}'$. In other words, the map $\Phi:\phi^*\p\to \p$ is a principal 2-bundle morphism over $\phi:\mathbb{X}\to \mathbb{X}'$. Note that if $\phi:\mathbb{X}\to \mathbb{X}'$ is a Morita map, so is $\Phi:\phi^*\p\to \p$.

We proceed now to relate the Atiyah extensions $\mathbb{At}(\p)$ and $\mathbb{At}(\phi^{*}\p)$ of $\p$ and $\phi^{*}\p$, respectively. The Atiyah LA-groupoid of $\phi^*\p \to \mathbb{X}$ is given by the \emph{pullback Lie algebroid} $\phi^{!}\At(\p)\to \mathbb{X}$, which turns out to be an LA-groupoid since $\phi:\mathbb{X}\to \mathbb{X}'$ is a Lie groupoid morphism. Similarly, the adjoint LA-groupoid of $\phi^*\p \to \mathbb{X}$ is $\Ad(\phi^*\p)=\phi^*\p\to \mathbb{X}$ the usual pullback of VB-groupoids. The principal 2-bundle map $\Phi:\phi^{!}\p\to \p$ induces natural LA-groupoid morphisms

$$\At(\Phi):\phi^{!}\At(\p)\to \At(\p); \quad \Ad(\Phi):\phi^*\Ad(\p)\to \Ad(\p),$$

\noindent which are surjective on objects. Hence, it follows from \eqref{eq:quasiLie2algebras} that there are morphisms of Lie 2-algebras

\begin{equation*}
    \begin{tikzcd}
   \mathbb{L}(\phi^{*}\Ad(\p)) &\arrow[l,hook',""']\mathbb{L}(\phi^{*}\Ad(\p))^{\Ad(\Phi)}\arrow[r,"\Ad(\Phi)_*"]&\mathbb{L}(\Ad(\p)),
    \end{tikzcd}
\end{equation*}

\begin{equation*}
    \begin{tikzcd}
   \mathbb{L}(\phi^{!}\At(\p)) &\arrow[l,hook',""']\mathbb{L}(\phi^{!}\At(\p))^{\At(\Phi)}\arrow[r,"\At(\Phi)_*"]&\mathbb{L}(\At(\p)),
    \end{tikzcd}
\end{equation*}

\noindent and

\begin{equation*}
    \begin{tikzcd}
   \mathbb{L}(T\mathbb{X}) &\arrow[l,hook',""']\mathbb{L}(T\mathbb{X})^{T\phi}\arrow[r,"(T\phi)_*"]&\mathbb{L}(T\mathbb{X}'),
    \end{tikzcd}
\end{equation*}

\noindent which determine morphisms of Lie 2-algebra extensions 

\begin{equation}\label{eq:Moritaequivalenceextensions}
    \begin{tikzcd}
   \mathbb{At}(\phi^*\p) &\arrow[l,hook',""']\mathbb{At}(\phi^{*}\p)^{\mathbb{At}(\Phi)}\arrow[r,"\mathbb{At}(\Phi)_*"]&\mathbb{At}(\p),
    \end{tikzcd}
\end{equation}

\noindent where $\mathbb{At}(\phi^{*}\p)^{\mathbb{At}(\Phi)}$ is the extension of Lie 2-algebras

$$\mathbb{L}(\phi^{*}\Ad(\p))^{\Ad(\Phi)} \to \mathbb{L}(\phi^{!}\At(\p))^{\At(\Phi)} \to \mathbb{L}(T\mathbb{X})^{T\phi}$$

\noindent defined by projectable sections.

\begin{remark}
One can see that  $\mathbb{At}(\phi^{*}\p)^{\mathbb{At}(\Phi)}$ is the pullback extension of $\mathbb{At}(\p)$ by the Lie 2-algebra morphism $(T\phi)_{*}:\mathbb{L}(T\mathbb{X})^{T\phi}\to \mathbb{L}(T\mathbb{X}')$. In particular, one has $\mathbb{L}(\phi^{*}\Ad(\p))^{\Ad(\Phi)}\cong \mathbb{L}(\Ad(\p))$.
\end{remark}

Note that every multiplicative horizontal lift $h:T\mathbb{X}'\to \At(\p)$ induces a multiplicative horizontal lift $h^{!}:T\mathbb{X}\to \phi^!(\At(\p)); v\mapsto (v,(h\circ T\phi)(v))$. This is the horizontal lift corresponding to the pullback 2-connection $\Phi^*\theta$ where $\theta\in \Omega^1(\p,\mathbb{g})$ is the 2-connection associated with $h:T\mathbb{X}'\to \At(\p)$.

\begin{remark}\label{rmk:pullbackatiyahextension}

The horizontal lift $h:T\mathbb{X}'\to \At(\p)$, induces a horizontal lift of the pullback extension $\mathbb{At}(\phi^{*}\p)^{\mathbb{At}(\Phi)}$, which coincides with the restriction of $h^{!}:T\mathbb{X}\to \phi^!(\At(\p))$ to projectable sections. The latter is well defined once $h^{!}$ preserves projectable sections. Hence one can see \eqref{eq:Moritaequivalenceextensions} as morphisms of extensions of Lie 2-algebras \emph{with horizontal lifts}.
\end{remark}

The inclusion in \eqref{eq:Moritaequivalenceextensions} gives rise to the following commutative diagram relating the associated Chern-Weil-Lecomte maps

\begin{equation}
    \begin{tikzcd}
    \underline{\Hom}^{\bullet}(\wedge^k\mathbb{L}(\phi^*\Ad(\p))\left[1\right],\mathbb{F(X)})^{\mathbb{L}(\phi^{!}\At(\p))}\arrow[r,"cw"]\arrow[d,""]&H^{k+\bullet}(\mathbb{L}(T\mathbb{X}),\mathbb{F(X)})\arrow[d,""]\\
       \underline{\Hom}^{\bullet}(\wedge^k\mathbb{L}(\Ad(\p))\left[1\right],\mathbb{F(X)})^{\mathbb{L}(\phi^{!}\At(\p))^{\At(\Phi)}}\arrow[r,"cw"] &H^{k+\bullet}(\mathbb{L}(T\mathbb{X})^{T\phi},\mathbb{F(X)}).
    \end{tikzcd}
\end{equation}

\noindent Additionally, it follows from Theorem \ref{thm:naturality} that

\begin{equation}
    \begin{tikzcd}
    \underline{\Hom}^{\bullet}(\wedge^k\mathbb{L}(\Ad(\p))\left[1\right],\mathbb{F(X')})^{\mathbb{L}(\At(\p))}\arrow[r,"cw"]\arrow[d,""]&H^{k+\bullet}(\mathbb{L}(T\mathbb{X}'),\mathbb{F(X')})\arrow[d,""]\\
       \underline{\Hom}^{\bullet}(\wedge^k\mathbb{L}(\Ad(\p))\left[1\right],\mathbb{F(X)})^{\mathbb{L}(\phi^{!}\At(\p))^{\At(\Phi)}}\arrow[r,"cw"] &H^{k+\bullet}(\mathbb{L}(T\mathbb{X})^{T\phi},\mathbb{F(X)}).
    \end{tikzcd}
\end{equation}

As a result, we have the following.

\begin{theorem}
    Let $(\mathbb{P}\curvearrowleft \mathbb{G}\to \mathbb{X}')$ be a principal 2-bundle equipped with a 2-connection. If $\phi:\mathbb{X}\to \mathbb{X}'$ is a  Lie groupoid morphism, then the following the diagram

    \begin{equation}\label{eq:MoritaCWL2bundles}
    \begin{tikzcd}
        \Hom^{\bullet}(\wedge^k \mathbb{L}(\phi^{*}\mathrm{Ad}(\mathbb{P}))[1],\mathbb{F}(\mathbb{X}))^{\mathbb{L}(\phi^{!}\mathrm{At}(\mathbb{P}))}\arrow[r,"cw"]\arrow[d]&H^{\bullet+ k}(\mathbb{L}(T\mathbb{X}),\mathbb{F}(\mathbb{X}))\arrow[d]\\
        \Hom^{\bullet}(\wedge^k \mathbb{L}(\mathrm{Ad}(\mathbb{P}))[1],\mathbb{F}(\mathbb{X}))^{\mathbb{L}(\phi^{!}\At(\p))^{\At(\Phi)}}\arrow[r,"cw"]&H^{\bullet+ k}(\mathbb{L}(T\mathbb{X})^{\phi},\mathbb{F}(\mathbb{X}))\\
        \Hom^{\bullet}(\wedge^k \mathbb{L}(\mathrm{Ad}(\mathbb{P}))[1],\mathbb{F}(\mathbb{X}'))^{\mathbb{L}(\mathrm{At}(\mathbb{P}))}\arrow[r,"cw"]\arrow[u]&H^{\bullet+ k}(\mathbb{L}(T\mathbb{X}'),\mathbb{F}(\mathbb{X}')),\arrow[u]\\
    \end{tikzcd}
\end{equation}

\noindent commutes. Additionally, if $\phi:\mathbb{X}\to \mathbb{X}'$ is a Morita map which is surjective on objects, then the vertical arrows on the right hand side of the diagram are isomorphisms.

\end{theorem}

Since $\mathbb{L}(\phi^{*}\Ad(\p))^{\Ad(\Phi)}\cong \mathbb{L}(\Ad(\p))$, see Remark \ref{rmk:pullbackatiyahextension}, the commutative diagram on top of \eqref{eq:MoritaCWL2bundles} says that given a ruth morphism $f:\wedge^k\mathbb{L}(\phi^{*}\Ad(\mathbb{P}))[1]\to \mathbb{F}(\mathbb{X})$, then:

$$\iota^{*}_{\mathbb{L}(T\mathbb{X})^{\phi}}cw(f)=cw(f|_{\mathbb{L}(\phi^{*}\Ad(\p))^{\Ad(\Phi)}}),$$

\noindent where $\iota_{\mathbb{L}(T\mathbb{X})^{\phi}}:\mathbb{L}(T\mathbb{X})^{\phi}\to \mathbb{L}(T\mathbb{X})$ is the inclusion. In other words, the Chern-Weil-Lecomte map is compatible with projectable complexes. The restriction to projectable complex does not change the space coefficients $\mathbb{F(X)}$ where the cohomology class $cw(f)$ takes values. However, the second commutative diagram in \eqref{eq:MoritaCWL2bundles} comes from the compatibility between the Chern-Weil-Lecomte map and pullbacks of extensions \ref{thm:naturality}, hence the space of coefficients $\mathbb{F(X)}$ is replaced by $\mathbb{F(X')}$. If $g:\wedge^k\mathbb{L}(\Ad(\p))\to \mathbb{F(X')}$, then

$$\big{(}(T\phi)_*\big{)}^{*}cw(g)=cw(g\circ \phi^{*}).$$

%%%%%%%%%%%%%%%%%%%
%%%%%%%%%%%%%%%%%%%
%%%%%%%%%%%%%%%%%%%

\end{document}